\documentclass{article}
\usepackage[utf8]{inputenc}
\usepackage[T1]{fontenc} 
\usepackage{amsmath}
\usepackage{amsfonts}
\usepackage{amssymb}
\usepackage{amsthm}
\usepackage{mathabx}
\usepackage{color}
\usepackage{esint}
\usepackage{enumerate}
\usepackage{dsfont}
\usepackage{psfrag}
\usepackage{stmaryrd}
\usepackage{hyperref}
\usepackage{graphicx}
\usepackage{auto-pst-pdf}
\usepackage{graphics}
\usepackage{epstopdf}
\usepackage{epsfig}
\usepackage{setspace}
\usepackage{tikz}

\usepackage{pgfplots}
\usepackage{pgfplotstable}
\newtheorem{theorem}{Theorem}
\newtheorem{lemma}[theorem]{Lemma}
\newtheorem{remark}[theorem]{Remark}
\newtheorem{proposition}[theorem]{Proposition}

\newcommand{\R}{\mathbb{R}}
\newcommand{\Z}{\mathbb{Z}}
\newcommand{\N}{\mathbb{N}}
\newcommand{\eps}{\varepsilon}
\newcommand{\dis}{\displaystyle}
\newcommand{\Th}{\mathcal{T}_H}

\newcommand{\MsFEMA}{Adv-MsFEM}
\newcommand{\MsFEMB}{MsFEM}
\newcommand{\MsFEMBSUPG}{Stab-MsFEM}
\newcommand{\VA}{V_H^{\eps,Adv}}
\newcommand{\VB}{V_H^{\eps}}

\newcommand{\VBn}{V_{H_n}^{\eps}}
\newcommand{\VBh}{V^{\eps}_{H,h}}

\newcommand*{\medcap}{\mathbin{\scalebox{1.15}{\ensuremath{\cap}}}}

\begin{document}
\date{\today}
\author{Claude Le Bris, Frédéric Legoll and François Madiot
\\
{\footnotesize \'Ecole des Ponts \& INRIA}\\
{\footnotesize 6 et 8 avenue Blaise Pascal, 77455 Marne-La-Vall\'ee Cedex 2, France}\\
{\footnotesize \tt \{lebris,madiotf\}@cermics.enpc.fr}\\
{\footnotesize \tt legoll@lami.enpc.fr}
}
\title{A numerical comparison of some \\Multiscale Finite Element approaches for convection-dominated problems in heterogeneous media}
\maketitle

%\doublespacing

\begin{abstract}
The purpose of this work is to investigate the behavior of Multiscale Finite Element type methods for advection-diffusion problems in the advection-dominated regime. We present, study and compare various options to address the issue of the simultaneous presence of both heterogeneity of scales and strong advection. Classical MsFEM methods are compared with adjusted MsFEM methods, stabilized versions of the methods, and a splitting method that treats the multiscale diffusion and the strong advection separately. 
\end{abstract}

\section{Introduction}

We consider in this work an advection-diffusion equation that has both a multiscale character (encoded in an highly oscillatory diffusion coefficient) and a dominating advection. Formally, the equation reads as
\begin{equation}
\label{eq:adv-diff-formal}
-\text{div } (A^\eps\nabla u^\eps) + b\cdot\nabla u^\eps = f \quad \text{in }\Omega,
\qquad
u^\eps=0 \quad\text{on }\partial\Omega.
\end{equation}
Our self-explanatory notation will be made precise in the sequel, along with the mathematical setting that allows to rigorously consider this equation. Our purpose is to investigate whether numerical methods dedicated to the treatment of multiscale phenomena, such as Multiscale Finite Element Methods (henceforth abbreviated as MsFEM) and methods specifically designed to address the dominating advection, such as Streamline-Upwind/Petrov-Galerkin (SUPG) type methods, can separately adequately address the twofold problem, or, if need be, to discover how these methods may be combined to form the best possible approach in various regimes.
 
Equation~\eqref{eq:adv-diff-formal} is practically relevant and interesting {\it per se}. Our study of this particular equation is nevertheless rather to be seen as a step toward the study of the following much more relevant case, which will be performed in an upcoming work~\cite{these-madiot}: an (single-scale) advection-diffusion equation, with a dominating advection term, posed on a {\em perforated} domain (in that vein, see~\cite{degond2014crouzeix}). In a previous, somewhat related couple of studies~\cite{lebris2013msfem,lebris2014msfem}, we have used with much benefit the highly oscillatory case as a test-bed for designing and studying approaches subsequently used for the more challenging perforated case. 

\medskip

Methods of the MsFEM type have proved efficient in a number of contexts. In essence, they are based upon choosing, as specific finite dimensional basis to expand the numerical solution upon, a set of functions that themselves are solutions to a highly oscillatory {\em local} problem, at scale~$\varepsilon$, involving the differential operator present in the original equation. This problem-dependent basis set is likely to better encode the fine-scale oscillations of the solution and therefore allow to capture the solution more accurately. Numerical observation along with mathematical arguments prove that this is indeed generically the case. For the specific advection-diffusion equation~\eqref{eq:adv-diff-formal} we consider here, two natural options for the construction of the basis set are (i) to pick as basis functions solutions to the (multiscale) diffusion operator only, or (ii) to also involve in the definition of the functions the advection operator. These two approaches will be among the set of approaches considered and tested below. In the former option, when the basis functions do not involve the advection operator, one may fear that, in the presence of advection, and especially in the presence of a strong advection that dominates the diffusion -- a regime we focus on throughout this work --, the accuracy of the classical MsFEM dramatically deteriorates. This is for instance the case, "when $\varepsilon=1$", for classical $\mathbb{P}^1$ finite element methods. Stabilization procedures are then in order and we will indeed adapt such a procedure to the present multiscale context. On the other hand, in the latter option, it is unclear whether the presence of the advection term {\em also} for the definition of the basis functions allows, or not, for the method to also perform well in the advection-dominated regime. This will be investigated below. However well such an approach performs, the fact that the advection is involved in the definition of the finite elements might create issues, and be prohibitively expensive computationally, when the advection varies and the equation needs to be solved repeatedly, either because the present steady state setting of~\eqref{eq:adv-diff-formal} is in fact a time iteration within the numerical simulation of a time-dependent equation, or because equation~\eqref{eq:adv-diff-formal} is part of an optimization, or inverse problem. Also, inserting the advection term in the definition of the basis functions is a \emph{very} invasive implementation, which might be problematic in some contexts. Both observations are sufficient motivations to also consider a splitting method, separately addressing the multiscale character with a classical MsFEM approach for the solution of the diffusion operator, and solving a single-scale advection-dominated advection-diffusion equation with a stabilized $\mathbb{P}^1$ method.

The four MsFEM-type approaches we have just mentioned (classical -- that is, with basis functions constructed from diffusion only --, classical and stabilized, advection-diffusion based, splitting the advection and the multiscale character) will be studied and compared. For reference, we will also use a $\mathbb{P}^1$ finite element method, stabilized or not, in particular to investigate when the multiscale nature of the problem and the domination of the convection matter, or not.

\medskip

In the context of HMM-type methods, multiscale advection-diffusion problems with dominating convection have been considered e.g. in~\cite{abdulle2012discontinuous}. 

\medskip

Our article is organized as follows. Section~\ref{sec:description} briefly recalls, essentially for the sake of self-consistency, some basic, classical and well-known facts on the building blocks (stabilization, multiscale approaches) we use, and describes in more details the numerical approaches we consider. We next provide, in Section~\ref{section_preliminary}, a complete numerical analysis of the approaches {\em in the one-dimensional setting}. We are unfortunately unable to conduct the same analysis in higher dimensions, but some of the issues we raise and discuss in the one-dimensional context are definitely useful to understand the approaches in a more general context. In particular, we point out that the direct application of an SUPG stabilization on MsFEM leads to an approach that is {\em not} strongly consistent (in sharp contrast to its single-scale, say $\mathbb{P}^1$ version), because the basis functions are not known analytically but only up to the numerical error present in the offline precomputation. We provide a solution to that difficulty. We show that, in spite of a lack of consistency, the method we design can be certified (and numerical observation will later show it performs efficiently). We also devote some time to the detailed study, in \emph{any} dimension, of the convergence of the splitting approach.

Our final Section~\ref{section_numerical_simulations} presents a comprehensive series of numerical tests and comparisons. An executive summary of our {\em main} conclusions is as follows: 
\begin{itemize}
\item (i) the best possible approach among all those we consider is the stabilized version of MsFEM, unless one does not want to be intrusive in which case the splitting approach performs approximately equally well, for an online computational cost that might be significantly larger, especially for problems of large size for which iterative solvers have to employed;
\item (ii) the method using basis functions built upon the full advection-diffusion operator is not sufficiently stable to perform well in the advection-dominated regime;
\item (iii) when advection outrageously dominates diffusion, the multiscale character of the solution (at least in the bulk of the domain) is essentially overshadowed by the convection, and a ``classical'' stabilized $\mathbb{P}^1$ finite element method performs as well as a MsFEM-type approach, a somewhat intuitive fact that our study allows to confirm.
\end{itemize}
Further details on the approaches considered are given in the body of the text. 

\section{Description of the numerical approaches}
\label{sec:description}

We describe in Section~\ref{section_building_blocks} the standard numerical tools we use throughout this work. We next present in Section~\ref{section_our_four_appr} the four numerical methods we study. 

\subsection{Building blocks}
\label{section_building_blocks}

In this section, we briefly recall for convenience some classical elements on the two building blocks we make use of to construct the approaches we study, namely stabilization methods (more specifically, SUPG type methods) and Multiscale Finite Element Methods (MsFEM). The reader already familiar with these notions may easily skip the present section and directly proceed to Section~\ref{section_our_four_appr}.

\subsubsection{Stabilized methods}
\label{section_stabilized_methods}

We temporarily consider the \emph{single-scale} advection-diffusion problem
\begin{align}
-\alpha \Delta u + b\cdot\nabla u = f \quad\text{in }\Omega,
\qquad
u=0 \quad\text{on }\partial\Omega, 
\label{single_scale_problem}
\end{align}
where $\Omega$ is a smooth bounded domain of $\R^d$, $\alpha>0$, $b \in (L^\infty(\Omega))^d$ and $f \in L^2(\Omega)$. We suppose that 
\begin{align}
\text{div } b=0 \quad\text{in }\Omega,
\label{div_b_zero}
\end{align}
so that problem~\eqref{single_scale_problem} is coercive and amenable to standard numerical analysis techniques for coercive problems. We shall discuss the case of non-coercive problems in Remark~\ref{remark_non_coercive} below.

Let $\mathcal{T}_H$ be a uniform regular mesh of size $H$ discretizing $\Omega$, and let $V_H$ be the classical $\mathbb{P}^1$ Finite Element space associated to this mesh. The classical Galerkin approximation of~\eqref{single_scale_problem} reads as the following variational formulation:
$$
%\begin{align}
\text{Find }u_H \in V_H \text{ such that, for any $v_H \in V_H$,}
\quad a(u_H,v_H)=F(v_H), 
%\label{varf_pb_singlescale}
%\end{align}
$$
where
\begin{equation}
a(u,v)=\int_\Omega \alpha\nabla u\cdot\nabla v + (b\cdot\nabla u) v,
\qquad
F(v)=\int_\Omega fv. 
\label{def_a_f}
\end{equation}
Since the solution $u$ to~\eqref{single_scale_problem} is in~$H^2(\Omega)$, we have the following error estimate as a direct consequence of C\'ea's lemma:
\begin{align}
\|u-u_H\|_{H^1(\Omega)} \leq CH(1+\text{Pe} \, H) \,\|u\|_{H^2(\Omega)},
\label{apriori_galerkin_single_scale}
\end{align}
where $C$ is independent of $H$ and $b$. 
%%%%%%%%
%, $|u-u_H|_{H^1(\Omega)} = \|\nabla (u-u_H)\|_{L^2(\Omega)}$, $\displaystyle |u|_{H^2(\Omega)}^2 = \sum_{i,j=1}^d \|\partial_{ij}u\|_{L^2(\Omega)}^2$
%%%%%%%%%
We have introduced, as is classical, the global P\'eclet number
\begin{equation}
\label{global-peclet}
\text{Pe} = \frac{\|b\|_{L^\infty(\Omega)}}{2\alpha}
\end{equation}
of problem~\eqref{single_scale_problem}. We thus see that the larger the product Pe$\, H$, the larger the potential numerical error. Intuitively, the problem becomes less and less coercive as advection increasingly dominates over diffusion and, eventually, the coercivity is lost~\cite[Section 3.5.2]{ern2004theory} when Pe goes to~$+\infty$. As is well-known, the P\'eclet number directly affects the quality of the numerical results. With the standard $\mathbb{P}^1$ finite element approximation, oscillations polluting the solution are observed (see Figure~\ref{fig_stab_P1} below). 

Stabilization is a classical subject of numerical analysis. Many works (see e.g.~\cite{hughes1995multiscale} and the textbooks~\cite{quarteroni1994numerical,roos1991numerical}) have been devoted to designing stabilized methods for the convection-dominated regime. They consist in considering the following problem:
\begin{align}
&\text{Find }u_H^s\in V_H \text{ such that, for any $v_H\in V_H$,}\nonumber\\
&a(u_H^s,v_H)+a_{\text{stab}}(u_H^s,v_H)= F(v_H)+F_{\text{stab}}(v_H),
\label{varf_supg_singlescale}
\end{align}
where $a$ and $F$ are defined by~\eqref{def_a_f} and $a_{\text{stab}}$ and $F_{\text{stab}}$ are defined by
\begin{align}
&a_{\text{stab}}(u_H^s,v_H)=\sum_{\mathbf{K}\in\mathcal{T}_H} \bigg(\tau_{\mathbf{K}}\mathcal{L}u_H^s,(\mathcal{L}_{ss}+\rho \mathcal{L}_{s})v_H\bigg)_\mathbf{K}, 
\label{single_scale_stab_term} 
\\
& F_{\text{stab}}(v_H)=\sum_{\mathbf{K}\in\mathcal{T}_H} \bigg(\tau_{\mathbf{K}} f,(\mathcal{L}_{ss}+\rho \mathcal{L}_{s})v_H\bigg)_\mathbf{K}, 
\label{single_scale_F_term} 
\end{align}
where, for any $u$ and $v$, $\displaystyle (u,v)_\mathbf{K}=\int_\mathbf{K} u \, v$, $\mathcal{L}_{s}u= -\alpha\Delta u$ and $\mathcal{L}_{ss}u =b\cdot\nabla u$ are the symmetric part and the skew-symmetric part of the advection-diffusion operator ${\cal L} v=-\alpha\Delta v+b\cdot\nabla v$, respectively (recall that $b$ is divergence free in view of~\eqref{div_b_zero}). The stabilization parameter $\tau_{\mathbf{K}}$ is chosen, roughly, of the order of~$\displaystyle \frac{H}{\|b\|_{L^\infty(\Omega)}}$. The choice of $\rho$ leads to different stabilized methods: (i) the Douglas-Wang method (DW) when $\rho=-1$, the Streamline Upwind Petrov-Galerkin method (SUPG) when $\rho=0$ and the Galerkin Least Square method (GLS) when $\rho=1$. The three stabilized methods, applied on problem~\eqref{single_scale_problem}, coincide when~$V_H$ is the $\mathbb{P}^1$ Finite Element space associated to the mesh $\mathcal{T}_H$.

The modification of the discrete bilinear form as in~\eqref{varf_supg_singlescale} allows to obtain the estimate
\begin{align}
\|u-u_H^s\|_{H^1(\Omega)}\leq CH\left(1+\sqrt{\text{Pe}\, H}\right)\|u\|_{H^2(\Omega)},
\label{apriori_supg_singlescale}
\end{align}
where again $C$ is independent of $H$ and $b$. For large P\'eclet numbers (that is, Pe$\, H>1$), this estimate is better than~\eqref{apriori_galerkin_single_scale}. More accurate numerical results are indeed obtained: see Figure~\ref{fig_stab_P1} below. Note also that, in the right-hand sides of~\eqref{apriori_galerkin_single_scale} and~\eqref{apriori_supg_singlescale}, $\| u \|_{H^2(\Omega)}$ depends on $b$, a fact that we will recall in Remark~\ref{rem:fact} below.

Estimate~\eqref{apriori_supg_singlescale} is typically obtained under the assumptions
\begin{align}
\frac{|b(x)|}{2\alpha}H \geq 1 \quad \text{for almost all $x\in\Omega$},
\label{advect_x_regime}
\end{align}
and for the stabilization parameter
\begin{align}
\label{eq:def_tau}
\tau_{\mathbf{K}}(x)=\frac{H}{2|b(x)|} \quad\text{for all }\mathbf{K}\in \mathcal{T}_H.
\end{align} 
For the sake of completeness, and also because we will use similar arguments in Section~\ref{section_preliminary} below for the multiscale setting, we provide the proof of~\eqref{apriori_supg_singlescale} in Appendix~\ref{sec:appA} below. 

\medskip

\begin{remark}
\label{remark_non_coercive}
Notice that all the above analysis assumes that problem~\eqref{single_scale_problem} is coercive (see~\eqref{div_b_zero}). This is usually the case in the literature, see~\cite{MR679322,roos1991numerical}. To the best of our knowledge, the analysis of the stabilized methods of the type~\eqref{varf_supg_singlescale} has not been performed in the non-coercive case. A stabilized numerical method designed for nonsymmetric noncoercive problems is proposed and studied in~\cite{burman2013stabilized}. The method requires to solve the original problem coupled with an adjoint problem using stabilized finite element methods. Error estimates in $H^1$ and $L^2$ norms are proved under the assumption of well-posedness of the problem. Least-square methods for noncoercive elliptic problems have also been studied, see e.g.~\cite{bramble1998least,ku2007least}.
\end{remark}

\begin{remark}
The stabilized methods~\eqref{varf_supg_singlescale} can also be understood in the framework of the Variational Multiscale Methods~\cite{hughes1995multiscale}.
\end{remark}

\begin{remark}
The choice of an optimal stabilization parameter $\tau_{\mathbf{K}}$ is a difficult and sensitive question, since it affects the quality of the numerical approximation. We refer e.g. to~\cite{brezzi-bristeau,franca-frey,principe}. The Variational Multiscale Methods~\cite{hughes1995multiscale} give an interpretation of the stabilization parameter. If we assume $\tau_\mathbf{K}$ to be constant on each mesh element $\mathbf{K}$, the Variational Multiscale Methods yield the formula
\begin{equation}
\label{eq:abstrait}
\tau_\mathbf{K}=\frac{1}{|\mathbf{K}|}\int_\mathbf{K}\int_\mathbf{K} g_\mathbf{K},
\end{equation}
where $g_\mathbf{K}$ is the Green's function of the operator ${\cal L}^\star$ (i.e. the adjoint of ${\cal L}$) with homogeneous Dirichlet boundary conditions on $\partial \mathbf{K}$. Simplifying assumptions are next used to infer, from~\eqref{eq:abstrait}, a practical expression for $\tau_\mathbf{K}$.
\end{remark}

For the sake of illustration, and because it allows us to introduce notions useful for what follows, we briefly consider the one-dimensional example
\begin{align}
-\alpha u''+ bu'=f \quad\text{in }(0,1),
\qquad
u(0)=u(1)=0,
\label{pb_1D_single_scale}
\end{align}
for a constant $b$. In that case, the expression~\eqref{eq:abstrait} can be analytically computed and yields the choice
\begin{align}
\tau_\mathbf{K}=\frac{H}{2|b|}\left(\coth(\text{Pe} \, H)-\frac{1}{\text{Pe} \, H}\right).
\label{formule_tauk_1d}
\end{align}
On Figure~\ref{fig_stab_P1}, we show the exact solution to~\eqref{pb_1D_single_scale} as well as two numerical approximations. We set $\alpha=1/256$, $b=1$, $f=1$ and $H=1/16$, so that we are in the convection-dominated regime. Table~\ref{tabl_HR} shows the relative errors of the methods.

\begin{figure}[htbp]
\begin{center}
\includegraphics[width=0.71\linewidth]{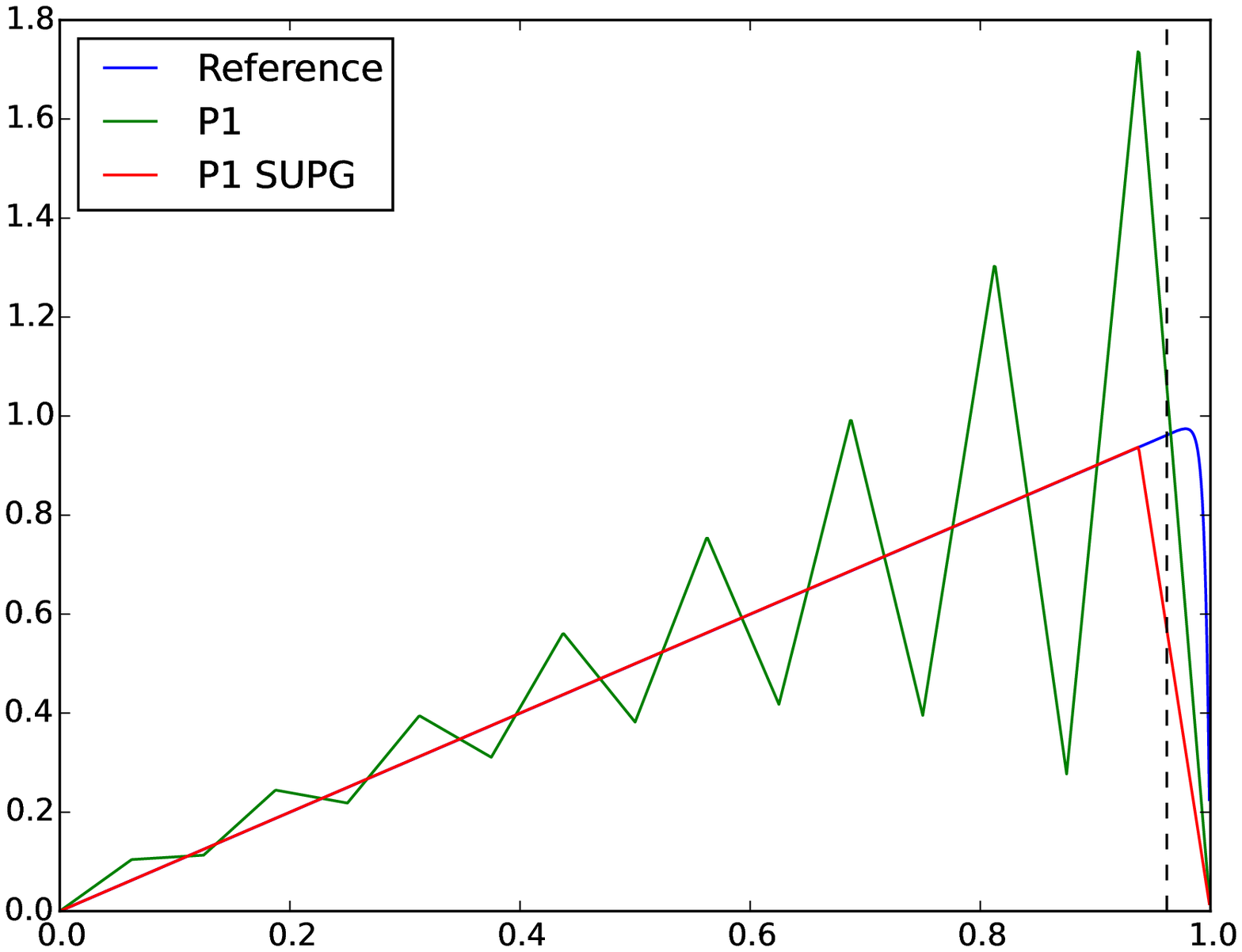}
\includegraphics[width=0.28\linewidth]{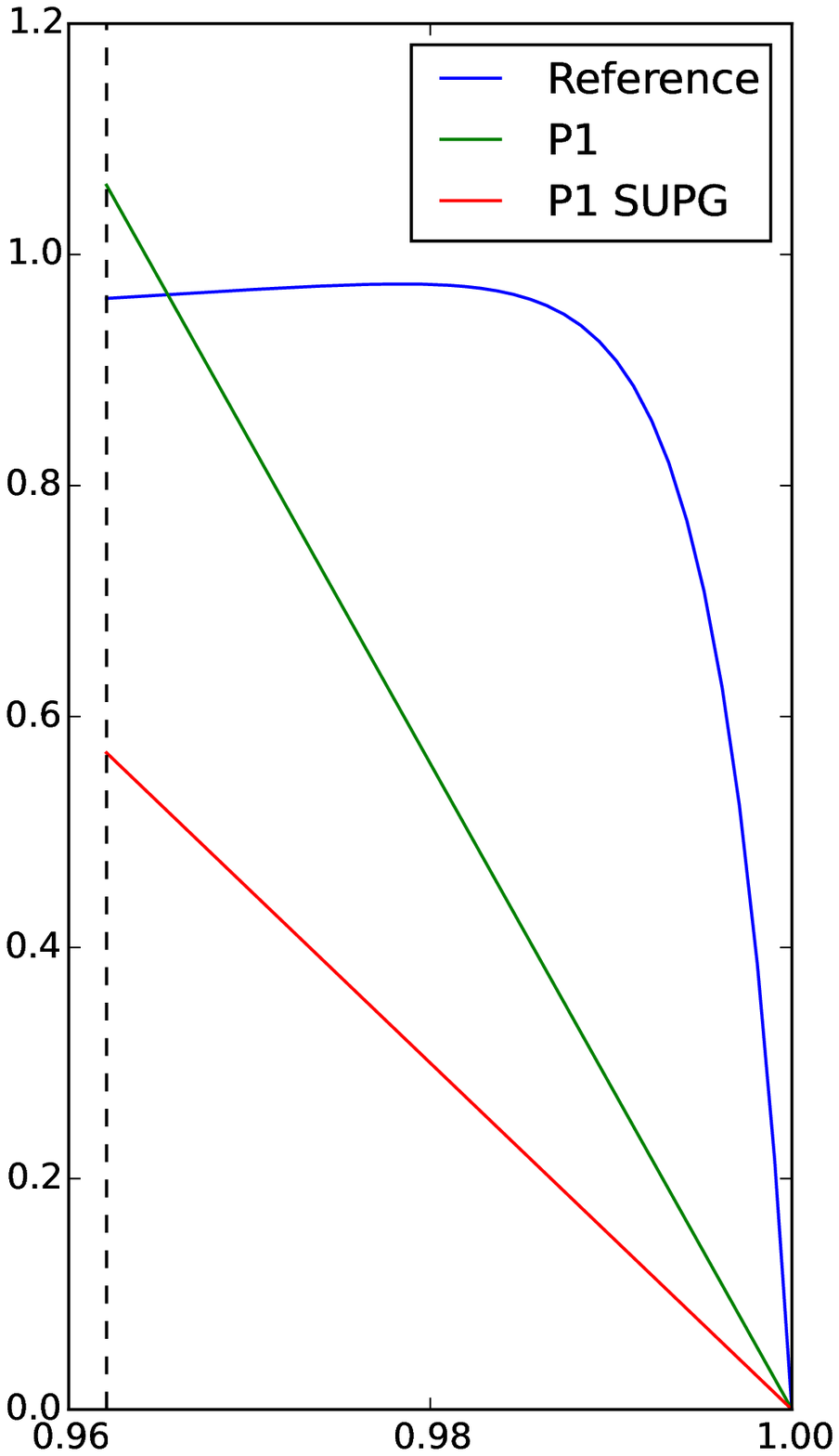}
\end{center}
\caption{Exact and numerical solutions to~\eqref{pb_1D_single_scale} for Pe$\, H=8$. Left: Plot on the whole domain. Right: Close-up on the boundary layer. The vertical dashed line delineates the boundary layer.}
\label{fig_stab_P1}
\end{figure}

\begin{table}[htbp]
\begin{center}
\begin{tabular}{c||c|c||c|c}
& $L^2$: $\mathbb{P}^1$ & $L^2:$ $\mathbb{P}^1$ SUPG & $H^1$: $\mathbb{P}^1$ & $H^1$: $\mathbb{P}^1$ SUPG\\
\hline
Outside the layer & 0.3217 &0.0624& 0.8913 & 0.2228\\
Inside the layer & 0.0297 & 0.1549 & 0.3722 & 0.7163\\
In the whole domain & 0.3513 & 0.2173 & 1.2635 & 0.9391 \\
\end{tabular}
\caption{Relative errors for \eqref{pb_1D_single_scale} with $\alpha=1/256$, $b=1$, $f=1$ and $H=1/16$. The parameter $\tau_\mathbf{K}$ is given by~\eqref{formule_tauk_1d}.
\label{tabl_HR}}
\end{center}
\end{table}

On Figure~\ref{fig_stab_P1}, we can distinguish two regions. Outside the boundary layer, the $\mathbb{P}^1$ SUPG method accurately approximates the solution. It has no spurious oscillations, in contrast to the standard $\mathbb{P}^1$ method. Inside the boundary layer, the $\mathbb{P}^1$ SUPG method only poorly performs. 

\subsubsection{MsFEM approaches}
\label{section_msfem}

We now insert a multiscale character in our problem and temporarily erase the transport field~$b$, which we will shortly reinstate in the next section. We consider the solution $u^\eps \in H^1_0(\Omega)$ to
\begin{align}
-\text{div } (A^\eps\nabla u^\eps) = f \quad\text{in }\Omega,
\qquad
u^\eps=0 \quad\text{on }\partial\Omega.
\label{pb_diff_multiscale}
\end{align}
We assume that the diffusion matrix $A^\eps$, encoding the oscillations at the small scale, is elliptic in the sense that there exists $0<\alpha_1\leq\alpha_2$ such that 
\begin{align}
\forall \eps, \quad \forall \xi \in \R^d, \quad \alpha_1|\xi|^2\leq(A^\eps(x)\xi)\cdot\xi\leq \alpha_2|\xi|^2 \quad \text{a.e. on }\Omega.
\label{condition_alpha}
\end{align}
Throughout this article, we shall perform our theoretical analysis for general, not necessarily symmetric, matrix-valued coefficients $A^\eps$, not necessarily either of the form~$\dis A^\eps=A\left(\cdot/\eps\right)$ for a fixed matrix~$A$ (although one may consider such a case to fix the ideas). In our numerical tests, however, we only consider a scalar coefficient~$A^\eps$. 

The bottom line of the MsFEM is to perform a Galerkin approximation using specific basis functions, which are precomputed (in an offline stage) and adapted to the problem considered. 

On the prototypical multiscale diffusion problem~\eqref{pb_diff_multiscale}, the method, in one of its simplest variant, consists of the following three steps:
\begin{enumerate}[i)]
\item Introduce a discretization of $\Omega$ with a coarse mesh; throughout this article, we work with the $\mathbb{P}^1$ Finite Element space
\begin{equation}
\label{VH-space}
V_H=\text{Span}\left\{\phi^0_i, 1\leq i\leq N_{V_H}\right\} \subset H^1_0(\Omega);
\end{equation}
\item Solve the local problems (one for each basis function for the coarse mesh)
\begin{align}
-\text{div }\left( A^\eps\nabla \psi_i^{\eps,\mathbf{K}}\right) =0 \quad \text{in }\mathbf{K},
\qquad
\psi^{\eps,\mathbf{K}}_i=\phi^0_i \quad\text{on }\partial\mathbf{K},
\label{local_pb_msfem_general}
\end{align} 
on each element $\mathbf{K}$ of the coarse mesh, in order to build the multiscale basis functions. 
\item Apply a standard Galerkin approximation of~\eqref{pb_diff_multiscale} on the space
\begin{align}
\VB = \text{Span} \left\{ \psi^\eps_i, \ 1 \leq i \leq N_{V_H} \right\} \subset H^1_0(\Omega),
\label{def_VB}
\end{align}
where $\psi^\eps_i$ is such that $\left. \psi^\eps_i \right|_\mathbf{\mathbf{K}} = \psi^{\varepsilon,\mathbf{K}}_i$ for all $\mathbf{K}\in\mathcal{T}_H$.
\end{enumerate}

The error analysis of the MsFEM method in the above case~\eqref{pb_diff_multiscale}, for $\dis A^\eps = A_{\rm per}\left(\cdot/\eps\right)$ with $A_{\rm per}$ a fixed periodic matrix, has been performed in~\cite{hou1999convergence} (see also~\cite[Theorem 6.5]{efendiev2009multiscale}). The main result is stated in the following Theorem.

\begin{theorem}
We consider the periodic case $\dis A^\eps(x)=A_{\rm per}\left(x/\eps\right)$. We assume that $A_{\rm per}$ is H\"older continuous, that the solution $u^\star$ to the homogenized problem associated to~\eqref{pb_diff_multiscale} belongs to $W^{2,\infty}(\Omega)$ and that $H>\eps$. Let $u^\eps_H$ be the MsFEM approximation of the solution $u^\eps$ to~\eqref{pb_diff_multiscale}. Then
\begin{align}
\|u^\eps-u^\eps_H\|_{H^1(\Omega)}\leq C \left( H + \sqrt{\eps} + \sqrt{\frac{\eps}{H}} \right),
\label{error_estimate_nos}
\end{align}
where $C$ is a constant independent of $H$ and $\eps$.
\label{theor_msfem_nos}
\end{theorem}

When the coarse mesh size $H$ is close to the scale $\eps$, a resonance phenomenon, encoded in the term $\sqrt{\eps/H}$ in~\eqref{error_estimate_nos}, occurs and deteriorates the numerical solution. The oversampling method~\cite{hou1997multiscale} is a popular technique to reduce this effect. In short, the approach, which is non-conforming, consists in setting each local problem on a domain slightly larger than the actual element considered, so as to become less sensitive to the arbitrary choice of boundary conditions on that larger domain, and next truncate on the element the functions obtained. That approach allows to significantly improve the results compared to using linear boundary conditions as in~\eqref{local_pb_msfem_general}. In the periodic case, we have the following estimate (see~\cite{efendiev2000convergence}). 

\begin{theorem}
Assume the setting and the notation of Theorem~\ref{theor_msfem_nos}. Assume additionally that the distance between an element $\mathbf{K}$ and the boundary of the macro element used in the oversampling is larger than~$H$. Then
$$
\|u^\eps-u^\eps_H\|_{H^1(\mathcal{T}_H)}\leq C \left( H+\sqrt{\eps}+ \frac{\eps}{H} \right),
$$
where $\dis \| u^\eps-u^\eps_H \|_{H^1(\mathcal{T}_H)} = \sqrt{\sum_{\mathbf{K} \in \Th} \| u^\eps-u^\eps_H \|^2_{H^1(\mathbf{K})}}$ is the $H^1$ broken norm of $u^\eps-u^\eps_H$.
\label{theor_msfem_os}
\end{theorem}

\begin{remark}
The boundary conditions imposed in~\eqref{local_pb_msfem_general} are the so-called linear boundary conditions. Besides the linear boundary conditions, and the oversampling technique alluded to above, there are many other possible boundary conditions for the local problems. They may give rise to conforming, or non-conforming approximations. The choice sensitively affects the overall accuracy. We will explore this issue, in our specific context, in Section~\ref{change_bc} below. 
\end{remark}

It is important to notice that the estimates of Theorems~\ref{theor_msfem_nos} and~\ref{theor_msfem_os} hold true assuming that the multiscale basis functions employed to compute the approximation~$u^\eps_H$ are the \emph{exact} solutions of the local problems. In practice of course, the local problems~\eqref{local_pb_msfem_general} are only approximated numerically, using a fine mesh of size $h$ sufficiently small to capture the oscillations at scale~$\eps$. 

\bigskip

As mentioned above, our purpose is to understand how to adapt the stabilization methods and the MsFEM methods in order to efficiently approximate 
\begin{align}
-\text{div } (A^\eps\nabla u^\eps) + b\cdot\nabla u^\eps = f \quad\text{in }\Omega,
\qquad
u^\eps=0 \quad\text{on }\partial\Omega,
\label{pb_multiscale}
\end{align}
where $A^\eps \in (L^\infty(\Omega))^{d\times d}$ satisfies~\eqref{condition_alpha}, $b \in (L^\infty(\Omega))^d$ and $f \in L^2(\Omega)$. Notice that the transport field~$b$ is assumed to be independent of $\eps$. We also choose it divergence-free as in~\eqref{div_b_zero}. The variational formulation of~\eqref{pb_multiscale} is:
\begin{align}
\text{Find }u^\eps \in H^1_0(\Omega) \text{ such that, for any $v \in H^1_0(\Omega)$,}
\quad a^\eps(u^\eps,v)=F(v),
\label{varf_pb_multiscale}
\end{align}
where
\begin{equation}
\label{eq:varf_plus}
a^\eps(u,v)=\int_\Omega (A^\eps\nabla u) \cdot \nabla v + (b\cdot\nabla u)\, v, 
\qquad
F(v)=\int_\Omega fv.
\end{equation}
We now introduce in Section~\ref{section_our_four_appr} below the four numerical approaches we consider.

\subsection{Our four numerical approaches}
\label{section_our_four_appr}

\subsubsection{The classical \MsFEMB{} and its stabilized version}
\label{section_msfemB}

The classical \MsFEMB{} described in Section~\ref{section_msfem} is the first approach we consider. It performs a Galerkin approximation of~\eqref{pb_multiscale} on the space~\eqref{local_pb_msfem_general}-\eqref{def_VB}. Notice that in this approximation, the transport term~$b\cdot\nabla$, although present in the equation~\eqref{pb_multiscale}, is absent from the local problems~\eqref{local_pb_msfem_general} and thus from the definition of the basis functions. It is immediate to realize that this approach coincides with the standard $\mathbb{P}^1$ method on~\eqref{single_scale_problem} when $A^\eps= \alpha \, \text{Id}$. Consequently, the method is expected to be unstable in the convection-dominated regime, as recalled in Section~\ref{section_stabilized_methods}, and this is indeed observed in practice, as will be seen in Section~\ref{section_sensitivity_peclet}.

This motivates the introduction of a stabilized version of this method, which is the adaptation to the multiscale context of the classical SUPG method. As we shall now see, some difficulty arises regarding the consistency of the approach, owing to the fact that the basis functions we use in practice are only approximate.

First, we consider the exact approximation space $\VB$ defined by~\eqref{def_VB}. The SUPG stabilization, readily applied to our problem~\eqref{varf_pb_multiscale}, yields the following variational formulation:
\begin{align}
&\text{Find }u^\eps_H\in \VB \text{ such that, for any $v_H^\eps\in \VB$,} \nonumber
\\
&a^\eps(u^\eps_H,v^\eps_H)
+a_{\text{stab}}(u^\eps_H,v^\eps_H)
= F(v_H^\eps)
+F_{\text{stab}}(v_H^\eps),
\label{varf_msfemB_SUPG_helene}
%\label{varf_msfemB_SUPG_ideal}
\end{align}
where we recall that the SUPG stabilization terms are (see~\eqref{single_scale_stab_term} and~\eqref{single_scale_F_term})
\begin{align}
a_{\text{stab}}(u_H^\eps,v_H^\eps) &= \sum_{\textbf{K}\in\mathcal{T}_H} \Big( \tau_\mathbf{K} \left(-\text{div }\left( A^\eps\nabla u_H^\eps\right) + b\cdot \nabla u_H^\eps\right),b\cdot\nabla v_H^\eps \Big)_{L^2(\mathbf{K})},
\label{supg_ideal}
\\
F_{\text{stab}}(v_H^\eps) &= \sum_{\mathbf{K}\in\mathcal{T}_H} \bigg(\tau_{\mathbf{K}} f,b\cdot \nabla v_H^\eps\bigg)_\mathbf{K}. 
\nonumber
\end{align} 
The method is, as is well known, strongly consistent. Because of the definition of the approximation space $\VB$, we have
\begin{align}
a_{\text{stab}}(u_H^\eps,v_H^\eps)=a_{\text{upw}}(u_H^\eps,v_H^\eps) \quad \text{for any } (u_H^\eps,v_H^\eps)\in(\VB)^2,
\label{supg_trick}
\end{align}
where 
\begin{align}
a_{\text{upw}}(u_H^\eps,v_H^\eps)=\sum_{\textbf{K}\in\mathcal{T}_H} \left(\tau_\mathbf{K} b\cdot\nabla u_H^\eps,b\cdot\nabla v_H^\eps\right)_{L^2(\mathbf{K})}.
\label{supg_real}
\end{align}
In practice however, we only know a discrete approximation~$\psi^{\eps,h}$, on a fine mesh $\mathbf{K}_h$, of the solution $\psi^{\eps}$ to~\eqref{local_pb_msfem_general}. Put differently, we manipulate~$\VBh = \text{Span} \{\psi^{\eps ,h}_i,1\leq i\leq N_{V_H}\}$ instead of~$\VB$. It follows that, for example when~$A^\eps\in \mathcal{C}^0(\overline{\Omega})$ and we use a $\mathbb{P}^1$ approximation on a fine mesh $\mathbf{K}_h$ for the local problem~\eqref{local_pb_msfem_general}, $A^\eps\nabla u_{H,h}^\eps$ may be discontinuous at the edges of the mesh~$\mathbf{K}_h$, and $-\text{div }( A^\eps\nabla u_{H,h}^\eps) \notin L^1_{\rm loc}(\mathbf{K})$.

We may consider at least two ways to circumvent that difficulty. First, if the matrix coefficient $A^\eps$ is locally sufficiently regular, we may define the stabilization term as
\begin{multline*}
\widetilde{a}_{\text{stab}}(u_{H,h}^\eps,v_{H,h}^\eps) \\ = \sum_{\mathbf{K}\in\mathcal{T}_H}\sum_{\kappa\subset K_h}\Big( \tau_\mathbf{K} \left(-\text{div }\left( A^\eps\nabla u_{H,h}^\eps\right) + b\cdot \nabla u_{H,h}^\eps\right),b\cdot\nabla v_{H,h}^\eps \Big)_{L^2(\mathbf{\kappa})}.
%\label{supg_fine}
\end{multline*} 
When, as is the case here, we employ a $\mathbb{P}^1$ approximation on $\mathbf{K}_h$, all we need for this stabilization term to make sense is that the vector field~$\text{div} \left(A^\eps\right)$ belongs to~$L^1(\kappa)$ for all $\kappa\subset K_h$. This is more demanding than the simple classical assumption $A^\eps\in L^\infty(\Omega)$. Under this assumption, we obtain a strongly consistent stabilized method. We will however not proceed in this direction and favor an alternate approach, to which we now turn.

Based upon the observation~\eqref{supg_trick} for the "ideal" space~$\VB$, we may use the stabilization term~\eqref{supg_real} rather than~\eqref{supg_ideal}. In contrast to~\eqref{supg_ideal}, the quantity~\eqref{supg_real} is also well defined on $\VBh$. And this holds true without any additional regularity assumption on $A^\eps$. The \MsFEMBSUPG{} method we employ is hence defined by the following variational formulation:
\begin{align}
&\text{Find }u^\eps_{H,h}\in \VBh \text{ such that, for any $v_{H,h}^\eps\in \VBh$,}\nonumber\\
&a^\eps(u^\eps_{H,h},v^\eps_{H,h})
+a_{\text{upw}}(u^\eps_{H,h},v^\eps_{H,h})
= F(v_{H,h}^\eps)
+F_{\text{stab}}(v_{H,h}^\eps).
\label{varf_msfemB_SUPG}
\end{align}
We emphasize that employing that stabilization comes at a price: we give up on strong consistency. We provide in Section~\ref{section_msfemB_SUPG_1D}, Theorem~\ref{thm_HRFM_Bstabh} below, an error estimate in the one-dimensional setting for this method. Despite the absence of consistency, we can still prove that the method is convergent.

\subsubsection{The \MsFEMA{} variant}
\label{section_msfemA}

In contrast to our first two approaches, the \MsFEMA{} approach we discuss in this section accounts for the transport field in the local problems. For each mesh element $\mathbf{K}\in\mathcal{T}_H$, we indeed now consider 
\begin{equation}
-\text{div }\left(A^\eps\nabla\phi^{\eps,\mathbf{K}}_i\right)+b\cdot\nabla\phi^{\eps,\mathbf{K}}_i = 0 \ \ \text{ in $\mathbf{K}$},
\qquad
\phi^{\eps,\mathbf{K}}_i = \phi^0_i \ \ \text{ on $\partial \mathbf{K}$},
\label{local_problem_nos_msfemA}
\end{equation}
instead of~\eqref{local_pb_msfem_general}, and next the approximation space
$$
\VA = \text{Span} \left\{ \phi^\eps_i, \ 1 \leq i \leq N_{V_H} \right\} \subset H^1_0(\Omega)
$$
defined as in~\eqref{def_VB}. Problem~\eqref{local_problem_nos_msfemA} is an advection-diffusion problem with, in principle, a high P\'eclet number. Nevertheless, the problem is local and is to be solved offline, so we may easily employ a mesh size sufficiently fine to avoid the issues presented in Section~\ref{section_stabilized_methods}. 

There is however a difficulty in considering~\eqref{local_problem_nos_msfemA} and $b$-dependent basis functions~$\phi^{\eps,\mathbf{K}}_i$. In the context where we want to repeatedly solve~\eqref{pb_multiscale} for multiple~$b$, for instance when $b$ depends on an external parameter such as time, the method becomes prohibitively expensive as we will see in Section~\ref{section_cost}. 

\medskip 

We note in passing the following consistency. In the one-dimensional single-scale example~\eqref{pb_1D_single_scale}, the stiffness matrix of the \MsFEMA{} method is
$$
M_{\text{\MsFEMA}}=\text{Tridiag}\left( \frac{-b\exp\left(bH/\alpha\right)}{\exp\left(bH/\alpha\right)-1}, 
|b|\coth\left(\frac{|b|H}{2\alpha}\right), 
\frac{-b}{\exp\left(bH/\alpha\right)-1}
\right).
$$
It then coincides with the stiffness matrix $M_{\mathbb{P}^1 \text{SUPG}}$ of the $\mathbb{P}^1$ SUPG method with $\tau_{\mathbf{K}}$ given by~\eqref{formule_tauk_1d}. 

We also note that, in view of~\eqref{supg_ideal}--\eqref{local_problem_nos_msfemA}, we have that $a_{\text{stab}}(u_H^{\eps,\text{Adv}},v_H^{\eps,\text{Adv}})=0$ for any~$(u_H^{\eps,\text{Adv}},v_H^{\eps,\text{Adv}})\in(\VA)^2$. Such a stabilization is therefore void on the \MsFEMA{} method. Actually, we shall see in the numerical tests of Section~\ref{section_accuracies} that the \MsFEMA{} method is only moderately sensitive to the P\'eclet number.

\medskip

MsFEM type basis functions depending on the transport term for multiscale advection-diffusion problems have already considered in the literature. In~\cite{park2004multiscale}, two settings are investigated. The \MsFEMA{} is first applied to the time-dependent multiscale advection-diffusion equation
$$
\partial_tu^\eps -\Delta u_\eps +\frac{1}{\eps}b\left(\frac{\cdot}{\eps}\right)\cdot\nabla u^\eps = 0 \quad\text{ in }\R^2,
%\label{pb_Park}
$$
with $b=\nabla^\perp\psi$ where $\dis\psi\left(x\right)=\psi\left(x_1,x_2\right)=\frac{1}{4\pi^2}\sin({2\pi x_1})\sin({2\pi x_2})$. The field $b$ is thus $\Z^d$-periodic, divergence-free and of mean zero. The purpose is then to only capture macroscopic properties of the solution $u^\eps$. Also in~\cite{park2004multiscale}, the \MsFEMA{} is investigated on the problem
$$
-\Delta u^\eps +b^{\eps}\cdot\nabla u^\eps = f,
$$
with $b^\eps \in (L^\infty(\Omega))^2$ and $f \in L^2(\Omega)$. Only the following $L^2$ error estimate
$$
\frac{\|u^\eps-u^\eps_H\|_{L^2(\Omega)}}{\|u^\eps\|_{L^2(\Omega)}}\leq C\frac{\eps}{H}+CH^2\|f\|_{L^2(\Omega)}
$$
is derived, and not an~$H^1$ estimate which would be sensitive to how well the fine oscillations are captured by the numerical approach. It is completed in the periodic case, where $\dis b^\eps(x) = \frac{1}{\eps} b_{\rm per}\left(\frac{x}{\eps}\right)$ for a fixed, periodic, divergence-free function $b_{\rm per}$ of mean zero, under some assumptions which have been numerically verified on some examples. An experimental study of convergence is performed and shows good agreement with the above theoretical error estimate. 

\medskip

A second reference we wish to cite is~\cite{ouaki2013etude}. The author studies there the problem
\begin{align*}
\left\{\begin{aligned}
&\rho^\eps \partial_t u^\eps - \text{div} (A^\eps \nabla u^\eps)+ \frac{1}{\eps} b^\eps \cdot \nabla u^\eps=0 \quad\text{ in }(0,1)^d\times(0,T),\\
&u^\eps(0,\cdot)=u^0 \quad\text{ in }(0,1)^d,\\
&u^\eps(t,\cdot) \text{ is }(0,1)^d\text{-periodic,}
\end{aligned}\right.
%\label{ifp}
\end{align*}
where $u^0\in W^{m,\infty}_{\rm per}((0,1)^d)$ with $m\geq 3$. The functions $\rho^\eps \in L^\infty((0,1)^d)$, $b^\eps \in (L^\infty((0,1)^d))^d$ and $A^\eps \in (L^\infty((0,1)^d))^{d\times d}$ do not depend on time. It is assumed that there exists a constant $\rho_m>0$ such that $\rho^\eps\geq\rho_m$ a.e. on $(0,1)^d$, and that $b^\eps$ is divergence-free. In contrast to~\cite{park2004multiscale}, the mean of $b^\eps$ is not assumed to vanish (but periodic boundary conditions are imposed on $\partial (0,1)^d$). In the convection-dominated regime, the problem is stabilized using the characteristics method for integrating the transport operator $\dis\partial_t +\frac{b^\star_H}{\eps}\cdot\nabla$, and the multiscale finite element method for the remaining part of the convection term, i.e. $\dis\frac{b^\eps-\rho^\eps b^\star_H}{\eps}\cdot\nabla$, where $\dis b^\star_H|_{\mathbf{K}}=\frac{\int_{\mathbf{K}}b^\eps}{\int_{\mathbf{K}}\rho^\eps}$ for all $\mathbf{K}\in\Th$. The MsFEM approach which is used in~\cite{ouaki2013etude} is inspired by the variant of the Multiscale Finite Element approach introduced in~\cite{allaire2005multiscale} for purely diffusive problems. The multiscale basis functions are thus defined by~$\phi^\eps_j(x)=\phi^0_j\Big( w^{\eps, H}(x) \Big)$ for~$1 \leq j \leq N_{V_H}$, where $\phi^0_j$ are the $\mathbb{P}^1$ basis functions and~$w^{\eps, H}|_{\mathbf{K}} = (w^{\eps, \mathbf{K}}_{1},\dots,w^{\eps, \mathbf{K}}_{d})$ for each~$\mathbf{K}\in\Th$, where, for any $i=1,\dots,d$, the function $w^{\eps, \mathbf{K}}_i$ is the solution to
$$
-\text{div }\left(A^\eps\nabla w^{\eps, \mathbf{K}}_{i}\right)+\frac{b^\eps-\rho^\eps b^\star_H}{\eps}\cdot\nabla w^{\eps, \mathbf{K}}_{i} = 0 \ \ \text{ in $\mathbf{K}$},
\qquad
w^{\eps, \mathbf{K}}_{i} = x_i \ \ \text{ on $\partial \mathbf{K}$}.
$$
Note that, as in~\eqref{local_problem_nos_msfemA}, the basis functions depend on the convection field. An error estimate is established in~\cite{ouaki2013etude} for the periodic case. 

\subsubsection{A splitting approach}
\label{section_splitting}

The fourth, and last approach we consider is a \emph{splitting method} that decomposes~\eqref{pb_multiscale} into a single-scale, convection-dominated problem and a multiscale, purely diffusive problem. The main motivation for considering such a splitting approach is the non-intrusive character of the approach. In practice, one may couple legacy codes that are already optimized for each of the two subproblems.

Of course, splitting methods have been used in a large number of contexts. To cite only a couple of works relevant to our context, we mention~\cite{hundsdorfer2003numerical} for a review on the splitting methods for time-dependent advection-diffusion equations, and~\cite{szymczak1988analysis} for the introduction of a viscous splitting method based on a Fourier analysis for the steady-state advection-diffusion equation.

Our splitting approach for~\eqref{pb_multiscale} is the following. We define the iterations by
\begin{align}
&\left\{\begin{aligned}
&-\alpha_{\rm spl}\Delta u_{2n+2} + b\cdot\nabla u_{2n+2} = f + b\cdot\nabla (u_{2n}-u_{2n+1}) \quad\text{ in }\Omega, \\
&u_{2n+2}=0 \quad\text{ on }\partial\Omega ,
\end{aligned}\right.\label{pb_spl_1}\\
&\left\{\begin{aligned}
&-\text{div }(A^\eps\nabla u_{2n+3})=-\alpha_{\rm spl}\Delta u_{2n+2}\quad\text{ in }\Omega,\\
&u_{2n+3}=0\quad\text{ on }\partial\Omega,
\end{aligned}\right.
\label{pb_spl_2}
\end{align}
with $\alpha_{\rm spl}>0$. The initialization is e.g.~$u_0=u_1=0$. 

\medskip

The functions $u_{2n}$ with even indices are approximations defined on a coarse mesh, using $\mathbb{P}^1$ finite elements, and, since our context is that of advection-dominated problems, obtained with a SUPG formulation, as explained in Section~\ref{section_stabilized_methods}. Note that, in the right-hand side of~\eqref{pb_spl_1}, the term $-b\cdot\nabla u_{2n+1}$ is integrated on a fine mesh, as we expect this term to vary at the scale $\eps$. The discretized variational formulation of~\eqref{pb_spl_1} reads
\begin{align}
&\text{Find }u^{H}_{2n+2}\in V_H \text{ such that, for any $v\in V_H$,}\nonumber\\
&a^0(u^{H}_{2n+2},v)+a_{\text{stab}}(u^{H}_{2n+2},v)
= F^1(v)+F_{\text{stab}}(v),
\label{pb_spl_1_discrete_varf}
\end{align}
where $a_{\text{stab}}$ and $F_{\text{stab}}$ are defined by~\eqref{single_scale_stab_term} and~\eqref{single_scale_F_term}, and 
\begin{align*}
&a^0(u,v)= \int_\Omega \alpha_{\rm spl}\nabla u\cdot\nabla v + (b\cdot\nabla u) v, 
\\ 
& F_1(v) = \int_\Omega f_1 v \quad \text{with} \quad f_1=f + b\cdot\nabla (u_{2n}^{H}-u^{H}_{2n+1}).
\end{align*}

\medskip

The functions $u_{2n+1}$ with odd indices are obtained using a MsFEM type approach. A natural choice for the discretization of this problem is the \MsFEMB{} method presented in Section~\ref{section_msfem} above. The variational formulation is
\begin{align}
&\text{Find }u^{H}_{2n+3} \in \VB \text{ such that, for any $v \in \VB$,}\nonumber \\
& \int_\Omega (A^\eps\nabla u^{H}_{2n+3}) \cdot\nabla v=\int_\Omega \alpha_{\rm spl}\nabla u_{2n+2}^{H} \cdot\nabla v,
\label{pb_spl_2_discrete_varf}
\end{align}
where $\VB$ is defined by~\eqref{def_VB}.

\medskip

The termination criterion we use for the iterations is fixed as follows. Equation~\eqref{pb_spl_1_discrete_varf} is equivalent to the linear system~$M_0 [u^{H}_{2n+2}]=F^{0,H} + M_2[u^{H}_{2n}]-M_3 [u^{H}_{2n+1}]$, where $[u^{H}_{2n}]$ is the vector representing the Finite Element function $u^{H}_{2n}$ (i.e.~$\displaystyle u^{H}_{2n}(x)=\sum_{i=1}^{N_{V_H}}[u^{H}_{2n}]_i \ \phi^0_i(x) $) and likewise for $[u^{H}_{2n+2}]$ and $[u^{H}_{2n+1}]$. We stop the iterations if~$\left\|M_0 [u^{H}_{2n+2}]-(F^{0,H} + M_2[u^{H}_{2n+2}]-M_3 [u^{H}_{2n+3}])\right\|<10^ {-9}$.

\medskip

We immediately note that, if we \emph{assume} that~$u_{2n}$ and $u_{2n+1}$ converge to some $u_{\text{even}}$ and $u_{\text{odd}}$, respectively, then we have
\begin{align}
&-\alpha_{\rm spl}\Delta u_{\text{even}} = f - b\cdot\nabla u_{\text{odd}} \quad\text{ in }\Omega, 
\qquad
u_{\text{even}}=0 \quad\text{ on }\partial\Omega,
\label{pb_spl_lim_1}
\\
&-\text{div }(A^\eps\nabla u_{\text{odd}})=-\alpha_{\rm spl}\Delta u_{\text{even}}\quad\text{ in }\Omega,
\qquad
u_{\text{odd}}=0\quad\text{ on }\partial\Omega.
\label{pb_spl_lim_2}
\end{align} 
Adding~\eqref{pb_spl_lim_1} and~\eqref{pb_spl_lim_2}, we get that $u_{\text{odd}}$ is actually the solution to~\eqref{pb_multiscale}. A detailed analysis and a proof, under suitable assumptions, of the actual convergence of our splitting approach is provided in Section~\ref{prelim_splitting} below.

In theory however, there is no guarantee that, in all circumstances, the naive, fixed point iterations~\eqref{pb_spl_1}-\eqref{pb_spl_2} above converge. In all the test cases presented in Section~\ref{section_accuracies}, the iterations indeed converge. With a view to address difficult cases where the iterations might not converge, we design and study in Section~\ref{prelim_splitting} a possible alternate iteration scheme, based on a damping, which, for a well adjusted damping parameter, unconditionally converges. As will be shown in Section~\ref{section_accuracies}, this unconditional convergence comes however at the price of yielding results that are generically less accurate and longer to obtain than when using the direct fixed point iteration, when the latter converges of course. We therefore only advocate this alternate approach in the difficult cases.

\medskip

As will be seen in Section~\ref{section_accuracies} below, the splitting method and the \MsFEMBSUPG{} method provide numerical solutions of approximately identical accuracy. The non-intrusive character of the splitting method is somehow balanced by its online cost which, owing to the iterations, is larger than that of the \MsFEMBSUPG{} method. This is especially true in a multi-query context and/or for problems of large sizes only amenable to iterative linear algebra solvers. 

\section{Elements of theoretical analysis}
\label{section_preliminary}

This section is devoted to the theoretical study of our four numerical approaches. Throughout the section, we mostly work in the one-dimensional setting (in Sections~\ref{prelim_msfem}, \ref{section_msfemB_SUPG_1D} and~\ref{prelim_msfem-adv}), with the notable exception of the mathematical study in Section~\ref{prelim_splitting} of the iteration scheme~\eqref{pb_spl_1}--\eqref{pb_spl_2} used in our splitting method and of an alternative unconditionally convergent iteration scheme, which is performed with all the possible generality. Some of our results were first established in the preliminary study~\cite{HR2013}.

\medskip

The \MsFEMB{} method, the \MsFEMBSUPG{} method and the \MsFEMA{} method are studied, in Sections~\ref{prelim_msfem}, \ref{section_msfemB_SUPG_1D} and~\ref{prelim_msfem-adv} respectively, on the one-dimensional problem
\begin{equation}
-\frac{d}{d x}\left(A^\eps \frac{du^{\eps}}{d x} \right) + b \frac{du^\eps}{d x}=f \ \ \text{in $\Omega=(0,L)$},
\qquad
u^\eps(0)=u^\eps(L)=0,
\label{pb_helene}
\end{equation}
with a constant convection field $b \neq 0$, $f \in L^2(0,L)$ and a diffusion coefficient such that $0<\alpha_1\leq A^\eps(x) \leq\alpha_2$ a.e. on $\Omega$. We estimate the error in terms of $\eps$, $b$, the macroscopic mesh size $H$ and possibly the mesh size $h$ used to solve the local problems. 

\medskip

For further use, we first establish the following two propositions, namely Propositions~\ref{a_cea_lemma} and~\ref{bound_HR}. The first one is a C\'ea-type result, which holds in any dimension.

\begin{proposition}
For $d \geq 1$, let $u_H$ be the numerical solution obtained by applying any conforming Galerkin method to problem~\eqref{pb_multiscale} (on some finite dimensional space $W_H$). Then, if the matrix $A^\eps$ is symmetric, elliptic in the sense of~\eqref{condition_alpha} and $b$ satisfies~\eqref{div_b_zero}, we have
\begin{multline*}
|u^\eps-u_H|_{H^1(\Omega)} \\ \leq \inf_{v_H\in W_H}\left(\sqrt{\frac{\alpha_2}{\alpha_1}} \ |u^\eps-v_H|_{H^1(\Omega)} + \frac{\left\| \, \sqrt{b^T(A^\eps)^{-1} b} \, \right\|_{L^\infty(\Omega)}}{\sqrt{\alpha_1}} \ \|u^\eps-v_H\|_{L^2(\Omega)}\right),
\end{multline*}
where $|v|_{H^1(\Omega)}=\|\nabla v\|_{L^2(\Omega)}$ for any $v \in H^1(\Omega)$.
\label{a_cea_lemma}
\end{proposition}

\begin{proof}
We follow the proof of~\cite{HR2013}. Using~\eqref{div_b_zero} and the Galerkin orthogonality, we have, for any $v_H \in W_H$,
\begin{align}
& \int_\Omega \left[ \nabla (u^\eps -u_H) \right]^T A^\eps \nabla (u^\eps -u_H) \nonumber\\
&= a^\eps(u^\eps -u_H,u^\eps -u_H)\nonumber\\
& = a^\eps(u^\eps -u_H,u^\eps -v_H)\nonumber\\
& = \int_\Omega \left[ \nabla (u^\eps -v_H) \right]^T A^\eps \nabla (u^\eps -u_H) + \int_\Omega \left[ b\cdot \nabla (u^\eps -u_H) \right] (u^\eps -v_H),
\label{eq:estim}
\end{align}
where $a^\eps$ is defined by~\eqref{eq:varf_plus}. Considering the square root~$(A^\eps)^{1/2}$ of the symmetric positive definite matrix~$A^\eps(x)$, and using the Cauchy-Schwarz inequality, we have, on the one hand,
\begin{multline*}
\int_\Omega \left[ \nabla (u^\eps -v_H) \right]^T A^\eps \nabla (u^\eps -u_H) 
\\
\leq 
\left(\int_\Omega \left[ \nabla (u^\eps -u_H) \right]^T A^\eps \nabla (u^\eps -u_H) \right)^{1/2} \left(\int_\Omega \left[ \nabla (u^\eps -v_H) \right]^T A^\eps \nabla (u^\eps -v_H) \right)^{1/2},
\end{multline*}
and, on the other hand,
\begin{align*}
& \int_\Omega \left[ b\cdot \nabla (u^\eps -u_H) \right] (u^\eps -v_H) \\
&= \int_\Omega \left[ (A^\eps)^{-1/2}b \right] \cdot \left[ (A^\eps)^{1/2} \nabla (u^\eps -u_H) \right] (u^\eps -v_H)\\
&\leq \left(\int_\Omega b^T (A^\eps)^{-1} b \ (u^\eps -v_H)^2\right)^{1/2}\left(\int_\Omega \left[ \nabla (u^\eps -u_H) \right]^T A^\eps \nabla (u^\eps -u_H) \right)^{1/2}.
\end{align*}
Inserting these estimates in~\eqref{eq:estim}, we obtain
\begin{multline*}
\left(\int_\Omega \left[ \nabla (u^\eps -u_H) \right]^T A^\eps \nabla (u^\eps -u_H) \right)^{1/2}\\
\leq \left(\int_\Omega \left[ \nabla (u^\eps -v_H) \right]^T A^\eps \nabla (u^\eps -v_H) \right)^{1/2} + \left(\int_\Omega b^T (A^\eps)^{-1} b \ (u^\eps -v_H)^2\right)^{1/2}.
\end{multline*}
Using~\eqref{condition_alpha}, we infer that, for any $v_H \in W_H$,
$$
\sqrt{\alpha_1}|u^\eps - u_H|_{H^1(\Omega)}\leq \sqrt{\alpha_2}|u^\eps - v_H|_{H^1(\Omega)} + \left\| \, \sqrt{ b^T (A^\eps)^{-1} b } \, \right\|_{L^\infty(\Omega)} \, \|u^\eps - v_H\|_{L^2(\Omega)}. 
$$
This concludes the proof of Proposition~\ref{a_cea_lemma}.
\end{proof}

\begin{proposition}
Assume the ambient dimension is one. Consider $u^\eps \in H^1_0(\Omega)$ the solution to~\eqref{pb_helene}. If $\dis \frac{|b|L}{\alpha_2} \geq 1$, then
$$
|u^\eps|_{H^1(\Omega)}\leq \frac{\sqrt{2 \alpha_2 L}}{\alpha_1\sqrt{|b|}} \ \|f\|_{L^2(\Omega)}.
$$
\label{bound_HR}
\end{proposition}

\begin{proof}
Without loss of generality, we can assume that $b>0$. We decompose the right-hand side into a zero mean part (considered in Step 1 of the proof) and a constant part (considered in Step 2).

\medskip

\noindent {\bf Step 1.} 
We first consider the case when the mean of $f$ vanishes. Introduce $\dis F(x) = \int_0^x f$ and note that $b u^\eps -F \in H^1_0(\Omega)$, so that we can use it as test function in~\eqref{pb_helene}. This leads to
$$
\int_\Omega A^\eps (u^\eps)'(bu^\eps-F)' + b(u^\eps)'(bu^\eps-F)= \int_\Omega f(bu^\eps-F),
$$
which also reads as $\dis \int_\Omega A^\eps (u^\eps)'(bu^\eps-F)' + (bu^\eps-F)'(bu^\eps-F)= 0$, whence
$$
\int_\Omega bA^\eps (u^\eps)'(u^\eps)' = \int_\Omega A^\eps (u^\eps)'f.
$$
Using the Cauchy-Schwarz inequality and the fact that $\alpha_2 \leq |b|L$, we get
\begin{align}
\label{eq:helene_1}
|u^\eps|_{H^1(\Omega)}\leq \frac{\alpha_2}{\alpha_1 |b|}\|f\|_{L^2(\Omega)}\leq \frac{\sqrt{\alpha_2 L}}{\alpha_1 \sqrt{|b|}} \ \|f\|_{L^2(\Omega)}.
\end{align}

\medskip

\noindent {\bf Step 2.} 
We now consider the case when $f$ is constant. Without loss of generality, we can assume that $f \equiv 1$. The proof is based on the maximum principle. Introduce the function $\dis v(x)=\frac{\alpha_2}{b}\int_0^x \frac{dy}{A^\eps(y)} - u^\eps(x)$. This function is such that 
$$
-\left(A^\eps v' \right)' + bv' = \alpha_2(A^{\eps})^{-1} -1 \geq 0, \quad
v(0) = 0 \ \ \text{and} \ \ v(L)\geq 0.
$$
According to the maximum principle~\cite[Theorem 8.1]{gilbarg2001elliptic}, we have that $v(x)\geq 0$ for all $x\in [0,L]$. We deduce that, for any $x\in [0,L]$, $\dis u^\eps(x) \leq \frac{\alpha_2}{|b|}\int_0^x \frac{dy}{A^\eps(y)} \leq \frac{\alpha_2 L}{\alpha_1 |b|}$. Taking $u^\eps$ as a test function in~\eqref{pb_helene} and using that $b$ is constant, we obtain~$\dis \int_\Omega A^\eps (u^\eps)'(u^\eps)'=\int_\Omega u^\eps$. Using the Cauchy-Schwarz inequality, we obtain $\dis |u^\eps|^2_{H^1(\Omega)}\leq \frac{\alpha_2 L^2}{\alpha_1^2 |b|}$. Hence, for any constant $f$, we have
\begin{align}
\label{eq:helene_2}
|u^\eps|_{H^1(\Omega)} \leq \frac{\sqrt{\alpha_2 L }}{\alpha_1 \sqrt{|b|}} \ \|f\|_{L^2(\Omega)}.
\end{align}

\medskip

\noindent {\bf Step 3.} 
For a general right-hand side $f$, we write $f=f_1 + f_2$ where $f_1$ is constant and the mean of $f_2$ vanishes. In view of~\eqref{eq:helene_1} and~\eqref{eq:helene_2}, we see that
$$
|u^\eps|_{H^1(\Omega)} \leq \frac{\sqrt{\alpha_2 L }}{\alpha_1 \sqrt{|b|}} \left( \|f_1 \|_{L^2(\Omega)} + \|f_2 \|_{L^2(\Omega)} \right).
$$
We observe that $\|f_1\|_{L^2(\Omega)}^2+\|f_2\|^2_{L^2(\Omega)}=\|f\|^2_{L^2(\Omega)}$, due to the fact $f_1$ is constant and $f_2$ has zero mean. Hence, we have that $\|f_1\|_{L^2(\Omega)}+\|f_2\|_{L^2(\Omega)}\leq \sqrt{2}\|f\|_{L^2(\Omega)}$, and we thus deduce that
$$
|u^\eps|_{H^1(\Omega)} \leq \frac{\sqrt{2\alpha_2 L }}{\alpha_1 \sqrt{|b|}} \ \|f \|_{L^2(\Omega)}.
$$
This concludes the proof of Proposition~\ref{bound_HR}.
\end{proof}

\subsection{The \MsFEMB{} method}
\label{prelim_msfem}

In the convection-dominated regime, the error bound of the \MsFEMB{} method, introduced in Section~\ref{section_msfemB}, is given by the following theorem.

\begin{theorem}
Let $u^\eps$ be the solution to the one-dimensional problem~\eqref{pb_helene} and $u^\eps_H \in \VB$ be its approximation by the \MsFEMB{} method. Assume that $\dis \frac{|b|L}{\alpha_2} \geq 1$. Then the following estimate holds:
\begin{equation}
\label{eq:resu_theo3}
|u^\eps-u^\eps_H|_{H^1(\Omega)} \leq H \left( \sqrt{\frac{\alpha_2}{\alpha_1}} + \frac{|b|H}{\alpha_1}\right) \left( 1 + \frac{\sqrt{2\alpha_2L|b|}}{\alpha_1}\right)\frac{\|f\|_{L^2(\Omega)}}{\alpha_1}.
\end{equation}
\label{thm_HR_B}
\end{theorem}

\begin{proof}
We follow the proof of~\cite{HR2013}. The error $u^\eps-u^\eps_H$ is decomposed in two parts:
$$
e^I=u^\eps-R_{H} u^\eps,
\qquad 
e^I_H=R_{H} u^\eps -u^\eps_{H},
$$
where $R_H u^\eps$ is the interpolant of $u^\eps$ in $\VB$. We have that $\dis -\frac{d}{d x}\left(A^\eps \frac{d (R_H u^\eps)}{d x} \right)=0$ in each mesh element $\mathbf{K} \in \mathcal{T}_H$ and $e_I \in H^1_0(\mathbf{K})$ for all $\mathbf{K} \in \mathcal{T}_H$. Using the variational formulation of~\eqref{pb_helene}, we get
\begin{align}
\alpha_1|e^I|_{H^1(\Omega)}^2
\leq 
a^\eps(e^I,e^I)
=
\sum_{\mathbf{K}\in\mathcal{T}_H} \bigg({\cal L}^\eps e^I,e^I\bigg)_\mathbf{K}
=
\sum_{\mathbf{K}\in\mathcal{T}_H} \bigg(f-b(R_{H} u^\eps)',e^I\bigg)_\mathbf{K},
\label{ineq_17}
\end{align}
where $\dis {\cal L}^\eps v = -\frac{d}{d x}\left(A^\eps \frac{dv}{d x} \right) + b \frac{dv}{d x}$. Now, since $e^I \in H^1_0(\mathbf{K})$, we have
\begin{align}
0
=
\bigg((e^I)',e^I\bigg)_\mathbf{K}
=
\bigg((u^\eps)',e^I\bigg)_\mathbf{K} - \bigg((R_H u^\eps)',e^I\bigg)_\mathbf{K}. 
\label{ineq_18}
\end{align}
Using that $b$ is constant, we deduce from~\eqref{ineq_17} and~\eqref{ineq_18} that
\begin{align}
\alpha_1|e^I|_{H^1(\Omega)}^2&\leq \sum_{\mathbf{K}\in\mathcal{T}_H} \|f-b(u^\eps)'\|_{L^2(\mathbf{K})}\|e^I\|_{L^2(\mathbf{K})}\nonumber\\
&\leq \sum_{\mathbf{K}\in\mathcal{T}_H} H\|f-b(u^\eps)'\|_{L^2(\mathbf{K})}|e^I|_{H^1(\mathbf{K})}\nonumber\\
&\leq H\left(\|f\|_{L^2(\Omega)}+|b| \ |u^\eps|_{H^1(\Omega)}\right)|e^I|_{H^1(\Omega)}\nonumber\\
&\leq H\left(1+\frac{\sqrt{2\alpha_2L|b|}}{\alpha_1}\right) \ \|f\|_{L^2(\Omega)} \ |e^I|_{H^1(\Omega)},
\label{ineq_8}
\end{align}
successively using the Poincar\'e inequality and Proposition~\ref{bound_HR}. 

On the other hand, using Proposition~\ref{a_cea_lemma} and the Poincar\'e inequality, we have 
\begin{align}
|u^\eps-u^\eps_H|_{H^1(\Omega)}
\leq
\sqrt{\frac{\alpha_2}{\alpha_1}} |e^I|_{H^1(\Omega)} +\frac{|b|}{\alpha_1} \|e^I\|_{L^2(\Omega)}
\leq
\left(\sqrt{\frac{\alpha_2}{\alpha_1}}+\frac{|b|H}{\alpha_1}\right)|e^I|_{H^1(\Omega)}.
\label{ineq_9}
\end{align}
Collecting~\eqref{ineq_8} and~\eqref{ineq_9}, we conclude the proof of Theorem~\ref{thm_HR_B}.
\end{proof}

\begin{remark}
Note that the estimate in Theorem~\ref{thm_HR_B} does not depend on the oscillation scale~$\eps$ of $A^\eps$, but only on the contrast~$\alpha_2/\alpha_1$.
\end{remark}

\begin{remark}
\label{rem:fact}
Assume that $A^\eps(x)=\alpha$. Then the \MsFEMB{} method reduces to the classical $\mathbb{P}^1$ method and the estimate~\eqref{eq:resu_theo3} then reads as
\begin{equation}
\label{eq:toto1}
|u-u_H|_{H^1(\Omega)} \leq H \left( 1 + \frac{|b|H}{\alpha}\right) \left( 1 + \sqrt{\frac{2L|b|}{\alpha}}\right) \frac{\|f\|_{L^2(\Omega)}}{\alpha}.
\end{equation}
On the other hand, the classical numerical analysis result for that problem has been recalled in~\eqref{apriori_galerkin_single_scale}. It is 
%\begin{equation}
%\label{eq:toto2}
$\dis |u-u_H|_{H^1(\Omega)} \leq CH \left( 1+\frac{|b|H}{2\alpha} \right) \, |u|_{H^2(\Omega)}$. 
Since $-\alpha u'' + b u' = f$, $|u|_{H^2(\Omega)}$ may be bounded, using Proposition~\ref{bound_HR}, as 
%\begin{equation}
%\label{eq:toto3}
$\alpha |u|_{H^2(\Omega)} 
\leq 
\|f\|_{L^2(\Omega)} + |b| \ |u|_{H^1(\Omega)} 
\leq 
\|f\|_{L^2(\Omega)} + \sqrt{2L|b|/\alpha} \ \|f\|_{L^2(\Omega)}$.
We therefore obtain
$$
|u-u_H|_{H^1(\Omega)} \leq CH \left( 1+\frac{|b|H}{2\alpha} \right) \, \left( 1 + \sqrt{\frac{2L|b|}{\alpha}} \right) \frac{\|f\|_{L^2(\Omega)}}{\alpha},
$$
which exactly coincides, up to constants independent of $b$, $\alpha$, $H$ and $f$, with~\eqref{eq:toto1}.
\end{remark}

\subsection{The \MsFEMBSUPG{} method}
\label{section_msfemB_SUPG_1D} 

For the \MsFEMBSUPG{} method, also introduced in Section~\ref{section_msfemB}, we successively consider two cases. We first consider the "ideal" approach employing the {\em exact} multiscale basis functions, solution to~\eqref{local_pb_msfem_general}. Next, we account for the discretization error when numerically solving the local problem~\eqref{local_pb_msfem_general}. 

\medskip

When the discretization error is ignored, the error estimate is the following. 
\begin{theorem}
Let $u^\eps$ be the solution to the one-dimensional problem~\eqref{pb_helene} and $u^\eps_H\in \VB$ be the solution to~\eqref{varf_msfemB_SUPG_helene}-\eqref{supg_ideal} with $\dis \tau_{\mathbf{K}} = \frac{H}{2|b|}$. Assume that $\dis\frac{|b|L}{\alpha_2} \geq 1$, and that we are in a convection-dominated regime, and hence that $\dis\frac{|b|H}{2\alpha_1} \geq 1$. Then the following estimate holds:
\begin{equation}
\label{eq:resu_theo4}
|u^\eps-u^\eps_H|_{H^1(\Omega)}\leq CH\left(1+\sqrt{\frac{\alpha_2^2}{\alpha_1^2}+\frac{|b|H}{\alpha_1}}\right) \left(1+\frac{\sqrt{2\alpha_2L|b|}}{\alpha_1}\right)\frac{\|f\|_{L^2(\Omega)}}{\alpha_1},
\end{equation}
where $C$ is a universal constant.
\label{thm_HR_Bstab}
\end{theorem}

\begin{remark}
In the case where $A^\eps$ is constant, the \MsFEMBSUPG{} method is simply the $\mathbb{P}^1$ SUPG method. In that case, we observe, as above, that the estimate of Theorem~\ref{thm_HR_Bstab} is similar to the estimate~\eqref{apriori_supg_singlescale} obtained for the $\mathbb{P}^1$ SUPG method.
\end{remark}

Note that the right-hand side of~\eqref{eq:resu_theo4} is thought to be smaller than that of~\eqref{eq:resu_theo3}, as we think of $|b|H/\alpha_1$ as being large. Theorem~\ref{thm_HR_Bstab} is actually established following Steps 2, 3 and 4 of the proof of Theorem~\ref{thm_HRFM_Bstabh} below.

\medskip

Accounting now for the discretization error in the local problems and employing the method~\eqref{varf_msfemB_SUPG}, we now have the following error estimate.

\begin{theorem}
Let $u^\eps$ be the solution to the one-dimensional problem~\eqref{pb_helene} and $u^\eps_{H,h}\in \VBh$ be the solution to~\eqref{varf_msfemB_SUPG} with $\dis \tau_{\mathbf{K}} = \frac{H}{2|b|}$. Assume that $A^\eps \in W^{1,\infty}(\Omega)$ and that $\dis \frac{|b|L}{\alpha_2} \geq 1$. Assume also that we are in a convection-dominated regime, and hence that $\dis\frac{|b|H}{2\alpha_1} \geq 1$. Then the following estimate holds:
\begin{multline}
\label{eq:resu_theo5}
|u^\eps-u^\eps_{H,h}|_{H^1(\Omega)}
\leq 
C \left(1+\frac{H|b|}{\alpha_1} + \frac{H|b|}{\alpha_1} \sqrt{\frac{\alpha_2^2}{\alpha_1^2}+ \frac{|b|H}{\alpha_1}}\right) e_h
\\ + CH\left(1+\sqrt{\frac{\alpha_2^2}{\alpha_1^2}+ \frac{|b|H}{\alpha_1}}\right)\left(1+\frac{\sqrt{2\alpha_2L|b|}}{\alpha_1}\right)\frac{\|f\|_{L^2(\Omega)}}{\alpha_1},
\end{multline}
where $C$ only depends on $\Omega$ and where
\begin{equation}
\label{eq:def_eh}
e_h= h\left(\sqrt{\frac{\alpha_2}{\alpha_1}}+\frac{|b|h}{\alpha_1}\right) \left(1+ \sqrt{\frac{2\alpha_2L}{|b|}} \ \frac{\|(A^\eps)'-b\|_{L^\infty(\Omega)}}{\alpha_1}\right)\frac{\|f\|_{L^2(\Omega)}}{\alpha_1}.
\end{equation}
\label{thm_HRFM_Bstabh}
\end{theorem}

\begin{proof}
This proof is an adaptation of the analysis in~\cite{HR2013}. We proceed as in the proof of~\eqref{apriori_supg_singlescale} (see Appendix~\ref{sec:appA}). We decompose the error $u^\eps-u^\eps_{H,h}$ in three parts:
$$
e^I_h=u^\eps- u^\eps_h,
\qquad
e^I=u^\eps_h-R_{H,h} u^\eps_h,
\qquad
e^I_H=u^\eps_{H,h} -R_{H,h} u^\eps_h,
$$
where $u_h^\eps$ is the Galerkin approximation of $u^\eps$ in $V_h$ (the $\mathbb{P}^1$ finite element space associated to the {\em fine} mesh of size $h$) and $R_{H,h} u^\eps_h$ is the interpolant of $u_h^\eps$ in $V^\eps_{H,h}$. We successively estimate $e^I_h$, $e^I$ and $e^I_H$.

\medskip

\noindent {\bf Step 1: estimation of $e^I_h$.} 
Using Proposition~\ref{a_cea_lemma} and the Poincar\'e inequality, we have
\begin{align}
|e^I_h|_{H^1(\Omega)}
& \leq \sqrt{\frac{\alpha_2}{\alpha_1}}|u^\eps-I_hu^\eps|_{H^1(\Omega)}+\frac{|b|}{\alpha_1}|u^\eps-I_hu^\eps|_{L^2(\Omega)}
\nonumber \\
&\leq
\left(\sqrt{\frac{\alpha_2}{\alpha_1}}+\frac{|b|h}{\alpha_1}\right)|u^\eps-I_hu^\eps|_{H^1(\Omega)},
\label{ineq_1}
\end{align}
where $I_hu^\eps$ is the interpolant of $u^\eps$ in $V_h$. Standard results on finite elements show that
\begin{align}
|u^\eps-I_hu^\eps|_{H^1(\Omega)}\leq Ch|u^\eps|_{H^2(\Omega)}.
\label{ineq_2}
\end{align}
Because of the equation, we have
\begin{align}
|u^\eps|_{H^2(\Omega)}&\leq \frac{\|f\|_{L^2(\Omega)}+\|(A^\eps)'-b\|_{L^\infty(\Omega)}|u^\eps|_{H^1(\Omega)}}{\alpha_1} \nonumber\\ 
&\leq\left(1+\sqrt{\frac{2\alpha_2L}{|b|}} \ \frac{\|(A^\eps)'-b\|_{L^\infty(\Omega)}}{\alpha_1}\right)\frac{\|f\|_{L^2(\Omega)}}{\alpha_1}, 
\label{ineq_3}
\end{align}
where we have used Proposition~\ref{bound_HR}. Collecting~\eqref{ineq_1}, \eqref{ineq_2} and~\eqref{ineq_3}, we obtain
\begin{align}
|e^I_h|_{H^1(\Omega)} \leq C e_h,
\label{ineq_4}
\end{align}
where $e_h$ is defined by~\eqref{eq:def_eh}.

\medskip

\noindent {\bf Step 2: estimation of $e^I$.} 
Using the coercivity of $A^\eps$, we get
\begin{equation}
\label{eq:ineq_45}
\alpha_1| e^I|^2_{H^1(\Omega)} 
\leq
\int_\Omega A^\eps\left(e^I\right)'\left(e^I\right)'
= 
\int_\Omega A^\eps (u^\eps_h)'\left(e^I\right)' - \int_\Omega A^\eps (R_{H,h} u^\eps_h)' \left(e^I\right)'.
\end{equation}
Using that $e^I$ vanishes on the macroscopic mesh nodes and the variational formulation of the basis functions $\psi^{\eps ,h}_i$ of $\VBh$ on $\mathbf{K}$, we observe that
$$
\int_\Omega A^\eps (R_{H,h} u^\eps_h)' \left(e^I\right)'
=
\sum_{\mathbf{K}\in\mathcal{T}_H} \int_{\mathbf{K}} A^\eps (R_{H,h} u^\eps_h)' \left(e^I\right)'
=
0.
$$
We thus deduce from~\eqref{eq:ineq_45} and the variational formulation satisfied by $u^\eps_h$ that
$$
\alpha_1| e^I|^2_{H^1(\Omega)} 
\leq
\int_\Omega A^\eps (u^\eps_h)'\left(e^I\right)' 
=
\int_\Omega \Big( f-b(u^\eps_h)' \Big) e^I
=
\int_\Omega \Big( f-b(u^\eps)' + b(e^I_h)' \Big) e^I.
$$
Using a Poincar\'e inequality for $e^I \in H^1_0(\mathbf{K})$ and Proposition~\ref{bound_HR}, we deduce that
\begin{eqnarray*}
\alpha_1| e^I|^2_{H^1(\Omega)} 
&\leq& 
H \Big( \|f-b(u^\eps)'\|_{L^2(\Omega)} +|b| \ |e^I_h|_{H^1(\Omega)} \Big) |e^I|_{H^1(\Omega)}
\\
&\leq& 
H \left[ \left(1+\frac{\sqrt{2\alpha_2L|b|}}{\alpha_1}\right)\|f\|_{L^2(\Omega)} +|b| \ |e^I_h|_{H^1(\Omega)} \right] |e^I|_{H^1(\Omega)}
\end{eqnarray*}
and thus
\begin{align}
|e^I|_{H^1(\Omega)} \leq H\left(1+\frac{\sqrt{2\alpha_2L|b|}}{\alpha_1}\right)\frac{\|f\|_{L^2(\Omega)}}{\alpha_1} + \frac{H|b|}{\alpha_1}|e^I_h|_{H^1(\Omega)}.
\label{ineq_5}
\end{align}

\medskip

\noindent {\bf Step 3: estimation of $e^I_H$.} 
We write
\begin{eqnarray*}
& &\alpha_1|e^I_H|^2_{H^1(\Omega)} +a_{\text{upw}}(e^I_H,e^I_H) %\nonumber
\\
&\leq& a^\eps(e^I_H,e^I_H)+a_{\text{upw}}(e^I_H,e^I_H) %\nonumber
\\
&\leq& a^\eps(u^\eps_{H,h},e^I_H) + a_{\text{upw}}(u^\eps_{H,h},e^I_H) - a^\eps(R_{H,h} u^\eps_h,e^I_H) - a_{\text{upw}}(R_{H,h} u^\eps_h,e^I_H) %\nonumber
\\
&\leq& F(e^I_H) + F_{\text{stab}}(e^I_H) - a^\eps(R_{H,h} u^\eps_h,e^I_H) - a_{\text{upw}}(R_{H,h} u^\eps_h,e^I_H) %\nonumber
\\
&\leq& a^\eps(u^\eps_h,e^I_H) + F_{\text{stab}}(e^I_H) - a^\eps(R_{H,h} u^\eps_h,e^I_H) - a_{\text{upw}}(R_{H,h} u^\eps_h,e^I_H), %\nonumber
\end{eqnarray*}
making use of the variational formulation satisfied by $u^\eps_{H,h}$ and $u^\eps_h$, respectively. Using that $e^I=u^\eps_h-R_{H,h} u^\eps_h$, we next obtain
\begin{eqnarray}
& &\alpha_1|e^I_H|^2_{H^1(\Omega)} +a_{\text{upw}}(e^I_H,e^I_H) \nonumber
\\
&\leq& 
a^\eps(e^I,e^I_H) + F_{\text{stab}}(e^I_H) + a_{\text{upw}}(e^I,e^I_H) - a_{\text{upw}}(u^\eps_h,e^I_H) \nonumber
\\
&=& \int_\Omega \left( A^\eps \left(e^I\right)' \left(e^I_H\right)' - b\left(e^I_H\right)' e^I\right) + \sum_{\mathbf{K}\in\mathcal{T}_H} \bigg(\tau_{\mathbf{K}} f,b (e^I_H)' \bigg)_\mathbf{K} \nonumber
\\
&& + \sum_{\mathbf{K}\in\mathcal{T}_H} \bigg(\tau_{\mathbf{K}}b\left(e^I\right)' ,b\left(e^I_H\right)'\bigg)_\mathbf{K} 
- \sum_{\mathbf{K}\in\mathcal{T}_H} \bigg(\tau_{\mathbf{K}}b(u^ \eps_h)',b\left(e^I_H\right)'\bigg)_\mathbf{K} \nonumber
\\
&=& \int_\Omega \left( A^\eps \left(e^I\right)' \left(e^I_H\right)' - b\left(e^I_H\right)' e^I\right) + \sum_{\mathbf{K}\in\mathcal{T}_H} \bigg(\tau_{\mathbf{K}} \left( f - b(u^ \eps_h)' \right), b (e^I_H)' \bigg)_\mathbf{K} \nonumber
\\
&& + \sum_{\mathbf{K}\in\mathcal{T}_H} \bigg(\tau_{\mathbf{K}}b\left(e^I\right)' ,b\left(e^I_H\right)'\bigg)_\mathbf{K}. 
\label{err_detail_hrfm}
\end{eqnarray}
We now successively estimate each term of the right-hand side of~\eqref{err_detail_hrfm}. For the first part of the first term, we have
$$
\left| \int_\Omega A^\eps \left(e^I\right)' \left(e^I_H\right)' \right|
\leq
\int_\Omega\alpha_2\left|\left(e^I\right)'\left(e^I_H\right)'\right|\leq \frac{\alpha_1}{4}|e^I_H|_{H^1(\Omega)}^2 + \frac{\alpha_2^2}{\alpha_1} |e^I|_{H^1(\Omega)}^2.
$$
For the second part of the first term, we obtain
\begin{eqnarray*}
- \int_\Omega b\left(e^I_H\right)' e^I 
& \leq &
\frac{1}{4}\sum_{\mathbf{K}\in\mathcal{T}_H}\|\tau_\mathbf{K}^{1/2}b\left(e^I_H\right)'\|_{L^2(\mathbf{K})}^2+ \sum_{\mathbf{K}\in\mathcal{T}_H}\|\tau_\mathbf{K}^{-1/2}e^I\|_{L^2(\mathbf{K})}^2\\
&\leq& 
\frac{1}{4}\sum_{\mathbf{K}\in\mathcal{T}_H}\|\tau_\mathbf{K}^{1/2}b\left(e^I_H\right)'\|_{L^2(\mathbf{K})}^2 + \sum_{\mathbf{K}\in\mathcal{T}_H}\frac{2|b|}{H} H^2|e^I|_{H^1(\mathbf{K})}^2\\
&\leq &
\frac{1}{4}\sum_{\mathbf{K}\in\mathcal{T}_H}\|\tau_\mathbf{K}^{1/2}b\left(e^I_H\right)'\|_{L^2(\mathbf{K})}^2+ 2|b|H|e^I|_{H^1(\Omega)}^2,
\end{eqnarray*}
where, in the second line, we have used the value of $\tau_\mathbf{K}$ and a Poincar\'e inequality.

We bound the second term as follows:
\begin{eqnarray*}
&& \sum_{\mathbf{K}\in\mathcal{T}_H} \bigg(\tau_{\mathbf{K}} \left( f - b(u^\eps_h)' \right),b (e^I_H)' \bigg)_\mathbf{K}
\\
& \leq &
\sum_{\mathbf{K}\in\mathcal{T}_H} \frac{1}{2} \|\tau_\mathbf{K}^{1/2} \left( f - b(u^\eps_h)' \right) \|_{L^2(\mathbf{K})}^2 + \sum_{\mathbf{K}\in\mathcal{T}_H} \frac{1}{2} \|\tau_\mathbf{K}^{1/2} b\left(e^I_H\right)'\|_{L^2(\mathbf{K})}^2
\\ 
&\leq& 
\frac{H}{4|b|}\|f-b(u^\eps_h)'\|_{L^2(\Omega)}^2 + \sum_{\mathbf{K}\in\mathcal{T}_H}\frac{1}{2}\|\tau_\mathbf{K}^{1/2}b\left(e^I_H\right)'\|_{L^2(\mathbf{K})}^2
\\ 
&\leq& 
\frac{H}{2|b|} \left( \|f\|_{L^2(\Omega)}^2 + \|b(u^\eps_h)'\|_{L^2(\Omega)}^2 \right) + \sum_{\mathbf{K}\in\mathcal{T}_H}\frac{1}{2}\|\tau_\mathbf{K}^{1/2}b\left(e^I_H\right)'\|_{L^2(\mathbf{K})}^2
\\
&\leq& 
\frac{H^2}{4\alpha_1} \left( \|f\|_{L^2(\Omega)}^2 + \|b(u^\eps_h)'\|_{L^2(\Omega)}^2 \right) + \sum_{\mathbf{K}\in\mathcal{T}_H}\frac{1}{2}\|\tau_\mathbf{K}^{1/2}b\left(e^I_H\right)'\|_{L^2(\mathbf{K})}^2,
\end{eqnarray*}
where we have used the fact that $\dis\frac{|b|H}{2\alpha_1}\geq 1$ in the last line.

For the third term, we get, using the expression of $\tau_\mathbf{K}$,
\begin{eqnarray*}
&& \sum_{\mathbf{K}\in\mathcal{T}_H} \bigg(\tau_{\mathbf{\mathbf{K}}}b\left(e^I\right)',b \left(e^I_H\right)'\bigg)_\mathbf{K}
\\
&\leq& 
\sum_{\mathbf{K}\in\mathcal{T}_H}\|\tau_\mathbf{K}^{1/2}b\left(e^I\right)'\|_{L^2(\mathbf{K})}^2 + \sum_{\mathbf{K}\in\mathcal{T}_H} \frac{1}{4} \|\tau_\mathbf{K}^{1/2}b\left(e^I_H\right)'\|_{L^2(\mathbf{K})}^2
\\
&\leq& 
\frac{|b| H}{2}|e^I|_{H^1(\Omega)}^2 +\sum_{\mathbf{K}\in\mathcal{T}_H}\frac{1}{4}\|\tau_\mathbf{K}^{1/2}b\left(e^I_H\right)'\|_{L^2(\mathbf{K})}^2.
\end{eqnarray*}
Collecting the terms, we deduce from~\eqref{err_detail_hrfm} that
\begin{multline*}
\alpha_1|e^I_H|_{H^1(\Omega)}^2 
\leq 
\frac{\alpha_1}{4}|e^I_H|_{H^1(\Omega)}^2+\left(\frac{\alpha_2^2}{\alpha_1}+ 2|b|H + \frac{|b|H}{2}\right) |e^I|^2_{H^1(\Omega)} 
\\
+ \frac{H^2}{4\alpha_1} \| f \|^2_{L^2(\Omega)} + \frac{H^2|b|^2}{4\alpha_1}|u^\eps_h|^ 2_{H^1(\Omega)},
\end{multline*}
which yields 
\begin{align}
|e^I_H|_{H^1(\Omega)}\leq C \left[ \left(\frac{\alpha_2^2}{\alpha_1^2}+ \frac{|b|H}{\alpha_1}\right) |e^I|^2_{H^1(\Omega)} + \frac{H^2}{\alpha_1^2} \| f \|^2_{L^2(\Omega)} + \frac{H^2|b|^2}{\alpha_1^2}|u^\eps_h|^2_{H^1(\Omega)} \right]^{1/2},
\label{ineq_6}
\end{align}
where $C$ is a universal constant. Using Proposition~\ref{bound_HR}, we have 
$$
|u^\eps_h|_{H^1(\Omega)}
\leq 
|u^\eps|_{H^1(\Omega)}+|e^I_h|_{H^1(\Omega)}
\leq 
\frac{\sqrt{2\alpha_2L}}{\alpha_1\sqrt{|b|}} \ \|f\|_{L^2(\Omega)}+ |e^I_h|_{H^1(\Omega)},
%\label{ineq_7}
$$
and we thus deduce from~\eqref{ineq_6} that
\begin{multline}
|e^I_H|_{H^1(\Omega)} \leq C \left[ \sqrt{\frac{\alpha_2^2}{\alpha_1^2} + \frac{|b|H}{\alpha_1}} \ |e^I|_{H^1(\Omega)} \right. \\
\left. + H \left( 1 + \frac{\sqrt{2\alpha_2L|b|}}{\alpha_1} \right) \frac{\| f \|_{L^2(\Omega)}}{\alpha_1} + \frac{H|b|}{\alpha_1}|e^I_h|_{H^1(\Omega)} \right].
\label{ineq_6_bis}
\end{multline}

\medskip

\noindent {\bf Step 4: conclusion.} 
Successively using the triangle inequality,~\eqref{ineq_6_bis}, \eqref{ineq_5} and~\eqref{ineq_4}, we obtain
\begin{eqnarray*}
&&|u^\eps-u^\eps_{H,h}|_{H^1(\Omega)}
\\
&\leq& 
|e^I_H|_{H^1(\Omega)}+|e^I|_{H^1(\Omega)}+|e^I_h|_{H^1(\Omega)}
\\
&\leq& 
\left(1+C\sqrt{\frac{\alpha_2^2}{\alpha_1^2}+ \frac{|b|H}{\alpha_1}}\right)|e^I|_{H^1(\Omega)} + \left(1+\frac{CH|b|}{\alpha_1}\right)|e^I_h|_{H^1(\Omega)} 
\\
&& \qquad + C H \left( 1 + \frac{\sqrt{2\alpha_2L|b|}}{\alpha_1} \right) \frac{\|f\|_{L^2(\Omega)}}{\alpha_1}
\\
&\leq& 
\left[1+\frac{CH|b|}{\alpha_1}+\frac{H|b|}{\alpha_1} \left( 1 + C \sqrt{\frac{\alpha_2^2}{\alpha_1^2}+ \frac{|b|H}{\alpha_1}} \right)\right]|e^I_h|_{H^1(\Omega)}
\\
&& + \left(1+C\sqrt{\frac{\alpha_2^2}{\alpha_1^2}+ \frac{|b|H}{\alpha_1}}\right) H\left(1+\frac{\sqrt{2\alpha_2L|b|}}{\alpha_1}\right)\frac{\|f\|_{L^2(\Omega)}}{\alpha_1}
\\
&& \qquad + C H \left( 1 + \frac{\sqrt{2\alpha_2L|b|}}{\alpha_1} \right) \frac{\|f\|_{L^2(\Omega)}}{\alpha_1}
\\
&\leq& 
C \left[1+\frac{H|b|}{\alpha_1}+\frac{H|b|}{\alpha_1} \sqrt{\frac{\alpha_2^2}{\alpha_1^2}+ \frac{|b|H}{\alpha_1}} \right] e_h
\\
& &
+ C H \left(1+\sqrt{\frac{\alpha_2^2}{\alpha_1^2}+ \frac{|b|H}{\alpha_1}}\right) \left(1+\frac{\sqrt{2\alpha_2L|b|}}{\alpha_1}\right)\frac{\|f\|_{L^2(\Omega)}}{\alpha_1}.
\end{eqnarray*}
This concludes the proof of Theorem~\ref{thm_HRFM_Bstabh}.
\end{proof}

\subsection{The \MsFEMA{} method}
\label{prelim_msfem-adv}

The error bound of the \MsFEMA{} method (introduced in Section~\ref{section_msfemA}) is given by the following theorem.

\begin{theorem}
Let $u^\eps$ be the solution to the one-dimensional problem~\eqref{pb_helene} and $u^\eps_H \in \VA$ be the solution to the \MsFEMA{} method. The following estimate holds:
$$
|u^\eps-u^\eps_H|_{H^1(\Omega)}\leq H \left(\sqrt{\frac{\alpha_2}{\alpha_1}} + \frac{|b|H}{\alpha_1}\right) \frac{\|f\|_{L^2(\Omega)}}{\alpha_1}.
$$
\label{thm_HR_A}
\end{theorem}

The proof of this theorem follows the same pattern as the proof of Theorem~\ref{thm_HR_B}. We therefore skip it. 

\subsection{Splitting approach}
\label{prelim_splitting}

We now turn to the splitting method introduced in Section~\ref{section_splitting}. In contrast to Sections~\ref{prelim_msfem}, \ref{section_msfemB_SUPG_1D} and~\ref{prelim_msfem-adv}, we do not restrict ourselves to the one-dimensional setting. In what follows, we denote $C_\Omega$ the Poincar\'e constant as defined by $\| \varphi \|_{L^2(\Omega)} \leq C_\Omega | \varphi |_{H^1(\Omega)}$ for any $\varphi \in H^1_0(\Omega)$.

\subsubsection{The method~\eqref{pb_spl_1}--\eqref{pb_spl_2}}

\begin{lemma}
\label{lemma_conv}
Consider the splitting method~\eqref{pb_spl_1}--\eqref{pb_spl_2}. If 
\begin{align}
\frac{C_\Omega \|b\|_{L^\infty(\Omega)}}{\alpha_1}\left(\frac{\| A^\eps-\alpha_{\rm spl} \text{Id}\|_{L^\infty(\Omega)}}{\alpha_{\rm spl}}\right)<1,
\label{naive_cond}
\end{align}
then $u_{2n+1}$ converges in $H^1_0(\Omega)$ to $u^\eps$ solution to~\eqref{pb_multiscale}.
\end{lemma}

\begin{proof}
Let $\widetilde{u}_n=u_{n+2}-u_{n}$. We reformulate the system~\eqref{pb_spl_1}--\eqref{pb_spl_2} as
\begin{align}
&\left\{\begin{aligned}
&-\alpha_{\rm spl}\Delta \widetilde{u}_{2n+2} + b\cdot\nabla \widetilde{u}_{2n+2} = b\cdot\nabla (\widetilde{u}_{2n}-\widetilde{u}_{2n+1}) \quad\text{ in }\Omega, \\
&\widetilde{u}_{2n+2}=0 \quad\text{ on }\partial\Omega ,
\end{aligned}\right.\label{pb_spl_zero_1}
\\
&\left\{\begin{aligned}
&-\text{div }(A^\eps\nabla \widetilde{u}_{2n+3})=-\alpha_{\rm spl}\Delta \widetilde{u}_{2n+2}\quad\text{ in }\Omega,\\
&\widetilde{u}_{2n+3}=0\quad\text{ on }\partial\Omega.
\end{aligned}\right.
\label{pb_spl_zero_2}
\end{align} 
Using the variational formulations of~\eqref{pb_spl_zero_1} and~\eqref{pb_spl_zero_2}, we have
\begin{align}
&\alpha_{\rm spl} |\widetilde{u}_{2n+2}|_{H^1(\Omega)}\leq C_\Omega \|b\|_{L^\infty(\Omega)}|\widetilde{u}_{2n}-\widetilde{u}_{2n+1}|_{H^1(\Omega)},\label{ineq_12}\\
&\alpha_1|\widetilde{u}_{2n+1}|_{H^1(\Omega)}\leq \alpha_{\rm spl}|\widetilde{u}_{2n}|_{H^1(\Omega)},\label{ineq_13}
\end{align} 
where we have used~\eqref{div_b_zero} and~\eqref{condition_alpha}. Letting $w_n=\widetilde{u}_{2n+1}-\widetilde{u}_{2n}$, we have
$$
-\text{div}(A^\eps\nabla w_n)=-\alpha_{\rm spl} \Delta \widetilde{u}_{2n} + \text{div}(A^\eps \nabla \widetilde{u}_{2n}) \ \ \text{in $\Omega$}, 
\qquad
w_n=0 \ \ \text{on $\partial\Omega$}.
$$
We deduce that 
\begin{align}
\alpha_1|w_n|_{H^1(\Omega)} \leq \| A^\eps-\alpha_{\rm spl} \text{Id} \|_{L^\infty(\Omega)} |\widetilde{u}_{2n}|_{H^1(\Omega)}.
\label{ineq_14}
\end{align}
Collecting~\eqref{ineq_12} and~\eqref{ineq_14}, we get
$$
|\widetilde{u}_{2n+2}|_{H^1(\Omega)} \leq \rho^{1+n}|\widetilde{u}_{0}|_{H^1(\Omega)},
$$
where 
$$
\rho = \frac{C_\Omega \|b\|_{L^\infty(\Omega)}}{\alpha_1}\left(\frac{\| A^\eps-\alpha_{\rm spl} \text{Id} \|_{L^\infty(\Omega)}}{\alpha_{\rm spl}}\right).
$$
Because of~\eqref{naive_cond}, the sequence $u_{2n}$ therefore converges in $H^1_0(\Omega)$ to some $u_{\text{even}}$. In view of~\eqref{ineq_13}, the sequence $u_{2n+1}$ also converges in $H^1_0(\Omega)$ to some $u_{\text{odd}}$. Passing to the limit $n \to \infty$ in~\eqref{pb_spl_1} and~\eqref{pb_spl_2}, we obtain that $u_{\text{even}}$ and $u_{\text{odd}}$ are the solutions to
\begin{align}
-\alpha_{\rm spl}\Delta u_{\text{even}} = f - b\cdot\nabla u_{\text{odd}} \quad\text{ in }\Omega, 
\quad
u_{\text{even}}=0 \quad\text{ on }\partial\Omega,
\label{pb_spl_lim_1d_1}
\\
-\text{div}(A^\eps\nabla u_{\text{odd}})=-\alpha_{\rm spl}\Delta u_{\text{even}}\quad\text{ in }\Omega,
\quad
u_{\text{odd}}=0\quad\text{ on }\partial\Omega.
\label{pb_spl_lim_1d_2}
\end{align}
Adding~\eqref{pb_spl_lim_1d_1} and~\eqref{pb_spl_lim_1d_2}, we get that $u_{\text{odd}}$ is actually the solution to~\eqref{pb_multiscale}.
\end{proof}

There are unfortunately simple situations where~\eqref{naive_cond} is not satisfied, whatever the choice of $\alpha_{\rm spl}$. Consider for instance the one-dimensional setting where $A^\eps$ is continuous. Then~$\| A^\eps-\alpha_{\rm spl} \text{Id} \|_{L^\infty(\Omega)} = \max(|a_+-\alpha_{\rm spl}|,|a_--\alpha_{\rm spl}|)$ where $\dis a_- = \inf_\Omega A^\eps$ and $\dis a_+ = \sup_\Omega A^\eps$. We observe that
\begin{align}
\rho\geq \rho_-,
\label{inf_rho}
\end{align}
where~$\dis \rho_- =\frac{C_\Omega \|b\|_{L^\infty(\Omega)}}{\alpha_1}\frac{a_+-a_-}{a_++a_-}$. If $\rho_- > 1$, then, for any $\alpha_{\rm spl}>0$, condition~\eqref{naive_cond} is not satisfied. Of course,~\eqref{naive_cond} is only a sufficient, and not a necessary condition for the convergence of the iterations. In most cases, and even in some cases when~\eqref{naive_cond} is not satisfied, the splitting method~\eqref{pb_spl_1}--\eqref{pb_spl_2} converges, see Section~\ref{section_accuracies}. In some cases, it does not. Lemma~\ref{example_explosion} below describes such a convergence failure, for a one-dimensional example that can be easily extended to higher dimensional settings using tensor products.

\begin{lemma}
\label{example_explosion}
Assume that $\Omega=(0,1)$, that~$A^\eps\equiv \alpha^\star$, that the initial guess for~\eqref{pb_spl_1}--\eqref{pb_spl_2} is $u_0=\cos(2\pi x)-1$ and $u_1$ the solution to~\eqref{pb_spl_2} with $u_0$ in the right-hand side. 
%Assume also that $\dis \frac{b}{\alpha^\star}$ is sufficiently large. 
Take $\alpha^\star$ and $\alpha_{\rm spl}$ such that
\begin{align}
\frac{b}{\alpha_{\rm spl}}<\frac{b}{2\alpha^\star}-2\pi^2\frac{\alpha^\star}{b}.
\label{conv_div_spl}
\end{align}
Then the sequences $(u_{2n})_{n \in \N}$ and $(u_{2n+1})_{n \in \N}$ do not converge in $H^1_0(\Omega)$.
\end{lemma}

\begin{proof}
We may assume without loss of generality that $f=0$. In this particular setting, equation~\eqref{pb_spl_2} reads as $u_{2n+1}=(\alpha_{\rm spl}/\alpha^\star) \, u_{2n}$, so~\eqref{pb_spl_1} reduces to
$$
-(u_{2n+2})''+\frac{b}{\alpha_{\rm spl}}(u_{2n+2})' = \lambda (u_{2n})' \ \text{in $(0,1)$}, \quad u_{2n+2}(0)=u_{2n+2}(1)=0,
$$
where $\dis \lambda=\frac{b}{\alpha_{\rm spl}}\left(1-\frac{\alpha_{\rm spl}}{\alpha^\star}\right)$. A simple calculation shows that, for any $n\in\N$, $(u_{2n})'= c_n\cos(2\pi x) + s_n\sin(2\pi x)$, with $[c_n,s_n]^T=(-1)^n\lambda^n A^n[0,-2\pi]^T$ and
$$
A=\left[ \left(\frac{b}{\alpha_{\rm spl}}\right)^2+4\pi^2 \right]^{-1} \begin{pmatrix} -b/\alpha_{\rm spl} & -2\pi\\
2\pi & -b/\alpha_{\rm spl}
\end{pmatrix}.
$$
If $\rho(\lambda A)= \frac{|\lambda|}{\sqrt{\left(b/\alpha_{\rm spl}\right)^2+4\pi^2}}>1$, a condition which is equivalent to~\eqref{conv_div_spl}, then the sequence $(u_{2n})_{n \in \N}$ does not converge.
\end{proof}

\subsubsection{An alternate splitting method}

We now present an alternate splitting method, which includes some element of damping, and which, when the damping parameter (denoted by~$\beta$) is suitably adjusted, unconditionally converges. We emphasize however that we have observed in our numerical tests that the convergence of this alternate approach, although guaranteed theoretically, is much slower than that of the method~\eqref{pb_spl_1}--\eqref{pb_spl_2}. See Figure~\ref{test_10_ite_spl} below.

\medskip

The iterates $u_{2n+2}$ and $u_{2n+3}$ are now defined by
\begin{align}
&\left\{\begin{aligned}
&-(\beta + \alpha_{\rm spl})\Delta u_{2n+2} + b\cdot\nabla u_{2n+2} = f + b\cdot\nabla (u_{2n}-u_{2n+1})-\beta \Delta u_{2n+1}\quad\text{ in }\Omega, \\
&u_{2n+2}=0 \quad\text{ on }\partial\Omega ,
\end{aligned}\right.
\label{pb_spl_beta_1}
\\
&\left\{\begin{aligned}
&-\text{div }((\beta \text{Id}+A^\eps)\nabla u_{2n+3})=-(\beta+\alpha_{\rm spl})\Delta u_{2n+2}\quad\text{ in }\Omega,\\
&u_{2n+3}=0\quad\text{ on }\partial\Omega,
\end{aligned}\right.
%\nonumber
\label{pb_spl_beta_2}
\end{align} 
with $\alpha_{\rm spl}>0$ and $\beta\geq 0$. Of course, $\beta=0$ yields~\eqref{pb_spl_1}--\eqref{pb_spl_2}.

The convergence of~\eqref{pb_spl_beta_1}--\eqref{pb_spl_beta_2} is established in the following lemma, in the infinite dimensional setting. The discretized, finite dimensional version will be studied next.

\begin{lemma}
\label{lemma_conv_beta}
Choose 
\begin{align}
\beta = \underset{x\geq 0}{\arg\!\min} \left(\frac{C_\Omega \|b\|_{L^\infty(\Omega)}}{x+\alpha_{\rm spl}}\frac{\| A^\eps-\alpha_{\rm spl} \text{Id}\|_{L^\infty(\Omega)}}{x+\alpha_1}+\frac{x}{x+\alpha_{1}}\right),
\label{beta_choice}
\end{align}
where $\alpha_1$ is such that~\eqref{condition_alpha} holds. Then $u_{2n+1}$ converges in $H^1_0(\Omega)$ to $u^\eps$ solution to~\eqref{pb_multiscale}.
\end{lemma}

\begin{proof}
Following the arguments of the proof of Lemma~\ref{lemma_conv}, we have
\begin{align}
|\widetilde{u}_{2n+2}|_{H^1(\Omega)}&\leq\frac{C_\Omega \|b\|_{L^\infty(\Omega)}}{\beta+\alpha_{\rm spl}}|\widetilde{u}_{2n}-\widetilde{u}_{2n+1}|_{H^1(\Omega)} + \frac{\beta}{\beta+\alpha_{\rm spl}}|\widetilde{u}_{2n+1}|_{H^1(\Omega)},\label{ineq_19} \\
|\widetilde{u}_{2n+1}|_{H^1(\Omega)}&\leq \frac{\beta + \alpha_{\rm spl}}{\beta + \alpha_1}|\widetilde{u}_{2n}|_{H^1(\Omega)},\label{ineq_20} \\
|w_n|_{H^1(\Omega)}&\leq \frac{\| A^\eps-\alpha_{\rm spl} \text{Id}\|_{L^\infty(\Omega)}}{\beta+\alpha_1} |\widetilde{u}_{2n}|_{H^1(\Omega)},
\label{ineq_21}
\end{align}
where we recall that $\widetilde{u}_n = u_{n+2} - u_n$ and $w_n = \widetilde{u}_{2n+1} - \widetilde{u}_{2n}$.

Collecting~\eqref{ineq_19}, \eqref{ineq_20} and~\eqref{ineq_21}, we have 
$$
%\begin{align}
|\widetilde{u}_{2n+2}|_{H^1(\Omega)}\leq \rho|\widetilde{u}_{2n}|_{H^1(\Omega)},
%\label{ineq_15}
%\end{align} 
$$
where $\rho = g(\beta)$ and where the function $g$ is defined by
$$
g(x) = \frac{C_\Omega \|b\|_{L^\infty(\Omega)}}{x+\alpha_{\rm spl}}\frac{\| A^\eps-\alpha_{\rm spl}\text{Id}\|_{L^\infty(\Omega)}}{x+\alpha_1}+\frac{x}{x+\alpha_{1}}. 
$$
We next observe that $\dis g(x) = 1 - \frac{\alpha_1}{x} + O\left(\frac{1}{x^2}\right)$. Since $\alpha_1>0$, this implies that $\dis \min_{x \geq 0} g(x) < 1$. In view of~\eqref{beta_choice}, we have $\dis \rho = g(\beta) = \min_{x \geq 0} g(x) < 1$. We next conclude the proof mimicking the argument in the proof of Lemma~\ref{lemma_conv}.
\end{proof}

\medskip

We now consider the discrete case. Given the approximations~$u_{2n}^H$ and~$u_{2n+1}^H$, we define $u_{2n+2}^H$ and $u_{2n+3}^H$ as follows. First, we discretize~\eqref{pb_spl_beta_1} on a coarse mesh and use the SUPG terms to stabilize the approach. We hence define $u_{2n+2}^H$ by the following variational formulation:
\begin{align}
&\text{Find $u^H_{2n+2}\in V_H$ such that, for any $v\in V_H$,}\nonumber\\
&a_1(u_{2n+2}^H,v) + a_{\text{conv}}(u_{2n+2}^H,v)
= \widetilde{F}^1(v)+\widetilde{F}_{\text{stab}}(v) + a_{\text{conv}}(P_{\VB}(u_{2n}^H),v),
\label{pb_spl_beta_discret_1_discrete_varf}
\end{align}
where we recall that $V_H$ is the $\mathbb{P}^1$ finite element space, and where
\begin{eqnarray}
a_1(u,v) 
&=& 
\int_\Omega (\beta + \alpha_{\rm spl})\nabla u\cdot\nabla v,
\label{eq:def_a1}
\\
a_{\text{conv}}(u,v) 
&=& 
\int_\Omega (b\cdot\nabla u) \, v + \sum_{\mathbf{K}\in\mathcal{T}_H} (\tau_\mathbf{K} b\cdot \nabla u,b\cdot\nabla v)_{L^2(\mathbf{K})},
\label{eq:def_aconv}
\\
\widetilde{F}^1(v) 
&=& 
\int_\Omega (f - b\cdot \nabla u_{2n+1}^{H})v +\int_\Omega \beta \nabla u_{2n+1}^{H}\cdot\nabla v, 
\nonumber
\\
\widetilde{F}_{\text{stab}}(v) 
&=& 
\sum_{\mathbf{K}\in\mathcal{T}_H} (\tau_\mathbf{K}(f - b\cdot \nabla u_{2n+1}^{H}),b\cdot\nabla v)_{L^2(\mathbf{K})}.
\nonumber
\end{eqnarray}
In~\eqref{pb_spl_beta_discret_1_discrete_varf}, $\dis P_{\VB}$ is the projector on the space $\VB$ defined as follows. For any $v \in H^1_0(\Omega)$, we define $P_{\VB}(v) \in \VB$ by
\begin{equation}
\label{projection_VB_varf}
\forall w \in \VB, \quad a_1 \Big( P_{\VB}(v),w \Big) = a_1(v,w).
\end{equation}

Second, we discretize~\eqref{pb_spl_beta_2} using the \MsFEMB{} approach: we define $u_{2n+3}^H$ by the following variational formulation:
\begin{align}
\text{Find $u^H_{2n+3} \in \VB$ such that, for any $w\in \VB$},
\ \
a_2(u_{2n+3}^H,w) = a_1(u_{2n+2}^H,w),
\label{pb_spl_beta_discret_2_discrete_varf}
\end{align}
where
\begin{equation}
\label{eq:def_a2}
a_2(u,v) = \int_\Omega (\nabla v)^T (\beta \text{Id} + A^\eps)\nabla u.
\end{equation}

\begin{remark}
\label{remark_introduction_projection}
Three remarks on~\eqref{pb_spl_beta_discret_1_discrete_varf} are in order. First, the term $-\beta \Delta u_{2n+1}^H$ is absent from~$\widetilde{F}_{\text{stab}}$ only because, as we use a $\mathbb{P}^1$ approach, that term identically vanishes in each element $\mathbf{K}$. Second, as already mentioned in Section~\ref{section_splitting}, the computation of $\widetilde{F}^1(v)$ needs to be performed on a fine mesh, since~$u_{2n+1}^H$ belongs to the \MsFEMB{} space $\VB$. Third, the introduction of the projector $\dis P_{\VB}$ in~\eqref{pb_spl_beta_discret_1_discrete_varf} is motivated by the need to guarantee the convergence of the iterations~\eqref{pb_spl_beta_discret_1_discrete_varf}--\eqref{pb_spl_beta_discret_2_discrete_varf} to an accurate approximation of the solution $u^\eps$ to the reference problem~\eqref{pb_multiscale}. Lemma~\ref{lem:iter_alternative} below will clarify and establish this convergence. Note that, instead of~\eqref{projection_VB_varf}, we could as well have defined $\dis P_{\VB}(v) \in \VB$, for any $v \in H^1_0(\Omega)$, by the relation $\dis a_2 \Big( P_{\VB}(v),w \Big) = a_2(v,w)$ for any $w \in \VB$.
\end{remark}

\medskip

We establish in Appendix~\ref{sec:appC} below the convergence of~\eqref{pb_spl_beta_discret_1_discrete_varf}--\eqref{pb_spl_beta_discret_2_discrete_varf}. Formally passing to the limit $n \to \infty$ in~\eqref{pb_spl_beta_discret_1_discrete_varf}--\eqref{pb_spl_beta_discret_2_discrete_varf}, we observe that, if $(u_{2n}^H,u_{2n+1}^H)$ converges to some $(u_{\text{even}}^H,u_{\text{odd}}^H) \in V_H\times \VB$, then $(u_{\text{even}}^H,u_{\text{odd}}^H)$ satisfies
\begin{multline}
\forall v\in V_H, \quad
a_1(u^H_{\text{even}},v) + a_{\text{conv}}(u^H_{\text{even}},v)
\\ = \widetilde{F}^1(v;u_{\text{odd}}^H)+\widetilde{F}_{\text{stab}}(v;u_{\text{odd}}^H) + a_{\text{conv}}(P_{\VB}(u^H_{\text{even}}),v),
\label{pb_spl_beta_discret_lim_1_discrete_varf}
\end{multline}
and
\begin{align}
\forall w\in \VB, \quad
a_2(u^H_{\text{odd}},w) = a_1(u^H_{\text{even}},w),
\label{pb_spl_beta_discret_lim_2_discrete_varf}
\end{align}
with
$\dis \widetilde{F}^1(v;u_{\text{odd}}^H) = \int_\Omega (f - b\cdot \nabla u_{\text{odd}}^H)v +\int_\Omega \beta \nabla u_{\text{odd}}^{H}\cdot\nabla v$
and
$\dis \widetilde{F}_{\text{stab}}(v;u_{\text{odd}}^H) = \sum_{\mathbf{K}\in\mathcal{T}_H} (\tau_\mathbf{K}(f - b\cdot \nabla u_{\text{odd}}^H),b\cdot\nabla v)_{L^2(\mathbf{K})}$. This convergence is rigorously stated in Lemma~\ref{lem:iter_alternative} below, as well as the convergence when $H \to 0$.

\begin{lemma}
\label{lem:iter_alternative}
Suppose that we set the stabilization parameter to
$$
\tau_{\mathbf{K}}(x)=\frac{H}{2|b(x)|} \quad\text{for any }\mathbf{K}\in \mathcal{T}_H.
$$
Choose 
\begin{align}
\beta = \underset{x\geq 0}{\arg\!\min} \left[ \left(C_\Omega+\frac{H}{2}\right)\frac{\|b\|_{L^\infty(\Omega)}}{x+\alpha_{\rm spl}}\frac{\| A^\eps-\alpha_{\rm spl}\text{Id}\|_{L^\infty(\Omega)}}{x+\alpha_1}+\frac{x}{x+\alpha_{1}}\right]
\label{beta_choice_discrete}
\end{align}
where $\alpha_1$ is such that~\eqref{condition_alpha} holds.

Then, when $n \to \infty$, $(u_{2n}^H,u_{2n+1}^H)$ converges in $H^1_0(\Omega)\times H^1_0(\Omega)$ to $(u_{\text{even}}^H,u_{\text{odd}}^H) \in V_H\times \VB$ solutions to the variational formulation~\eqref{pb_spl_beta_discret_lim_1_discrete_varf}--\eqref{pb_spl_beta_discret_lim_2_discrete_varf}.

Assume in addition that $A^\eps \in W^{1,\infty}(\Omega)$ and that
\begin{equation}
\label{eq:hyp_spl}
\alpha_{\rm spl} < 2 \alpha_1.
\end{equation}
Then, when $H \to 0$, $u_{\text{odd}}^H$ converges in $H^1_0(\Omega)$ to $u^\eps$ solution to~\eqref{pb_multiscale}.
\end{lemma}

The proof of Lemma~\ref{lem:iter_alternative} is postponed until Appendix~\ref{sec:appC}. 

\section{Numerical simulations}
\label{section_numerical_simulations}

In this section, we present and discuss our numerical experiments. They have all been performed using FreeFem++~\cite{MR3043640}. Our aim is to compare the four approaches of Section~\ref{section_our_four_appr}. Section~\ref{test_case} collects some preliminary material. Then we assess the accuracy and computational cost of our four numerical methods in Sections~\ref{section_accuracies} and~\ref{section_cost}, respectively.

\subsection{Test case} 
\label{test_case}
We work on the domain $\Omega=(0,1)^2$, discretized with a uniform coarse mesh $\mathcal{T}_H$ of size $H$. Let $V_H$ be the finite dimensional vector space~\eqref{VH-space} associated to the classical $\mathbb{P}^1$ discretization. In~\eqref{pb_multiscale}, we set $b=(1,1)^T$, $f=1$ and
$$
A^\eps(x_1,x_2)=\alpha \left(1+\delta\cos\left(\frac{2\pi}{\eps}x_1\right)\right)\text{Id}_2, \quad\text{with $\alpha$, $\delta>0$}.
$$
We recall that the convection-dominated regime is defined by the condition $\text{Pe}\, H>1$, where we define here the global P\'eclet number Pe of problem~\eqref{pb_multiscale} by~\eqref{global-peclet}. Here this regime corresponds to
\begin{align}
\alpha <\frac{H}{2}.
\label{test_conv_dom}
\end{align}
In this regime, the solution exhibits the boundary layer~$\Omega_{\text{layer}} = \Big( (0,1)\times(1-\delta_{\text{layer}},1)\Big) \cup \Big((1-\delta_{\text{layer}},1) \times (0,1)\Big)$, represented on Figure~\ref{omega_layer}, of approximate width $\displaystyle \delta_{\text{layer}}=\frac{1}{\text{Pe}}\log(\text{Pe})$. 

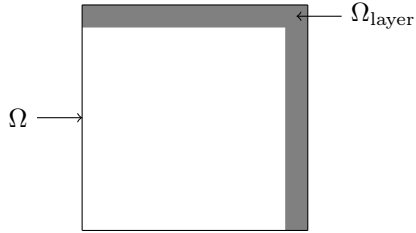
\begin{figure}[htbp]
\begin{center}
\begin{tikzpicture}[scale=3]
\fill [gray] (0.9,0) -- (1,0) -- (1,1) -- (0,1) -- (0,0.9) -- (0.9,0.9) -- cycle;
\draw (0,0) -- (1,0);
\draw (1,0) -- (1,1);
\draw (1,1) -- (0,1);
\draw (0,1) -- (0,0);
%\draw (0,0.9) -- (0.9,0.9);
%\draw (0.9,0) -- (0.9,0.9);
\draw [->] (-0.2,0.5) -- (0,0.5);
\draw (-0.2,0.5) node[left] {$\Omega$};
%\draw (1,1) node[below left] {$\Omega_{\text{layer}}$};
%\draw (0,0) circle (1);
\draw [<-] (0.95,0.95) -- (1.15,0.95);
\draw (1.15,0.95) node[right] {$\Omega_{\text{layer}}$};
\end{tikzpicture}
\end{center}
\caption{The domain $\Omega$ and the boundary layer $\Omega_{\text{layer}}$}
\label{omega_layer}
\end{figure}

We choose for the splitting method the value $\alpha_{\rm spl}=\alpha$. Motivated by the one-dimensional formula~\eqref{formule_tauk_1d}, the stabilization parameter $\tau_\mathbf{K}$ is chosen as
$$
%\begin{align}
\tau_\mathbf{K}(x)=\frac{|\mathbf{K}|}{2|b(x)|}\left(\coth\left(\text{Pe}_\mathbf{K}(x)\right)-\left(\text{Pe}_\mathbf{K}(x)\right)^{-1}\right) \quad\text{ for all }\mathbf{K}\in\mathcal{T}_H,
%\label{formula_tauK}
%\end{align}
$$
where $\dis \text{Pe}_\mathbf{K}(x)= \frac{|b(x)| \, H}{2\alpha}$.

\subsubsection{Evaluation of the accuracy}
\label{section_evaluation_accuracy}

Let $\mathcal{T}_h$ be a uniform fine mesh of $\Omega$ of size $h$ such that $\mathcal{T}_h$ is a refinement of $\mathcal{T}_H$. The reference solution~$u_{\text{ref}}$ is obtained by the standard $\mathbb{P}^1$ finite element discretization on $\mathcal{T}_h$ where $h$ is such that
$$
%\begin{align}
h\leq \frac{1}{16}\min(\eps,\delta_{\text{layer}}) \quad\text{and }\text{Pe}\, h\leq \frac{1}{4\sqrt{2}}<1.
%\label{condition_ref}
%\end{align}
$$
This condition ensures that the fine mesh can both resolve the oscillations throughout the domain at scale $\eps$ of the solution and the details within the boundary layer. It also ensures that, at scale $h$, the problem is not convection-dominated. This fine mesh is also that on which the local problems are solved, in order to determine the MsFEM basis functions.

In the sequel, the accuracies of the methods are compared using the following relative errors: $e_{H^1_{\text{in}}}(u_1)=\frac{\|u_1-u_{\text{ref}}\|_{H^1(\Omega_{\text{layer}})}}{\|u_{\text{ref}}\|_{H^1(\Omega)}}$ inside the boundary layer, and likewise~$e_{H^1_{\text{out}}}(u_1)=\frac{\|u_1-u_{\text{ref}}\|_{H^1(\Omega\setminus\Omega_{\text{layer}})}}{\|u_{\text{ref}}\|_{H^1(\Omega)}}$ outside that layer, and, in the whole domain, $e_{L^p}(u_1)=\frac{\|u_1-u_{\text{ref}}\|_{L^p(\Omega)}}{\|u_{\text{ref}}\|_{L^p(\Omega)}}$ for $p=2$ or $p=\infty$, and $e_{H^1}(u_1)=\frac{\|u_1-u_{\text{ref}}\|_{H^1(\Omega)}}{\|u_{\text{ref}}\|_{H^1(\Omega)}}$. All these relative errors are computed on the fine mesh $\mathcal{T}_h$.

\subsubsection{Evaluation of the computational costs}

The sizes of the local and global problems in the test cases we consider in Section~\ref{section_accuracies} are sufficiently small to allow for the use of direct linear solvers (in our case, the UMFPACK library). This clearly favors the splitting method as opposed to the other approaches, since that method is potentially the most expensive one of all four in its online stage. The factorization of the stiffness matrices is performed once and for all in the offline stage and is repeatedly used in the iterative process during the online stage. When, for problems of larger sizes, iterative linear solvers are in order, the online cost of the splitting method correspondingly increases. To evaluate this marginal cost, we have also performed tests using iterative solvers \emph{as if} the problem sizes were large. We have used either, for non-symmetric matrices, the GMRES solver and a value of the stopping criterion equal to $10^{-11}$, or, for symmetric matrices, the conjugate gradient method with a stopping criterion at $10^{-20}$. Both solvers are used with a simple diagonal preconditioner. The computations have all been performed on a Intel$\textregistered$ Xeon$\textregistered$ Processor E5-2667 v2. The specific function used to measure the CPU time is clock\_gettime() with the clock CLOCK\_PROCESS\_CPUTIME\_ID. 

\subsection{Accuracies}
\label{section_accuracies}

\pgfplotscreateplotcyclelist{mylistmethods}{
%{blue,dashed,mark=*},
{blue,mark=*},{red,mark=square*}, {brown!60!black,every mark/.append style={fill=brown!80!black},mark=otimes*},{dashdotted,mark=star},{mark=diamond*,loosely dashdotted,color=green}}

\subsubsection{Reference test}
\label{an_example}

We first consider problem~\eqref{pb_multiscale}, with the choices of $A^\eps$ and $b$ described in Section~\ref{test_case}, and the parameters $\alpha=1/128$, $\delta=0.5$ and $H=1/16$. Since Pe$\, H=4$, the problem is expected to be convection-dominated and, for~$\eps=1/64$, multiscale. 

In order to practically check that the dominating convection is a challenge to standard approaches, we temporarily set $\eps$ to one, and compare the results obtained by the $\mathbb{P}^1$ method and the $\mathbb{P}^1$ Upwind method~\cite{hughes1979multidimensional}. Table~\ref{tabl_eps_1} shows that, outside the boundary layer, the relative $H^1$ error of the $\mathbb{P}^1$ method is approximately 20 times as large as the error of the $\mathbb{P}^1$ Upwind method. This confirms the convection-dominated regime. 

\begin{table}[htbp]
\begin{center}
\begin{tabular}{l | l l l l l}
	&	$e_{L^2}$	&	\textbf{$e_{H^1}$}	&	\textbf{$e_{L^\infty}$}	&	\textbf{$e_{H^1_{\text{in}}}$}	&	\textbf{$e_{H^1_{\text{out}}}$}	\\ \noalign{\vskip 3pt}
\hline					
$\mathbb{P}^1$	&	0.24	&	1.08 	&	0.69	&	0.90	&	0.58	\\
$\mathbb{P}^1$ Upwind		&	0.21 	&	0.85	&	0.57	&	0.84	&	0.03
\end{tabular}
\caption{Relative errors in the single-scale case ($\alpha=1/128$, $\delta=0.5$, $\eps=1$ and $H=1/16$)
\label{tabl_eps_1}}
\end{center}
\end{table}

Likewise, in order to practically demonstrate the relevance of accounting for the small scale, we reinstate $\eps=1/64$ and display on Table~\ref{tabl1} the relative errors for the different methods. We indeed observe that, outside the boundary layer, the relative $H^1$ error of the $\mathbb{P}^1$ Upwind method is about three times as large as the error of the \MsFEMBSUPG{} method. 

\medskip

We now compare the accuracies of the methods. The results are shown on Table~\ref{tabl1}. We observe that all methods have an outrageously large error within the boundary layer (close to a hundred percent). The only exception to this is discussed in Section~\ref{change_bc} below, where we focus on the boundary layer and show that, specifically for the \MsFEMA{} method but not for the other methods, the accuracy (within the layer) is significantly improved upon changing the boundary conditions in the local problem~\eqref{local_problem_nos_msfemA}. 

As shown on Table~\ref{tabl1}, the \MsFEMA{} method has a relative $H^1$ error outside the layer about 7 times as large as the error of the \MsFEMBSUPG{} method. On this example, the methods that provide the lowest $H^1$ error outside the layer are the \MsFEMBSUPG{} method and the splitting method.

\begin{table}[htbp]
\begin{center}
\begin{tabular}{l | l l l l l}
	&	$e_{L^2}$	&	\textbf{$e_{H^1}$}	&	\textbf{$e_{L^\infty}$}	&	\textbf{$e_{H^1_{\text{in}}}$}	&	\textbf{$e_{H^1_{\text{out}}}$}	\\ \noalign{\vskip 3pt}	
\hline
 $\mathbb{P}^1$ Upwind		&	0.86	&	0.98	&	0.94	&	0.97	&	0.13	\\	
\MsFEMB		&	0.27	&	1.13	&	1.63	&	0.97	&	0.57	\\	
\MsFEMBSUPG	&	0.23	&	0.87	&	0.81	&	0.87	&	0.04	\\
\MsFEMA		&	0.11	&	0.74	&	0.62	&	0.68	&	0.29	\\
Splitting~\eqref{pb_spl_1}--\eqref{pb_spl_2} &	0.22	&	0.87	&	0.80	&	0.87	&	0.03	
\end{tabular}
\caption{Relative errors in the multiscale case ($\alpha=1/128$, $\delta=0.5$, $\eps=1/64$ and $H=1/16$)
\label{tabl1}}
\end{center}
\end{table}

\subsubsection{Comparison of the splitting methods}

We specifically compare here our two variants of the splitting approach:~\eqref{pb_spl_1_discrete_varf}--\eqref{pb_spl_2_discrete_varf} and~\eqref{pb_spl_beta_discret_1_discrete_varf}--\eqref{pb_spl_beta_discret_2_discrete_varf}.

In spite of the value of $\rho_-=128$ in~\eqref{inf_rho}, so that assumption~\eqref{naive_cond} of Lemma~\ref{lemma_conv} is violated, the method~\eqref{pb_spl_1_discrete_varf}--\eqref{pb_spl_2_discrete_varf} converges. For the approach~\eqref{pb_spl_beta_discret_1_discrete_varf}--\eqref{pb_spl_beta_discret_2_discrete_varf}, we choose $\beta$ as in~\eqref{beta_choice_discrete}, that is $\beta = 1.9941$. In the numerical tests, we have not used the projection $P_{\VB}$ (see Remark~\ref{remark_introduction_projection}). The contraction factor (see~\eqref{ineq_29} in the proof of Lemma~\ref{lem:iter_alternative}) is $\rho = 0.99902$. Given the proof of Lemma~\ref{lem:iter_alternative} and that value of $\rho$, the convergence is expected to be slow. It is indeed \emph{very} slow, as will now be seen, confirming that the approach is only advocated in the case where the convergence of~\eqref{pb_spl_1_discrete_varf}--\eqref{pb_spl_2_discrete_varf} fails. 

Table~\ref{tabl4} shows the accuracy of the methods. We see that the method~\eqref{pb_spl_1_discrete_varf}--\eqref{pb_spl_2_discrete_varf} is more accurate than the method~\eqref{pb_spl_beta_discret_1_discrete_varf}--\eqref{pb_spl_beta_discret_2_discrete_varf} (outside the layer). Both methods are inaccurate inside the layer. Figure~\ref{test_10_ite_spl} shows the error in function of the number of iterations. The method~\eqref{pb_spl_beta_discret_1_discrete_varf}--\eqref{pb_spl_beta_discret_2_discrete_varf} needs 100 times more iterations than the method~\eqref{pb_spl_1_discrete_varf}--\eqref{pb_spl_2_discrete_varf} to reach the same accuracy. 

\begin{table}[htbp]
\begin{center}
\begin{tabular}{l | l l l l l}
	&	$e_{L^2}$	&	\textbf{$e_{H^1}$}	&	\textbf{$e_{L^\infty}$}	&	\textbf{$e_{H^1_{\text{in}}}$}	&	\textbf{$e_{H^1_{\text{out}}}$}	\\ \noalign{\vskip 3pt}	
\hline
Method~\eqref{pb_spl_1_discrete_varf}--\eqref{pb_spl_2_discrete_varf}	&	0.22	&	0.87	&	0.80	&	0.87	&	0.03	\\
Method~\eqref{pb_spl_beta_discret_1_discrete_varf}--\eqref{pb_spl_beta_discret_2_discrete_varf} & 0.59 & 0.94 & 0.96 & 0.94 & 0.10\\
\end{tabular}
\caption{Relative errors for the two splitting methods ($\alpha=1/128$, $\delta=0.5$, $\eps=1/64$ and $H=1/16$)
\label{tabl4}}
\end{center}
\end{table}

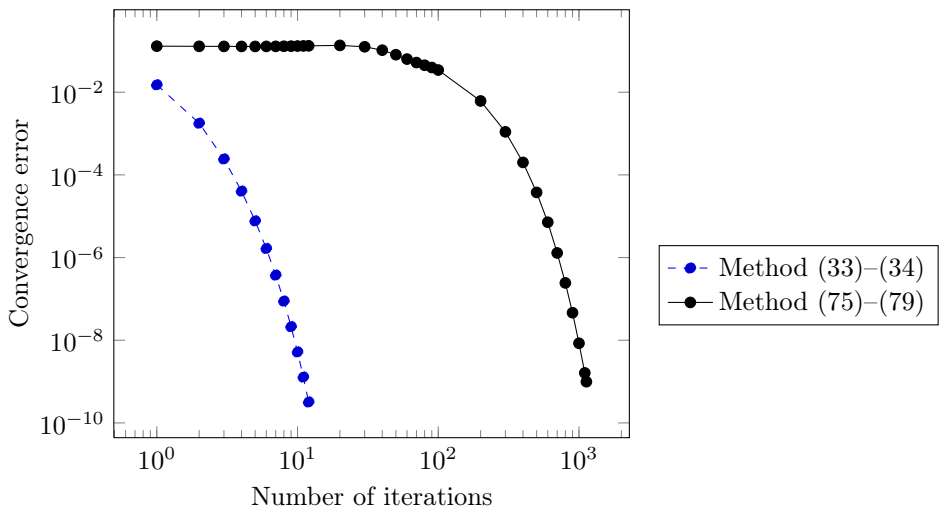
\begin{figure}[htbp]
\begin{center}
\begin{tikzpicture}
\begin{axis}[xlabel={Number of iterations},ylabel={Convergence error},legend entries={
{Method~\eqref{pb_spl_1_discrete_varf}--\eqref{pb_spl_2_discrete_varf}},{Method~\eqref{pb_spl_beta_discret_1_discrete_varf}--\eqref{pb_spl_beta_discret_2_discrete_varf}}},legend style={at={(1.6,0.45)}},xmode=log,ymode=log]
\addplot+[blue,dashed] table[x index=0,y index=1]{data/test_10_ite_spl-nos.dat};
\addplot[black,mark=*] table[x index=0,y index=1]{data/test_10_ite_spl_beta-nos.dat};
\end{axis}
\end{tikzpicture}
\end{center}
\caption{Convergence history of the two splitting methods ($\alpha=1/128$, $\delta=0.5$, $\eps=1/64$ and $H=1/16$)}
\label{test_10_ite_spl}
\end{figure}

In all what follows, we have only used the splitting method~\eqref{pb_spl_1_discrete_varf}--\eqref{pb_spl_2_discrete_varf}.

\subsubsection{Sensitivity with respect to the P\'eclet Number}
\label{section_sensitivity_peclet}

We set $\delta = 0.75$, $\eps=1/128$, $H=1/16$ and $\alpha=2^{-k}$, $k=2,\cdots,9$. We let~$\alpha$ vary in order to assess the robustness of the approaches with respect to the P\'eclet number. 

From~\eqref{test_conv_dom}, we suspect the convection-dominated regime corresponds to $k>5$. To doublecheck this is indeed the case, we first set $\eps=1$ and show on Figure~\ref{curve1_H1_eps_1} the relative errors of the $\mathbb{P}^1$ method and the $\mathbb{P}^1$ Upwind method. We indeed see that, for $k>5$, the relative $H^1$ error outside the layer of the $\mathbb{P}^1$ method is at least five times as large as the relative $H^1$ error outside the layer of the $\mathbb{P}^1$ Upwind method. In the sequel, we go back to the multiscale case with $\eps=1/128$.

\begin{figure}[htbp]
\begin{center}
\begin{tikzpicture}
\begin{axis}[extra x ticks={0.03125} , extra x tick labels={\color{red}Pe$\, H=1$},
extra x tick style={grid=major, dashed, tick label style={rotate=90,anchor=east}}, xlabel={$\alpha$},ylabel={$H^1$ relative error},legend entries={
%{$\mathbb{P}^1$},
{$\mathbb{P}^1: e_{H^1}$ },{$\mathbb{P}^1: e_{H^1_{\text{out}}}$},{$\mathbb{P}^1$ Upwind: $e_{H^1}$},{$\mathbb{P}^1$ Upwind: $e_{H^1_{\text{out}}}$}},cycle list name=mylistmethods,legend style={at={(1.6,0.45)}},xmode=log,ymode=log]
%\addplot table[x index=0,y index=6]{data/curve1H1};
\addplot table[x index=0 ,y index=1]{data/curve1_eps_1};
\addplot+[blue,dashed] table[x index=0,y index=2]{data/curve1_eps_1};
\addplot[black,mark=*] table[x index=0,y index=3]{data/curve1_eps_1};
\addplot table[x index=0,y index=4]{data/curve1_eps_1};
\end{axis}
\end{tikzpicture}
\end{center}
\caption{Relative errors in the single-scale case ($\delta = 0.75$, $\eps=1$ and $H=1/16$).}
\label{curve1_H1_eps_1}
\end{figure}
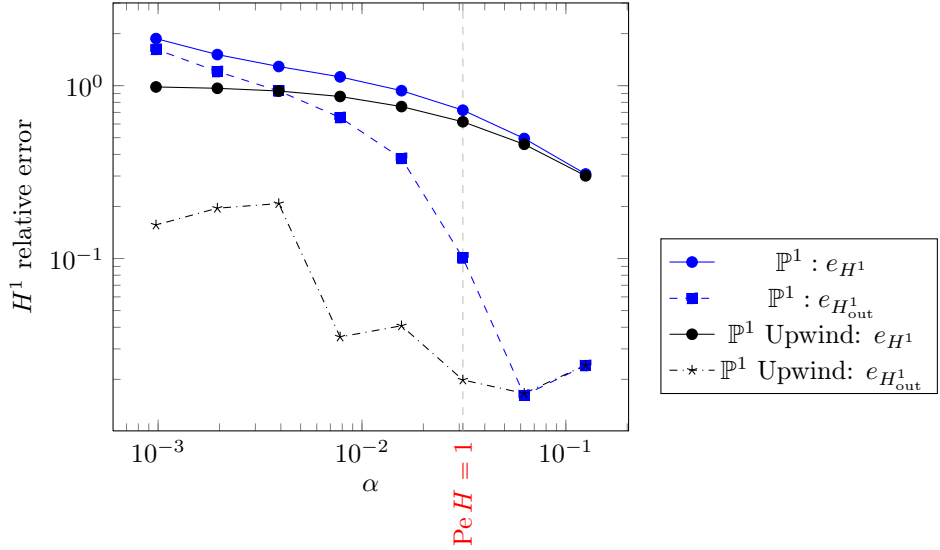

\pgfplotscreateplotcyclelist{mylist}{{red,mark=square*},{purple,mark=*},{violet,mark=o},{magenta,mark=+}}

\pgfplotscreateplotcyclelist{mylistmethodspone}{
{blue,dashed,mark=*},
{blue,mark=*},{red,mark=square*}, {brown!60!black,every mark/.append style={fill=brown!80!black},mark=otimes*},{dashdotted,mark=star},{mark=diamond*,loosely dashdotted,color=green}}

\paragraph*{Errors in the whole domain} 

Results are shown on Figure~\ref{curve1_H1_nos}. When $\alpha$ is small, all methods yield rather large errors. The error of the \MsFEMB{} method is significantly more important than the error of the \MsFEMBSUPG{} method. This indicates the presence of spurious oscillations on the solution obtained with the non stabilized \MsFEMB{} method. Hence the stabilization is important in this regime. The most robust methods are the \MsFEMA{} method, the \MsFEMBSUPG{} method and the splitting method. 

When $\alpha$ is large, the main difficulty is to capture the oscillations at scale $\eps$. As expected, all multiscale methods perform better than the $\mathbb{P}^1$ Upwind method. Note that the \MsFEMB{} method and the \MsFEMBSUPG{} method perform similarly. No stabilization is indeed necessary in that regime. 

In both regimes, we note that the \MsFEMA{} method performs the best. We also see that the errors of the splitting method are extremely close to the errors of the \MsFEMBSUPG{} method.

\begin{figure}[htbp]
\begin{center}
\begin{tikzpicture}
\begin{axis}[extra x ticks={0.03125} , extra x tick labels={\color{red}Pe$\, H=1$},
extra x tick style={grid=major, dashed, tick label style={rotate=90,anchor=east}}, xlabel={$\alpha$},ylabel={$H^1$ relative error},legend entries={
%{$\mathbb{P}^1$},
{$\mathbb{P}^1$ Upwind},{\MsFEMA},{\MsFEMB},{\MsFEMBSUPG},{Splitting}},cycle list name=mylistmethods,legend style={at={(1.6,0.45)}},xmode=log,ymode=log]
%\addplot table[x index=0,y index=6]{data/curve1H1};
\addplot table[x index=0 ,y index=1]{data/curve1H1};
\addplot table[x index=0,y index=2]{data/curve1H1};
\addplot table[x index=0,y index=3]{data/curve1H1};
\addplot table[x index=0,y index=4]{data/curve1H1};
\addplot table[x index=0,y index=5]{data/curve1H1};
\end{axis}
\end{tikzpicture}
\end{center}
\caption{Relative error $e_{H^1}$ ($\delta = 0.75$, $\eps=1/128$ and $H=1/16$).}
\label{curve1_H1_nos}
\end{figure}
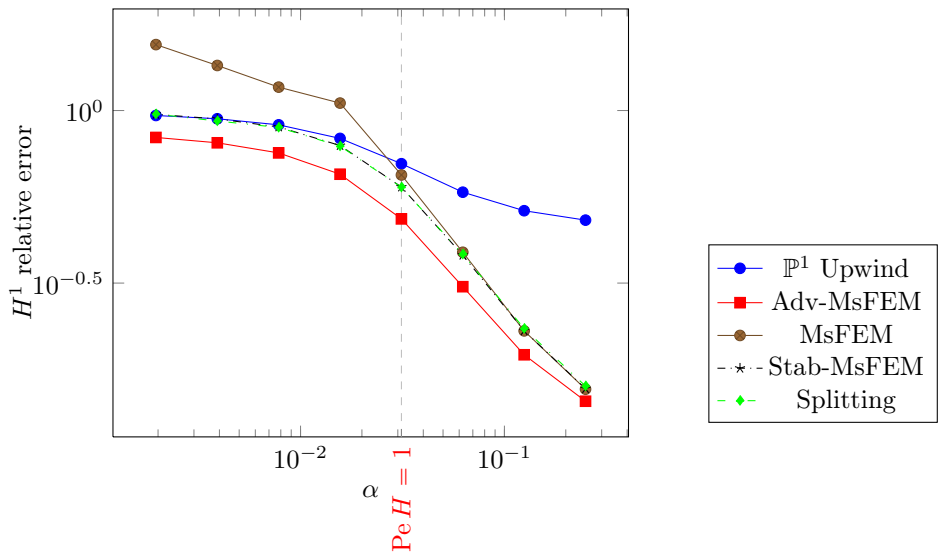

\paragraph{Errors outside the boundary layer}

It may be observed on Figure~\ref{curve1_H1out_nos} that the \MsFEMBSUPG{} method and the splitting method are the best methods outside the boundary layer. They essentially share the same accuracy. On the other hand, the \MsFEMA{} solution is systematically less accurate than the \MsFEMBSUPG{} solution. This suggests that encoding the convection in the multiscale basis functions is not necessary to obtain a good accuracy in this subdomain, and that it may even deteriorate the quality of the numerical solution. The \MsFEMB{} method is much less accurate than the stabilized \MsFEMBSUPG{} method when the coercivity constant $\alpha$ is small, and has a comparable accuracy when $\alpha$ is larger than $0.1$. 

When $\alpha$ is large (and hence the only difficulty is to capture the oscillation scale $\eps$), the $\mathbb{P}^1$ Upwind method is less accurate than the \MsFEMBSUPG{} method, as expected, since the latter encodes the oscillations of $A^\eps$ in the multiscale basis functions. When $\alpha$ is moderately small ($10^{-2} < \alpha < 1/32$ on Figure~\ref{curve1_H1out_nos}), the problem is both convection-dominated (we indeed observe that the \MsFEMBSUPG{} method provides a better accuracy than the \MsFEMB{} method) and multiscale (the \MsFEMBSUPG{} method is more accurate than the $\mathbb{P}^1$ Upwind method). However, when $\alpha$ is very small (here, $\alpha < 10^{-2}$), the convection is so large that it overshadows the multiscale nature of the problem. We then observe that the $\mathbb{P}^1$ Upwind method and the \MsFEMBSUPG{} method share the same accuracy. 

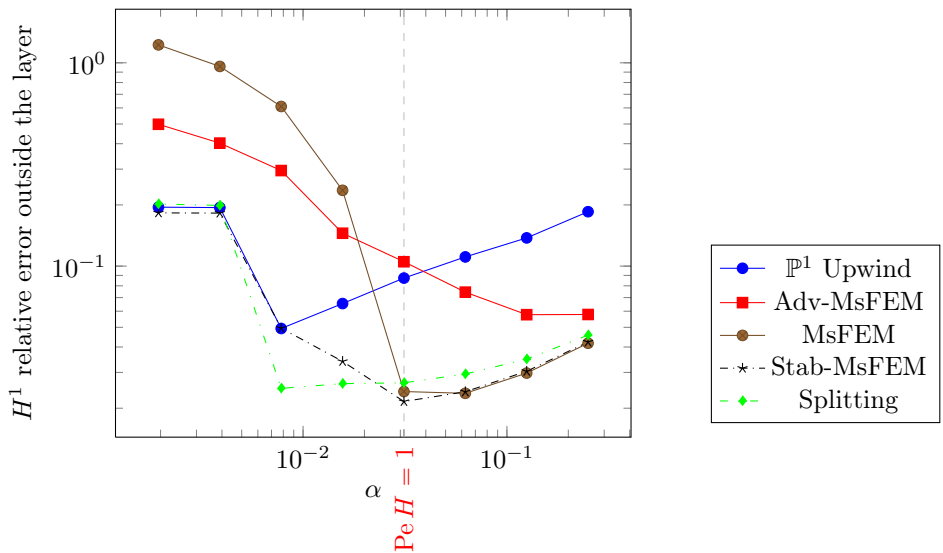
\begin{figure}[htbp]
\begin{center}
\begin{tikzpicture}
\begin{axis}[extra x ticks={0.03125} , extra x tick labels={\color{red}Pe$\, H=1$},
extra x tick style={grid=major,red, dashed, tick label style={rotate=90,anchor=east}}, xlabel={$\alpha$},ylabel={$H^1$ relative error outside the layer},legend entries={
%{$\mathbb{P}^1$},
{$\mathbb{P}^1$ Upwind},{\MsFEMA},{\MsFEMB},{\MsFEMBSUPG},{Splitting}},cycle list name=mylistmethods,legend style={at={(1.6,0.45)}},xmode=log,ymode=log]
%\addplot table[x index=0 ,y index=6]{data/curve1H1out};
\addplot table[x index=0 ,y index=1]{data/curve1H1out};
\addplot table[x index=0,y index=2]{data/curve1H1out};
\addplot table[x index=0,y index=3]{data/curve1H1out};
\addplot table[x index=0,y index=4]{data/curve1H1out};
\addplot table[x index=0,y index=5]{data/curve1H1out};
\end{axis}
\end{tikzpicture}
\end{center}
\caption{Relative error $e_{H^1_{\text{out}}}$ ($\delta = 0.75$, $\eps=1/128$ and $H=1/16$). }
\label{curve1_H1out_nos}
\end{figure}

Of course, the values of $\alpha$ that define these three regimes ((i) convection-dominated, (ii) {\em both} convection-dominated and multiscale, (iii) multiscale) depend on the problem considered, and in particular on the value of $\eps$. We have cheched this sensitivity by considering the following two test-cases: on Figures~\ref{fig:grand_eps} and~\ref{fig:petit_eps}, we consider the case $\eps=1/64$ and $\eps=1/256$, respectively. The other parameters are $\delta=0.75$ and $H=1/32$. When $\eps=1/256$, we observe that the \MsFEMBSUPG{} method is more accurate than the $\mathbb{P}^1$ Upwind method for any $1/512 \leq \alpha \leq 1/4$. In contrast, when $\eps=1/64$, the $\mathbb{P}^1$ Upwind method and the \MsFEMBSUPG{} method perform equally well when $\alpha \leq 1/256$. The sensitivity with respect to $\eps$ is also investigated in Section~\ref{sec:sen_eps} below (see e.g. Figure~\ref{curve2_H1out_nos}).

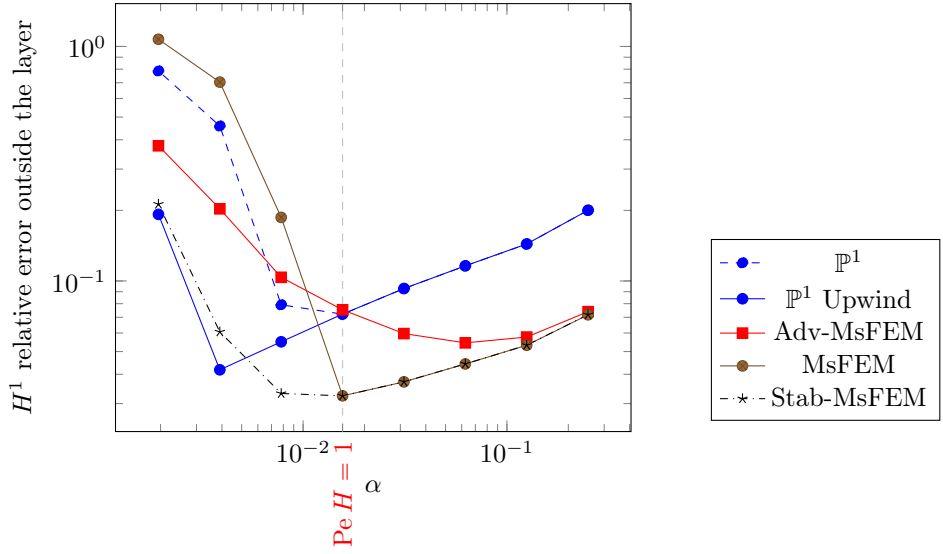
\begin{figure}[htbp]
\begin{center}
\begin{tikzpicture}
\begin{axis}[extra x ticks={0.015625} , extra x tick labels={\color{red}Pe$\, H=1$},
extra x tick style={grid=major,red, dashed, tick label style={rotate=90,anchor=east}}, xlabel={$\alpha$},ylabel={$H^1$ relative error outside the layer},legend entries={
{$\mathbb{P}^1$},
{$\mathbb{P}^1$ Upwind},{\MsFEMA},{\MsFEMB},{\MsFEMBSUPG}
%,{Splitting}
},cycle list name=mylistmethodspone,legend style={at={(1.6,0.45)}},xmode=log,ymode=log]
\addplot table[x index=0 ,y index=6]{courbes_francois_15oct15/sources_4nov15/test_119_H1out_H_32.asc};
\addplot table[x index=0 ,y index=1]{courbes_francois_15oct15/sources_4nov15/test_119_H1out_H_32.asc};
\addplot table[x index=0,y index=2]{courbes_francois_15oct15/sources_4nov15/test_119_H1out_H_32.asc};
\addplot table[x index=0,y index=3]{courbes_francois_15oct15/sources_4nov15/test_119_H1out_H_32.asc};
\addplot table[x index=0,y index=4]{courbes_francois_15oct15/sources_4nov15/test_119_H1out_H_32.asc};
%\addplot table[x index=0,y index=5]{data/test_94_H1out};
\end{axis}
\end{tikzpicture}
\end{center}
\caption{Relative error $e_{H^1_{\text{out}}}$ ($\delta = 0.75$, $\eps=1/64$ and $H=1/32$).}
\label{fig:grand_eps}
\end{figure}

\begin{figure}[htbp]
\begin{center}
\begin{tikzpicture}
\begin{axis}[extra x ticks={0.015625} , extra x tick labels={\color{red}Pe$\, H=1$},
extra x tick style={grid=major,red, dashed, tick label style={rotate=90,anchor=east}}, xlabel={$\alpha$},ylabel={$H^1$ relative error outside the layer},legend entries={
{$\mathbb{P}^1$},
{$\mathbb{P}^1$ Upwind},{\MsFEMA},{\MsFEMB},{\MsFEMBSUPG}
%,{Splitting}
},cycle list name=mylistmethodspone,legend style={at={(1.6,0.45)}},xmode=log,ymode=log]
\addplot table[x index=0 ,y index=5]{courbes_francois_15oct15/sources_4nov15/test_120_H1out_H_32.asc};
\addplot table[x index=0 ,y index=1]{courbes_francois_15oct15/sources_4nov15/test_120_H1out_H_32.asc};
\addplot table[x index=0,y index=2]{courbes_francois_15oct15/sources_4nov15/test_120_H1out_H_32.asc};
\addplot table[x index=0,y index=3]{courbes_francois_15oct15/sources_4nov15/test_120_H1out_H_32.asc};
\addplot table[x index=0,y index=4]{courbes_francois_15oct15/sources_4nov15/test_120_H1out_H_32.asc};
%\addplot table[x index=0,y index=5]{data/test_94_H1out};
\end{axis}
\end{tikzpicture}
\end{center}
\caption{Relative error $e_{H^1_{\text{out}}}$ ($\delta = 0.75$, $\eps=1/256$ and $H=1/32$).}
\label{fig:petit_eps}
\end{figure}
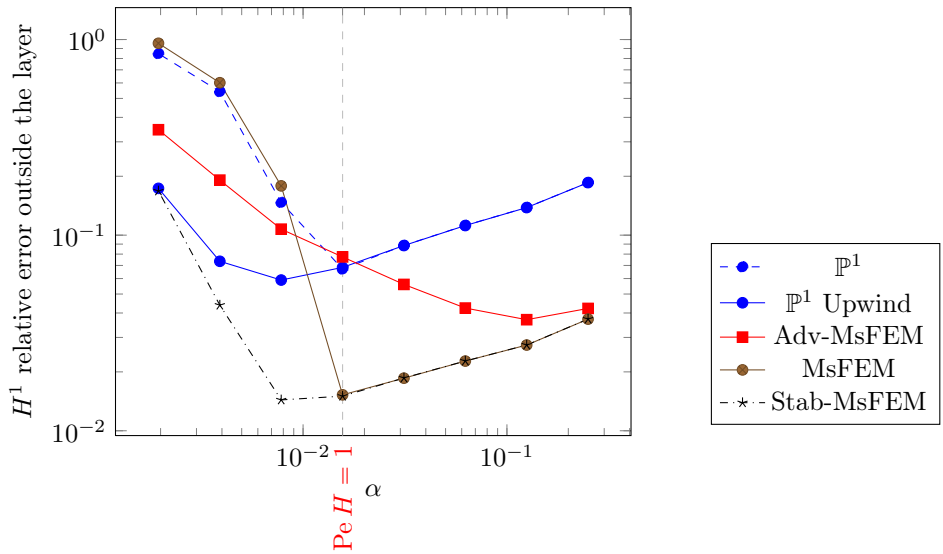

\subsubsection{Sensitivity with respect to the oscillation scale}
\label{sec:sen_eps}

In this section, the sensitivity of the different numerical methods to the oscillation scale $\eps$ is assessed. We work with the parameters $\delta = 0.75$, $H=1/32$, $\alpha=1/128$ and $\eps=2^{-k}$, $k=3,\cdots,8$, so that~Pe$\, H=2>1$. Table~\ref{tabl_eps_2} displays the relative errors of the $\mathbb{P}^1$ method and the $\mathbb{P}^1$ Upwind method for $\eps=1$. Outside the layer, the relative $H^1$ error of the $\mathbb{P}^1$ method is about 30 times as large as the error of the $\mathbb{P}^1$ Upwind method. The problem is convection-dominated.

\begin{table}[htbp]
\begin{center}
\begin{tabular}{l | l l l l l}
	&	$e_{L^2}$	&	\textbf{$e_{H^1}$}	&	\textbf{$e_{L^\infty}$}	&	\textbf{$e_{H^1_{\text{in}}}$}	&	\textbf{$e_{H^1_{\text{out}}}$}	\\ \noalign{\vskip 3pt}		
\hline					
$\mathbb{P}^1$	&	0.11 &	0.93 	&	0.48	&	0.86	&	0.33	\\
$\mathbb{P}^1$ Upwind		&	0.11 	&	0.75	&	0.46	&	0.75	&	0.01	\\
\end{tabular}
\caption{Relative errors in the single-scale case ($\alpha=1/128$, $\delta = 0.75$, $\eps=1$ and $H=1/32$).
\label{tabl_eps_2}}
\end{center}
\end{table}

Figures~\ref{curve2_H1_nos} and~\ref{curve2_H1out_nos} respectively show the relative global $H^1$ error and the relative $H^1$ error outside the boundary layer. The relative global $H^1$ error does not seem to be sensitive to the oscillation scale, as we can see on Figure~\ref{curve2_H1_nos}. This error is dominated by the error located in the thin boundary layer due to the convection-dominated regime. 

On Figure~\ref{curve2_H1out_nos}, two regions can be distinguished. In the region $\eps<H$, the \MsFEMBSUPG{} method performs better than the $\mathbb{P}^1$ Upwind method. The error of the \MsFEMBSUPG{} method decreases as $\eps$ decreases (but its cost increases correspondingly, as the mesh to compute the highly oscillatory basis functions has to be finer), whereas the error of the $\mathbb{P}^1$ Upwind method remains constant at a large value as $\eps$ decreases. The \MsFEMA{} method yields a large error (due to the mismatch between the shape of the solution outside the boundary layer and the shape of the basis functions). The \MsFEMB{} method is also inaccurate, given the absence of any stabilization. 

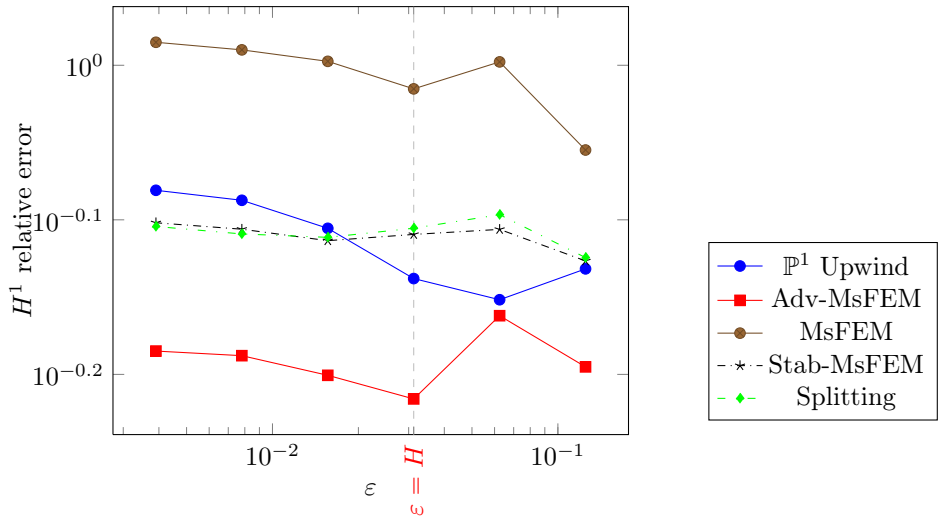
\begin{figure}[htbp]
\begin{center}
\begin{tikzpicture}
\begin{axis}[extra x ticks={0.03125} , extra x tick labels={\color{red}$\eps=H$},
extra x tick style={grid=major, dashed, tick label style={rotate=90,anchor=east}}, xlabel={$\eps$},ylabel={$H^1$ relative error},legend entries={
%{$\mathbb{P}^1$},
{$\mathbb{P}^1$ Upwind},{\MsFEMA},{\MsFEMB},{\MsFEMBSUPG},{Splitting}},cycle list name=mylistmethods,legend style={at={(1.6,0.45)}},xmode=log,ymode=log]
%\addplot table[x index=0 ,y index=6]{data/curve2H1};
\addplot table[x index=0 ,y index=1]{data/curve2H1};
\addplot table[x index=0,y index=2]{data/curve2H1};
\addplot table[x index=0,y index=3]{data/curve2H1};
\addplot table[x index=0,y index=4]{data/curve2H1};
\addplot table[x index=0,y index=5]{data/curve2H1};
\end{axis}
\end{tikzpicture}
\end{center}
\caption{Relative error $e_{H^1}$ ($\alpha=1/128$, $\delta = 0.75$ and $H=1/32$).}
\label{curve2_H1_nos}
\end{figure}

\begin{figure}[htbp]
\begin{center}
\begin{tikzpicture}
\begin{axis}[extra x ticks={0.03125} , extra x tick labels={\color{red}$\eps=H$},
extra x tick style={grid=major, dashed, tick label style={rotate=90,anchor=east}}, xlabel={$\eps$},ylabel={$H^1$ relative error outside the layer},legend entries={
%{$\mathbb{P}^1$},
{$\mathbb{P}^1$ Upwind},{\MsFEMA},{\MsFEMB},{\MsFEMBSUPG},{Splitting}},cycle list name=mylistmethods,legend style={at={(1.6,0.45)}},xmode=log,ymode=log]
%\addplot table[x index=0 ,y index=6]{data/curve2H1out};
\addplot table[x index=0 ,y index=1]{data/curve2H1out};
\addplot table[x index=0,y index=2]{data/curve2H1out};
\addplot table[x index=0,y index=3]{data/curve2H1out};
\addplot table[x index=0,y index=4]{data/curve2H1out};
\addplot table[x index=0,y index=5]{data/curve2H1out};
\end{axis}
\end{tikzpicture}
\end{center}
\caption{Relative error $e_{H^1_{\text{out}}}$ ($\alpha=1/128$, $\delta = 0.75$ and $H=1/32$).}
\label{curve2_H1out_nos}
\end{figure}
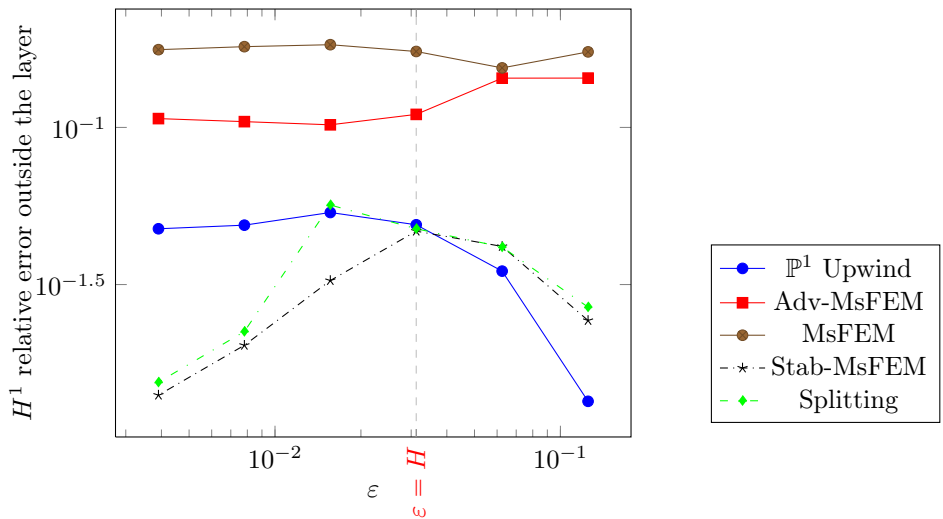

\subsubsection{Influence of the boundary conditions imposed on the local problems}
\label{change_bc}

In all the above experiments, the boundary conditions we have supplied the local problems~\eqref{local_pb_msfem_general} and~\eqref{local_problem_nos_msfemA} with are {\em linear} boundary conditions. For other choices of boundary conditions, our results remain qualitatively unchanged. We however wish to now investigate how the choice of boundary conditions affects the accuracy of the approaches \emph{within} the boundary layer, since this is there that the approaches equally poorly perform. It is known that, in general, the oversampling method is one of the best multiscale approach available for the multiscale diffusion problem~\eqref{pb_diff_multiscale}. Whether this superiority also survives in the presence of a strong advection is an interesting issue.

For clarity, the \MsFEMA{} method as presented above (i.e. based on the local problem~\eqref{local_problem_nos_msfemA}) is denoted here the \MsFEMA{} lin method. The other boundary conditions that we consider are 
\begin{itemize}
\item Oversampling boundary condition, with an oversampling ratio equal to 3. This method is denoted the \MsFEMA{} OS method;
\item Crouzeix-Raviart type boundary condition. This method is denoted the \MsFEMA{} CR method.
\end{itemize}
The oversampling method is described in~\cite{hou1997multiscale}. The MsFEM \`a la Crouzeix-Raviart has been introduced in~\cite{lebris2013msfem,lebris2014msfem}. Both the \MsFEMA{} OS and the \MsFEMA{} CR methods are non-conforming approaches. The relative $H^1$ error of those methods is therefore computed with the broken $H^1$ norm
$$
e_{H^1_{\text{in}}}(u_1)=\frac{\|u_1-u_{\text{ref}}\|_{H^1(\mathcal{T}_H\medcap\Omega_{\text{layer}})}}{\|u_{\text{ref}}\|_{H^1(\Omega)}},
$$
with $\dis \|u\|_{H^1(\mathcal{T}_H\medcap\Omega_{\text{layer}})}^2 =\sum_{\textbf{K}\in\mathcal{T}_H} \|u\|_{H^1(\textbf{K}\medcap\Omega_{\text{layer}})}^2$.
 
We first study the example presented in Section~\ref{an_example}. Table~\ref{tabl3} show the relative errors. We observe that there is at least a factor 2 between the relative $H^1$ error inside the layer of the \MsFEMA{} lin and the other \MsFEMA{} methods. The improvement in the accuracy outside the boundary layer is less important, although significant for the \MsFEMA{} CR method.

\begin{table}[htbp]
\begin{center}
\begin{tabular}{l | l l l l l}
	&	$e_{L^2}$	&	\textbf{$e_{H^1}$}	&	\textbf{$e_{L^\infty}$}	&	\textbf{$e_{H^1_{\text{in}}}$}	&	\textbf{$e_{H^1_{\text{out}}}$}	\\ \noalign{\vskip 3pt}						
\hline
\MsFEMA{} lin&	0.11	&	0.74	&	0.62	&	0.68	&	0.29\\
\MsFEMA{} OS	&	0.36	&	0.42	&	0.55	&	0.34&	0.24	\\
\MsFEMA{} CR		&	0.38	&	0.25	&	0.70	&	0.18	&	0.18
\end{tabular}
\caption{Relative errors for different boundary conditions in the local problems.
\label{tabl3}}
\end{center}
\end{table}

Second, we consider the setting presented in Section~\ref{section_sensitivity_peclet}. Figure~\ref{curve1_H1in_bcs} shows the relative $H^1$ error inside the layer for the different \MsFEMA{} methods. We observe that the boundary conditions imposed on the local problems affect the accuracy. The \MsFEMA{} lin method always has the largest error. In the convection-dominated regime, the \MsFEMA{} CR is the best method. At Pe$\, H=16$, there is a factor 2 between the relative $H^1$ error of the \MsFEMA{} lin method and the relative $H^1$ error of the \MsFEMA{} CR method (inside the layer). This shows that the convective profile should be encoded in some way in the boundary conditions imposed on the local problem in order for the solution to be accurate in the boundary layer region for the convection-dominated regime. For the MsFEM approaches other than the \MsFEMA{} approach, we have also performed similar experiments, which are not included here, and which do not seem to show any significant dependency of the accuracy inside the layer upon the boundary conditions of the local problems. 

Figure~\ref{curve1_H1out_bcs} shows the relative $H^1$ error outside the layer for the different \MsFEMA{} methods. It may be observed that the \MsFEMA{} lin method and the \MsFEMA{} OS method share the same error. The error of the \MsFEMA{} CR is the smallest in the convection-dominated regime. However, the errors outside the boundary layer of the various \MsFEMA{} methods are yet larger than the error outside the layer of the \MsFEMBSUPG{} method.

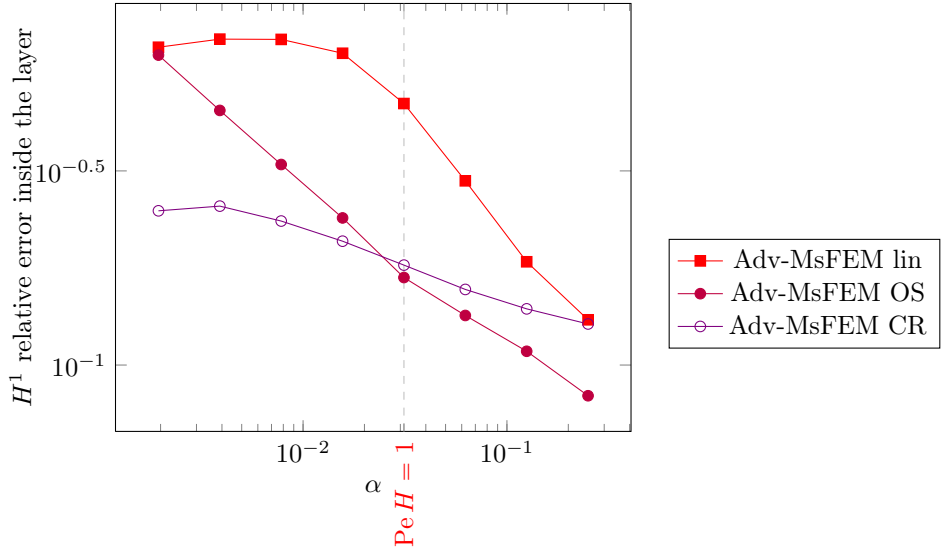
\begin{figure}[htbp]
\begin{center}
\begin{tikzpicture}
\begin{axis}[extra x ticks={0.03125} , extra x tick labels={\color{red}Pe$\, H=1$},
extra x tick style={grid=major, dashed, tick label style={rotate=90,anchor=east}}, xlabel={$\alpha$},ylabel={$H^1$ relative error inside the layer},cycle list name=mylist,legend entries={{\MsFEMA{} lin},{\MsFEMA{} OS},{\MsFEMA{} CR}},legend style={at={(1.6,0.45)}},xmode=log,ymode=log]
%\addplot table[x index=0,y index=2]{data/curve1H1};
%\addplot table[x index=0,y index=2]{data/curve1H1osc};
%\addplot table[x index=0,y index=2]{data/curve1H1OS};
%\addplot table[x index=0,y index=6]{data/curve1H1osc};
\addplot table[x index=0,y index=2]{data/curve1H1in};
%\addplot table[x index=0,y index=2]{data/curve1H1inosc};
\addplot table[x index=0,y index=2]{data/curve1H1inOS};
\addplot table[x index=0,y index=6]{data/curve1H1inosc};
%\addplot table[x index=0,y index=2]{data/curve1H1out};
%\addplot table[x index=0,y index=2]{data/curve1H1outosc};
%\addplot table[x index=0,y index=2]{data/curve1H1outOS};
%\addplot table[x index=0,y index=6]{data/curve1H1outosc};
\end{axis}
\end{tikzpicture}
\end{center}
\caption{Relative error $e_{H^1_{\text{in}}}$ for the \MsFEMA{} methods.}
\label{curve1_H1in_bcs}
\end{figure}

\begin{figure}[htbp]
\begin{center}
\begin{tikzpicture}
\begin{axis}[extra x ticks={0.03125} , extra x tick labels={\color{red}Pe$\, H=1$},
extra x tick style={grid=major, dashed, tick label style={rotate=90,anchor=east}}, xlabel={$\alpha$},ylabel={$H^1$ relative error outside the layer},cycle list name=mylist,legend entries={{\MsFEMA{} lin},{\MsFEMA{} OS},{\MsFEMA{} CR},{\MsFEMBSUPG},{\MsFEMBSUPG{} OS}},legend style={at={(1.6,0.45)}},xmode=log,ymode=log]
\addplot table[x index=0,y index=2]{data/curve1H1out};
%\addplot table[x index=0,y index=2]{data/curve1H1outosc};
\addplot table[x index=0,y index=2]{data/curve1H1outOS};
\addplot table[x index=0,y index=6]{data/curve1H1outosc};
\addplot[black,mark=asterisk] table[x index=0,y index=4]{data/curve1H1out};
%\addplot[black,mark=*] table[x index=0,y index=4]{data/curve1H1outosc};
%\addplot[black,mark=square] table[x index=0,y index=4]{data/curve1H1outOS};
\end{axis}
\end{tikzpicture}
\end{center}
\caption{Relative error $e_{H^1_{\text{out}}}$ for the \MsFEMA{} methods.}
\label{curve1_H1out_bcs}
\end{figure}
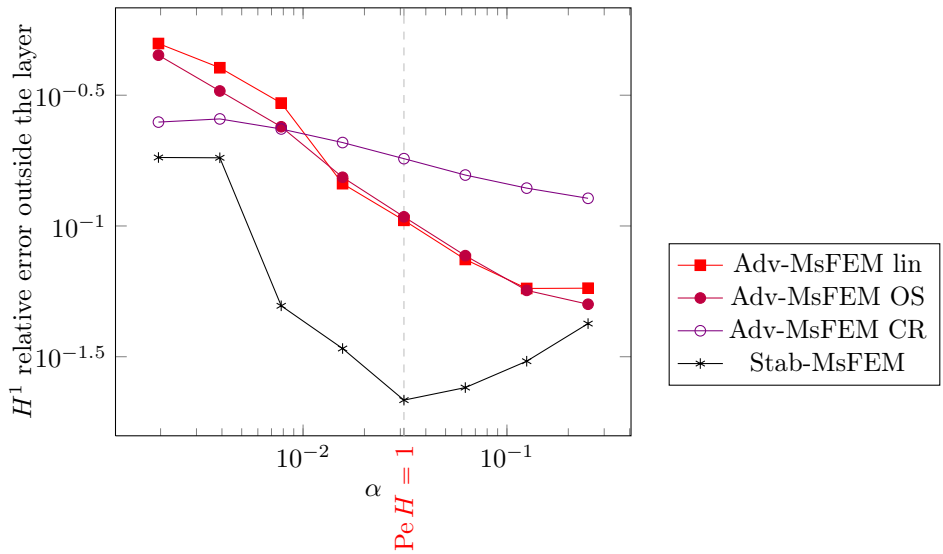

\subsection{Computational costs}
\label{section_cost}

We now turn to the computational costs of the different numerical methods. We recall that the splitting method we consider below is~\eqref{pb_spl_1_discrete_varf}--\eqref{pb_spl_2_discrete_varf}.

\subsubsection{Reference test}

We consider the reference test presented in Section~\ref{an_example}. Table~\ref{tabl2} shows the offline cost and the online cost (in seconds) of the different numerical methods.

\begin{table}[htbp]
\begin{center}
\begin{tabular}{l l l| l l l }
\textbf{Direct solvers}	&	Offline (s)	&	Online (s)	& \textbf{Iterative solvers}		&	Offline (s)	&	Online (s)	\\
\hline \noalign{\vskip 2pt}
\MsFEMBSUPG	&	$1.98\cdot 10^{2}$	&	$2.24\cdot 10^{-4}$	&	\MsFEMBSUPG	&	$2.63\cdot 10^{2}$	&	$5.78\cdot 10^{-4}$	\\
Splitting	&	$2.29\cdot 10^{2}$	&	$3.81\cdot 10^{-3}$	&	Splitting	&	$2.65\cdot 10^{2}$	&	$9.03\cdot 10^{-3}$	\\
\MsFEMB	&	$1.80\cdot 10^{2}$	&	$2.41\cdot 10^{-4}$	&	\MsFEMB	&	$2.33\cdot 10^{2}$	&	$1.63\cdot 10^{-3}$	\\
\MsFEMA	&	$1.84\cdot 10^{2}$	&	$2.20\cdot 10^{-4}$	&	\MsFEMA{} 	&	$5.89\cdot 10^{2}$	&	$6.99\cdot 10^{-4}$	\\
\end{tabular}
\caption{Computational costs. \label{tabl2}}
\end{center}
\end{table}

\paragraph*{Direct solvers} 

All the methods (but the splitting method) essentially share the same offline cost. The \MsFEMBSUPG{} method is slightly more expensive than the \MsFEMB{} variant because of the assembling of the stabilization term. The splitting method has the largest offline cost because there are more computations (two assemblings) than in the other methods.

The online cost of the splitting method is about 15 times as large as the online cost of the other methods. This corresponds to the number of iterations of the splitting method. Note that the online cost corresponds to solving the linear system from an already factorized matrix, which is negligible. 

\paragraph*{Iterative solvers}

The online cost of the intrusive methods (\MsFEMA, \MsFEMB, \MsFEMBSUPG) corresponds to calling the GMRES solver. There are some differences in these costs because the number of iterations of the GMRES solver is sensitive to the condition number of the matrix that depends on the method. The online cost of the splitting method is still the largest because of the iteration loop of the splitting method. It is again about 15 times larger than the online cost of the \MsFEMBSUPG{} method. In this particular case, the splitting method needs 12 iterations to converge. The online costs are larger now than for direct solvers, of course. 

The main part of the offline cost comes from solving the local problems. The \MsFEMB, \MsFEMBSUPG, and the splitting method share the same local problems, namely~\eqref{local_pb_msfem_general}. This is why they essentially share the same offline cost. In the \MsFEMA{} method, the local problem to solve is~\eqref{local_problem_nos_msfemA}. We observe that its offline cost is about 2 times larger than for the other methods. We thus see that the computational cost of solving with the GMRES solver the non-symmetric linear system corresponding to the local problem~\eqref{local_problem_nos_msfemA} is higher that the cost of solving with the conjugate gradient method the symmetric linear system stemming from the local problem~\eqref{local_pb_msfem_general}. 

\subsubsection{Dependency with respect to the P\'eclet number} 

We again consider the setting of Section~\ref{section_sensitivity_peclet} where we now vary the coefficient $\alpha$ and thus the P\'eclet number. Figures~\ref{curve1_online_nos} and~\ref{curve1_splitting_iteration_nos} respectively show the online cost (in seconds) of the different numerical methods and the number of iterations of the splitting method as a function of $\alpha$. 

In Figure~\ref{curve1_online_nos}, we observe that the \MsFEMA{} method and the \MsFEMBSUPG{} method share the same online cost. The online costs of the two methods and the online cost of the \MsFEMB{} method (with direct solvers) do not seem to strongly depend on the P\'eclet number. The online cost of the \MsFEMB{} method with iterative solvers increases as $\alpha$ decreases, since the condition number of the stiffness matrix then increases. The splitting method is, overall, significantly more expensive than the other approaches. 

Figure~\ref{curve1_splitting_iteration_nos} shows that the number of iterations in the splitting method grows as $\alpha$ decreases. The number of iterations is larger when using iterative solvers than when using direct solvers, although the difference fades as the convection becomes dominant.

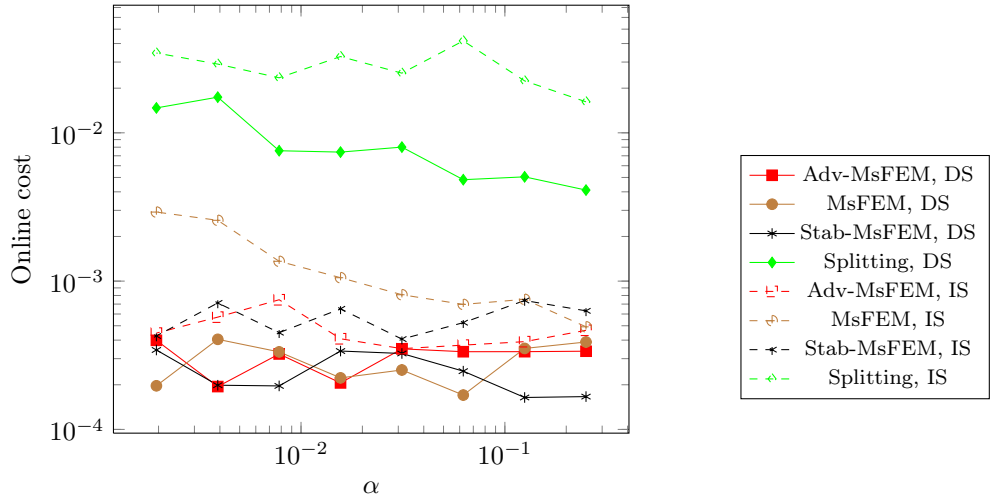
\begin{figure}[htbp]
\begin{center}
\begin{tikzpicture}
\begin{axis}[xlabel={$\alpha$},ylabel={Online cost},legend entries={{\footnotesize \MsFEMA, DS},{\footnotesize \MsFEMB, DS},{\footnotesize \MsFEMBSUPG, DS},{\footnotesize Splitting, DS},{\footnotesize \MsFEMA, IS},{\footnotesize \MsFEMB, IS},{\footnotesize \MsFEMBSUPG, IS},{\footnotesize Splitting, IS}},legend style={at={(1.7,0.65)}},xmode=log,ymode=log]
\addplot[mark=square*,color=red] table[x index=0 ,y index=1]{data/curve1online};
\addplot[mark=oplus*,color=brown] table[x index=0,y index=2]{data/curve1online};
\addplot[mark=asterisk,color=black] table[x index=0,y index=3]{data/curve1online};
\addplot[mark=diamond*,color=green] table[x index=0,y index=4]{data/curve1online};
\addplot[dashed,mark=square,color=red] table[x index=0 ,y index=1]{data/curve1online_iterative};
\addplot[dashed,mark=oplus,color=brown] table[x index=0,y index=2]{data/curve1online_iterative};
\addplot[dashed,mark=asterisk,color=black] table[x index=0,y index=3]{data/curve1online_iterative};
\addplot[dashed,mark=diamond,color=green] table[x index=0,y index=4]{data/curve1online_iterative};
\end{axis}
\end{tikzpicture}
\end{center}
\caption{Online costs (s) for the different numerical methods, using direct (DS) or iterative (IS) solvers.}
\label{curve1_online_nos}
\end{figure}

\begin{figure}[htbp]
\begin{center}
\begin{tikzpicture}
\begin{axis}[xlabel={$\alpha$},ylabel={Number of iterations},legend entries={{DS},{IS}},legend style={at={(0.65,0.25)}},xmode=log]
\addplot[mark=diamond*,color=green] table[x index=0,y index=5]{data/curve1online};
\addplot[dashed,mark=diamond,color=green] table[x index=0,y index=5]{data/curve1online_iterative};
\end{axis}
\end{tikzpicture}
\end{center}
\caption{Number of iterations of the splitting method for direct (DS) and iterative (IS) solvers.}
\label{curve1_splitting_iteration_nos}
\end{figure}
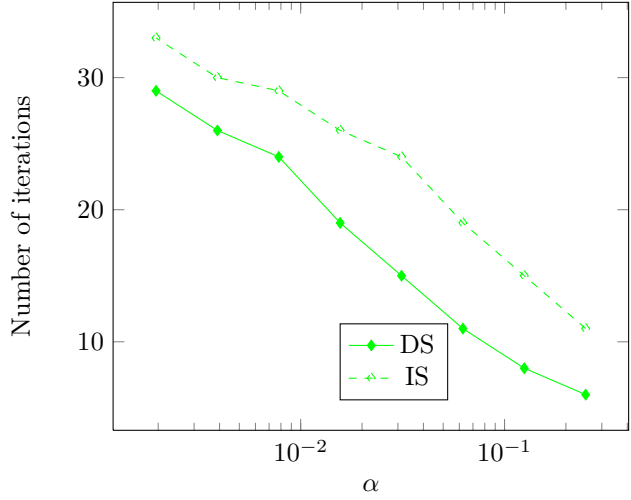

\section*{Acknowledgements}

The work presented in this article elaborates on a preliminary work that explored some of the issues on a prototypical one-dimensional setting, and which was performed in the context of the internship of H.~Ruffieux~\cite{HR2013} at CERMICS, \'Ecole des Ponts ParisTech. The present work benefits from this previous work. The authors wish to thank A.~Quarteroni for stimulating and enlightning discussions. CLB and FL also gratefully acknowledge the long term interaction with U.~Hetmaniuk (University of Washington in Seattle) and A.~Lozinski (Universit\'e de Besan\c con) on numerical methods for multiscale problems. The work of the authors is partially supported by ONR under Grant N00014-12-1-0383 and EOARD under Grant FA8655-13-1-3061.

\appendix

\section{Proof of~\eqref{apriori_supg_singlescale}}
\label{sec:appA}

By standard finite element results (see e.g.~\cite[end of Section IX.3]{bernardi-maday-rapetti} or~\cite[Remark 1.129 and Lemma 1.127]{ern2004theory}), we have the following best approximation property: for any $H$ and any $v \in H^1_0(\Omega) \cap H^2(\Omega)$, there exists $I_H v \in V_H$ such that
\begin{equation}
\label{app_prop}
\|v-I_Hv\|_{L^2(\Omega)} + H \, | v - I_H v |_{H^1(\Omega)} \leq C_{\rm FE} H^2 | v |_{H^2(\Omega)}
\end{equation}
where $C_{\rm FE}$ is independent of $H$ and $v$.

Our proof of~\eqref{apriori_supg_singlescale} is inspired by~\cite{HR2013} and~\cite[Theorem 11.2]{quarteroni2010numerical}. We split the error $u-u^s_H$ in two parts, $e^I=u-I_H u$ and~$
e^I_H=u^s_H -I_H u$. Given~\eqref{app_prop}, we have
\begin{align}
\label{eq:bout1}
|e^I|_{H^1(\Omega)}\leq CH|u|_{H^2(\Omega)}.
\end{align}
We now estimate the term $|e^I_H|_{H^1(\Omega)}$. The Galerkin orthogonality and~\eqref{div_b_zero} give
\begin{align}
&\alpha|e^I_H|_{H^1(\Omega)}^2 +a_{\text{stab}}(e^I_H,e^I_H)
\nonumber\\
&= a(e^I_H,e^I_H)+a_{\text{stab}}(e^I_H,e^I_H)
\nonumber\\
&=a(e^I,e^I_H)+a_{\text{stab}}(e^I,e^I_H)
\nonumber\\
&=\int_\Omega\left(\alpha\nabla e^I\cdot\nabla e^I_H + (b\cdot\nabla e^I) e^I_H\right) + \sum_{\mathbf{K}\in\mathcal{T}_H} \bigg(\tau_{\mathbf{K}}\mathcal{L}e^I,b\cdot \nabla e^I_H\bigg)_\mathbf{K}
\nonumber\\
&=\int_\Omega\left(\alpha\nabla e^I \cdot \nabla e^I_H - (b \cdot \nabla e^I_H) e^I\right) + \sum_{\mathbf{K} \in \mathcal{T}_H} \bigg( \tau_{\mathbf{K}} \mathcal{L} e^I,b\cdot \nabla e^I_H\bigg)_\mathbf{K}.
\label{err_detail}
\end{align}
Estimating each term of the right-hand side of~\eqref{err_detail}, we successively obtain:
\begin{itemize}
\item for the first term, using~\eqref{eq:bout1},
$$
\int_\Omega\alpha\nabla e^I\cdot\nabla e^I_H\leq \frac{\alpha}{4}|e^I_H|_{H^1(\Omega)}^2 + \alpha |e^I|_{H^1(\Omega)}^2 \leq \frac{\alpha}{4}|e^I_H|_{H^1(\Omega)}^2 + C \alpha H^2 |u|_{H^2(\Omega)}^2.
$$
\item for the second term,
\begin{eqnarray*} 
- \int_\Omega (b\cdot\nabla e^I_H) e^I 
&\leq&
\frac{1}{4}\sum_{\mathbf{K}\in\mathcal{T}_H}\|\tau_\mathbf{K}^{1/2} (b\cdot\nabla e^I_H) \|_{L^2(\mathbf{K})}^2 + \sum_{\mathbf{K}\in\mathcal{T}_H}\|\tau_\mathbf{K}^{-1/2} e^I\|_{L^2(\mathbf{K})}^2
\\
&\leq& 
\frac{1}{4} \sum_{\mathbf{K}\in\mathcal{T}_H}\|\tau_\mathbf{K}^{1/2} (b\cdot\nabla e^I_H) \|_{L^2(\mathbf{K})}^2 + \frac{2\|b\|_{L^\infty}}{H}\|e^I\|_{L^2(\Omega)}^2 
\\
&\leq& 
\frac{1}{4} a_{\text{stab}}(e^I_H,e^I_H) + C \|b\|_{L^\infty} H^3 |u|_{H^2(\Omega)}^2,
\end{eqnarray*}
where we have used that $\Delta e^I_H = 0$ in the first term and~\eqref{eq:def_tau} and~\eqref{app_prop} in the second term;
\item for the first part of the third term, using that $\Delta e^I = \Delta u$ (because $V_H$ is the $\mathbb{P}^1$ finite element space),
\begin{eqnarray*}
&& \sum_{\mathbf{K}\in\mathcal{T}_H} \bigg(\tau_{\mathbf{\mathbf{K}}}(-\alpha \Delta e^I),b\cdot \nabla e^I_H\bigg)_\mathbf{K}
\\
& \leq &
\frac{1}{2}\sum_{\mathbf{K}\in\mathcal{T}_H}\|\tau_\mathbf{K}^{1/2} \alpha \Delta u\|_{L^2(\mathbf{K})}^2 +\frac{1}{2}\sum_{\mathbf{K}\in\mathcal{T}_H}\|\tau_\mathbf{K}^{1/2} (b\cdot\nabla e^I_H) \|_{L^2(\mathbf{K})}^2
\\ 
&\leq &
\frac{\|b\|_{L^\infty} H^3}{16}|u|_{H^2(\Omega)}^2 +\frac{1}{2}\sum_{\mathbf{K}\in\mathcal{T}_H}\|\tau_\mathbf{K}^{1/2} (b\cdot\nabla e^I_H) \|_{L^2(\mathbf{K})}^2 
\\
&\leq &
C \|b\|_{L^\infty} H^3 |u|_{H^2(\Omega)}^2 +\frac{1}{2} a_{\text{stab}}(e^I_H,e^I_H),
\end{eqnarray*}
where we have used~\eqref{advect_x_regime} % in the third line 
to obtain $\dis \alpha^2\tau_\mathbf{K}(x) \leq \left(\frac{|b(x)|H}{2}\right)^2 \frac{H}{2|b(x)|} \leq \frac{\|b\|_{L^\infty} H^3}{8}$;
\item for the second part of the third term, 
\begin{eqnarray*}
&& \sum_{\mathbf{K}\in\mathcal{T}_H} \bigg(\tau_{\mathbf{K}}b\cdot\nabla e^I,b\cdot \nabla e^I_H\bigg)_\mathbf{K}
\\
&\leq& 
\sum_{\mathbf{K}\in\mathcal{T}_H}\|\tau_\mathbf{K}^{1/2} (b\cdot\nabla e^I) \|_{L^2(\mathbf{K})}^2+\sum_{\mathbf{K}\in\mathcal{T}_H}\frac{1}{4}\|\tau_\mathbf{K}^{1/2} (b\cdot\nabla e^I_H) \|_{L^2(\mathbf{K})}^2
\\
&\leq& 
\frac{\|b\|_{L^\infty} H}{2}|e^I|_{H^1(\Omega)}^2 +\sum_{\mathbf{K}\in\mathcal{T}_H}\frac{1}{4}\|\tau_\mathbf{K}^{1/2} (b\cdot\nabla e^I_H)\|_{L^2(\mathbf{K})}^2
\\
&\leq& 
C \|b\|_{L^\infty} H^3 |u|_{H^2(\Omega)}^2 +\frac{1}{4} a_{\text{stab}}(e^I_H,e^I_H),
\end{eqnarray*}
where we used %in the third line 
that $\dis\tau_\mathbf{K}(x)|b(x)\cdot\nabla e^I(x)|^2\leq \frac{H|b(x)|}{2}|\nabla e^I(x)|^2$.
\end{itemize}
Collecting the terms, we eventually deduce from~\eqref{err_detail} that
$$
\alpha|e^I_H|_{H^1(\Omega)}^2 +a_{\text{stab}}(e^I_H,e^I_H)\leq \frac{\alpha}{4}|e^I_H|_{H^1(\Omega)}^2+ a_{\text{stab}}(e^I_H,e^I_H) + CH^2\left(\alpha + \|b\|_{L^\infty} H\right)|u|^2_{H^2(\Omega)},
$$
hence $\dis \frac{3}{4} \alpha|e^I_H|_{H^1(\Omega)}^2 \leq CH^2\left(\alpha + \|b\|_{L^\infty} H\right)|u|^2_{H^2(\Omega)}$, thus
$$
|e^I_H|_{H^1(\Omega)} \leq CH\left(1 + \text{Pe} \, H\right)^{1/2} |u|_{H^2(\Omega)}.
$$
Along with~\eqref{eq:bout1}, this estimate shows~\eqref{apriori_supg_singlescale}.

\section{Density of the MsFEM spaces $\VB$ in $H^1_0(\Omega)$}
\label{sec:appB}

We prove here the following result:

\begin{lemma}
\label{lem:densite_msfem}
Let $\eps$ be fixed and assume that $A^\eps \in W^{1,\infty}(\Omega)$. Consider, for any $H$, the spaces $\VB$ defined by~\eqref{def_VB}. Then, for any $g \in H^1_0(\Omega)$ and any $\eta > 0$, there exists $H_0>0$ such that, for any $H<H_0$, there exists some $w_H \in \VB$ such that $\| g - w_H \|_{H^1(\Omega)} \leq \eta$. 
\end{lemma}

Note that we make no structural assumption on how $A^\eps$ depends on $\eps$, and that we do not assume $\eps$ to be small.

\begin{proof}
We fix $\eta$ and $g \in H^1_0(\Omega)$. By density of $H^2(\Omega) \cap H^1_0(\Omega)$ in $H^1_0(\Omega)$, there exists $\widetilde{g} \in H^2(\Omega) \cap H^1_0(\Omega)$ such that
\begin{equation}
\label{eq:tata0}
\| g - \widetilde{g} \|_{H^1(\Omega)} \leq \eta.
\end{equation}
By standard finite element results (see e.g.~\cite[end of Section IX.3]{bernardi-maday-rapetti} or~\cite[Remark 1.129 and Lemma 1.127]{ern2004theory}), for any $H$, there exists $v_H \in V_H$ such that
\begin{equation}
\label{eq:tata0bis}
\| \widetilde{g} - v_H \|_{H^1(\Omega)} \leq C_{\rm FE} H \| \widetilde{g} \|_{H^2(\Omega)}
\end{equation}
where $C_{\rm FE}$ is independent of $H$ and $\widetilde{g}$.

Picking $H_{\rm TR1} = \eta / (C_{\rm FE} \| \widetilde{g} \|_{H^2(\Omega)})$, we therefore deduce from~\eqref{eq:tata0} and~\eqref{eq:tata0bis} that, for any $H<H_{\rm TR1}$, we have $v_H \in V_H$ such that 
\begin{equation}
\label{eq:tata1}
\| g - v_H \|_{H^1(\Omega)} \leq 2 \eta.
\end{equation}

The function $v_H$ belongs to $V_H$, and hence reads $v_H = \sum_i v_i \, \phi_i^0$. We now consider $w_H = \sum_i v_i \, \psi_i^\eps$, which belongs to $\VB$. On each element $\mathbf{K}$, we have
$$
\left\{
\begin{array}{l}
\dis -\text{div } \big( A^\eps \nabla (w_H - v_H) \big) = \text{div } \left( A^\eps \nabla v_H \right) = \sum_{k,\ell=1}^d \partial_k (A^\eps)_{k \ell} \, \partial_\ell v_H \quad \text{in $\mathbf{K}$},
\\
w_H - v_H=0 \quad \text{on $\partial\mathbf{K}$}.
\end{array}
\right.
$$
We hence have that
$$
\int_{\mathbf{K}} \big( \nabla (w_H - v_H) \big)^T A^\eps \nabla (w_H - v_H) 
= 
\int_{\mathbf{K}} \sum_{k,\ell=1}^d (w_H - v_H) \partial_k (A^\eps)_{k \ell} \, \partial_\ell v_H
$$
and therefore, using~\eqref{condition_alpha},
$$
\alpha_1 \| \nabla (w_H - v_H) \|^2_{L^2(\mathbf{K})}
\leq
\| A^\eps \|_{W^{1,\infty}(\Omega)} \ \| w_H - v_H \|_{L^2(\mathbf{K})} \ \| \nabla v_H \|_{L^2(\mathbf{K})}.
$$
Using the Poincaré inequality on $\mathbf{K}$, we obtain that there exists a constant $C$ (which only depends on the shape of the elements, but not on their size) such that
$$
\alpha_1 \| \nabla (w_H - v_H) \|_{L^2(\mathbf{K})}
\leq
C H \| A^\eps \|_{W^{1,\infty}(\Omega)} \ \| \nabla v_H \|_{L^2(\mathbf{K})}.
$$
Summing over the elements, we obtain
$$
\alpha_1 \| \nabla (w_H - v_H) \|_{L^2(\Omega)}
\leq
C H \| A^\eps \|_{W^{1,\infty}(\Omega)} \ \| \nabla v_H \|_{L^2(\Omega)}
$$
which implies, using the Poincaré inequality on $\Omega$, that
$$
\alpha_1 \| w_H - v_H \|_{H^1(\Omega)}
\leq
C H \| A^\eps \|_{W^{1,\infty}(\Omega)} \ \| v_H \|_{H^1(\Omega)}.
$$
Using~\eqref{eq:tata1} in the above bound, we get that, for any $H < H_{\rm TR1}$,
\begin{equation}
\label{eq:tata2}
\| w_H - v_H \|_{H^1(\Omega)} \leq C_{\rm P} H \left[ 2 \eta + \| g \|_{H^1(\Omega)} \right].
\end{equation}
We set $H_{\rm TR2} = \min(1/C_{\rm P},\eta/(C_{\rm P} \| g \|_{H^1(\Omega)}))$ and $H_0 = \min(H_{\rm TR1},H_{\rm TR2})$. Collecting~\eqref{eq:tata1} and~\eqref{eq:tata2}, we deduce that, for any $H < H_0$, $w_H \in \VB$ satisfies
$$
\| g - w_H \|_{H^1(\Omega)} \leq 5\eta.
$$
This concludes the proof. 
\end{proof}

\section{Proof of Lemma~\ref{lem:iter_alternative}}
\label{sec:appC}

We first study the convergence when $n \to \infty$, and next when $H \to 0$.

\medskip

\noindent {\bf Step 1: Convergence when $n \to \infty$.} 
We directly infer from~\eqref{pb_spl_beta_discret_2_discrete_varf} that
\begin{align}
|\widetilde{u}_{2n+1}^{H}|_{H^1(\Omega)} \leq \frac{\beta + \alpha_{\rm spl}}{\beta + \alpha_1}|\widetilde{u}_{2n}^{H}|_{H^1(\Omega)}.
\label{ineq_22}
\end{align}
We now estimate $|\widetilde{u}_{2n+2}^{H}|_{H^1(\Omega)}$ and $|\widetilde{u}_{2n+1}^{H}-P_{\VB}(\widetilde{u}_{2n}^{H})|_{H^1(\Omega)}$. Using the variational formulations~\eqref{pb_spl_beta_discret_1_discrete_varf} for $u_{2n+2}^{H}$ and $u_{2n+4}^{H}$, we deduce a variational formulation for $\widetilde{u}_{2n+2}^{H} = u_{2n+4}^{H} - u_{2n+2}^{H}$. Taking $\widetilde{u}_{2n+2}^{H}$ as test function in that variational formulation, we get
\begin{align}
&(\alpha_{\rm spl}+\beta)|\widetilde{u}_{2n+2}^{H}|_{H^1(\Omega)}^2 
\nonumber 
\\
&\leq \int_\Omega (b\cdot\nabla w_n) \widetilde{u}_{2n+2}^{H} +\sum_{\mathbf{K}\in\mathcal{T}_H} (\tau_\mathbf{K}b\cdot\nabla w_n,b\cdot\nabla \widetilde{u}_{2n+2}^{H})_{L^2(\mathbf{K})} + \int_\Omega \beta \nabla \widetilde{u}_{2n+1}^{H} \cdot\nabla \widetilde{u}_{2n+2}^{H} 
\nonumber
\\
&\leq \| b \|_{L^\infty(\Omega)} |w_n|_{H^1(\Omega)} \| \widetilde{u}_{2n+2}^{H} \|_{L^2(\Omega)} 
\nonumber
\\
& \qquad + \sum_{\mathbf{K}\in\mathcal{T}_H} \frac{H\|b\|_{L^\infty(\Omega)}}{2}\|\nabla w_n\|_{L^2(\mathbf{K})}\|\nabla \widetilde{u}_{2n+2}^{H}\|_{L^2(\mathbf{K})} + \beta | \widetilde{u}_{2n+1}^{H} |_{H^1(\Omega)} | \widetilde{u}_{2n+2}^{H} |_{H^1(\Omega)}
\nonumber
\\
&\leq \| b \|_{L^\infty(\Omega)} |w_n|_{H^1(\Omega)} \| \widetilde{u}_{2n+2}^{H} \|_{L^2(\Omega)} + \frac{H\|b\|_{L^\infty(\Omega)}}{2} |w_n|_{H^1(\Omega)}|\widetilde{u}_{2n+2}^{H}|_{H^1(\Omega)} 
\nonumber
\\
& \qquad + \beta | \widetilde{u}_{2n+1}^{H} |_{H^1(\Omega)} | \widetilde{u}_{2n+2}^{H} |_{H^1(\Omega)}
\nonumber
\\
&\leq\left[\left(C_\Omega+\frac{H}{2}\right)\|b\|_{L^\infty(\Omega)}|w_n|_{H^1(\Omega)}+\beta|\widetilde{u}_{2n+1}^{H}|_{H^1(\Omega)}\right] |\widetilde{u}_{2n+2}^{H}|_{H^1(\Omega)},
\label{ineq_27}
\end{align}
where $w_n = P_{\VB}(\widetilde{u}_{2n}^{H})-\widetilde{u}_{2n+1}^{H}$. 

We now estimate $|w_n|_{H^1(\Omega)}$. We know that, for any $\psi\in \VB$,
\begin{align*}
a_1(w_n,\psi)&=a_1(\widetilde{u}_{2n}^{H}-\widetilde{u}_{2n+1}^{H},\psi)\\
&= a_2(\widetilde{u}_{2n+1}^{H},\psi)- a_1(\widetilde{u}_{2n+1}^{H},\psi)\\
&= \int_\Omega (\nabla \psi)^T (A^\eps-\alpha_{\rm spl}\text{Id})\nabla\widetilde{u}_{2n+1}^{H},
\end{align*}
where we used~\eqref{projection_VB_varf} in the first line and~\eqref{pb_spl_beta_discret_2_discrete_varf} in the second line. Using that $w_n \in \VB$ (this is where using the projection $P_{\VB}$ is needed), we deduce that 
\begin{align}
|w_n|_{H^1(\Omega)}\leq \frac{\|A^\eps - \alpha_{\rm spl}\text{Id}\|_{L^\infty(\Omega)}}{\beta + \alpha_{\rm spl}} |\widetilde{u}_{2n+1}^{H}|_{H^1(\Omega)}.
\label{ineq_28}
\end{align}
Collecting~\eqref{ineq_27}, \eqref{ineq_28} and~\eqref{ineq_22}, we obtain
\begin{align}
|\widetilde{u}_{2n+2}^{H}|_{H^1(\Omega)} \leq \rho|\widetilde{u}_{2n}^{H}|_{H^1(\Omega)},
\label{ineq_29}
\end{align}
where $\dis \rho = \left(C_\Omega+\frac{H}{2}\right)\frac{\|b\|_{L^\infty(\Omega)}}{\beta+\alpha_{\rm spl}}\frac{\| A^\eps-\alpha_{\rm spl}\text{Id}\|_{L^\infty(\Omega)}}{\beta+\alpha_1}+\frac{\beta}{\beta+\alpha_{1}}$. 

As in the proof of Lemma~\ref{lemma_conv_beta}, we introduce the function
$$
g(x) = \left( C_\Omega + \frac{H}{2} \right)\frac{\|b\|_{L^\infty(\Omega)}}{x+\alpha_{\rm spl}}\frac{\| A^\eps-\alpha_{\rm spl}\text{Id}\|_{L^\infty(\Omega)}}{x+\alpha_1}+\frac{x}{x+\alpha_{1}},
$$
observe that $\dis g(x) = 1 - \frac{\alpha_1}{x} + O\left(\frac{1}{x^2}\right)$, which implies, since $\alpha_1>0$, that $\dis \min_{x \geq 0} g(x) < 1$. In view of~\eqref{beta_choice_discrete}, we have $\dis \rho = g(\beta) = \min_{x \geq 0} g(x) < 1$. 

Arguing as in the proof of Lemma~\ref{lemma_conv}, we obtain that $(u_{2n}^{H},u_{2n+1}^{H})$ converges in $H^1_0(\Omega) \times H^1_0(\Omega)$ to some $(u^{H}_{\text{even}},u^{H}_{\text{odd}})\in V_H\times \VB$. Letting $n$ go to $+\infty$ in~\eqref{pb_spl_beta_discret_1_discrete_varf} and~\eqref{pb_spl_beta_discret_2_discrete_varf}, we obtain that $u^{H}_{\text{even}}$ and $u^{H}_{\text{odd}}$ satisfy the variational formulations~\eqref{pb_spl_beta_discret_lim_1_discrete_varf} and~\eqref{pb_spl_beta_discret_lim_2_discrete_varf}.

\medskip

\noindent {\bf Step 2: Convergence when $H \to 0$.} 
We recast~\eqref{pb_spl_beta_discret_lim_1_discrete_varf}--\eqref{pb_spl_beta_discret_lim_2_discrete_varf} as the following variational formulation:

\begin{equation}
\label{eq:coer9}
\begin{array}{c} 
\text{Find $(u_{\text{even}}^H,u_{\text{odd}}^H) \in V_H \times \VB$ such that, for any $(v,w) \in V_H\times\VB$},
\\
c_H \Big( (u_{\text{even}}^H,u_{\text{odd}}^H),(v,w) \Big)=B_H(v,w),
\end{array}
\end{equation}
where the bilinear form
\begin{eqnarray*}
c_H \Big( (u_{\text{even}}^H,u_{\text{odd}}^H),(v,w) \Big) 
&=& 
a_1(u^H_{\text{even}},v) + a_{\text{conv}}(u^H_{\text{even}},v)
\\
&& - \int_\Omega \beta \nabla u^H_{\text{odd}} \cdot \nabla v + a_{\text{conv}}(u^H_{\text{odd}},v)
\\
&& - a_{\text{conv}}(P_{\VB}(u^H_{\text{even}}),v) + a_2(u^H_{\text{odd}},w) - a_1(u^H_{\text{even}},w)
\end{eqnarray*}
is defined on $\Big( H^1_0(\Omega) \times H^1_0(\Omega) \Big)^2$. Recall that $a_1$, $a_{\text{conv}}$ and $a_2$ are defined by~\eqref{eq:def_a1}, \eqref{eq:def_aconv} and~\eqref{eq:def_a2}, respectively, while the operator $P_{\VB}$ is defined by~\eqref{projection_VB_varf}. The linear form
$$
B_H(v,w) = \int_\Omega fv + \sum_{\mathbf{K} \in \mathcal{T}_H}(\tau_{\mathbf{K}}f,b\cdot\nabla v)_{L^2(\mathbf{K})}
$$
is defined on $H^1_0(\Omega) \times H^1_0(\Omega)$. Note that $B_H(v,w)$ does not depend on $w$. 

\medskip

The convergence proof when $H \to 0$ is based on the following arguments. First, we are going to show that, if $H$ is sufficiently small, $c_H$ satisfies an inf-sup condition uniformly in the mesh size $H$. For that purpose, we adapt the arguments of~\cite[Theorem 4.2.9]{sauter} to our setting. We introduce the bilinear form
$$
\widetilde{c}_H \Big( (u,v),(\phi,\psi) \Big) = c_H \Big( (u,v),(\phi,\psi) \Big) + \lambda\int_\Omega u \, \phi,
$$
defined on $\Big( H^1_0(\Omega) \times H^1_0(\Omega) \Big)^2$, where $\lambda>0$ is a parameter, and show in Step~2a below that $\widetilde{c}_H$ is coercive in the $H^1(\Omega) \times H^1(\Omega)$ norm, provided $\lambda$ is large enough and $H$ is sufficiently small. This allows us to next show, as claimed above, that $c_H$ satisfies the inf-sup condition (see Step~2b), uniformly in $H$ (as soon as $H$ is sufficiently small). In contrast to the setting of~\cite[Theorem 4.2.9]{sauter}, the bilinear forms $c_H$ and $\widetilde{c}_H$ here depend on $H$.  

We are then in position to use classical numerical analysis arguments (see Step~2c) for estimating the discretization error (see~\eqref{eq:coer10} below). This error is bounded from above (up to some multiplicative constants) by the best approximation error and by the error introduced by the fact that $c_H$ and $B_H$ in~\eqref{eq:coer9} depend on $H$. The end of the proof amounts to showing that these two errors converge to 0 when $H \to 0$.

\medskip

\noindent {\bf Step 2a: Coercivity of $\widetilde{c}_H$.} 
We assume that
\begin{equation}
\label{eq:assum_sauter}
\lambda \geq \frac{4 \|b\|_{L^\infty(\Omega)}^2}{\alpha_{\rm spl}} + \frac{\|b\|^2_{L^\infty(\Omega)}}{2\left(\alpha_1 - \alpha_{\rm spl}/2 \right)} 
\quad \text{and} \quad 
H \| b\|_{L^\infty(\Omega)} < \min \left( 2 \alpha_1 - \alpha_{\rm spl}, \frac{\alpha_{\rm spl}}{5}\right)
\end{equation} 
where we recall that $\alpha_1$ is such that~\eqref{condition_alpha} holds and that $\alpha_{\rm spl}$ is such that $2 \alpha_1 > \alpha_{\rm spl}$ (see~\eqref{eq:hyp_spl}). We claim that 
\begin{equation}
\label{eq:claim_c_tilde}
\text{Under assumption~\eqref{eq:assum_sauter}, $\widetilde{c}_H$ is coercive}.
\end{equation} 
Note that~\eqref{eq:assum_sauter} does not impose any restriction on $\beta$. The assumption~\eqref{beta_choice_discrete} is only used in Step 1 above (to show the convergence when $n \to \infty$).

\medskip

For any $(u,v) \in H^1_0(\Omega) \times H^1_0(\Omega)$, we have
\begin{eqnarray}
\widetilde{c}_H \Big( (u,v),(u,v) \Big)
& = &
a_1(u,u) + a_2(v,v) + \lambda \| u \|_{L^2(\Omega)}^2
\nonumber\\
&&- \int_\Omega \beta \nabla v \cdot \nabla u - a_1(u,v) 
\nonumber\\
&& + a_{\text{conv}}(u-P_{\VB}(u),u) + a_{\text{conv}}(v,u). 
\label{term2}
\end{eqnarray}
Using the coercivity of the bilinear forms $a_1$ and $a_2$, we get
\begin{equation}
\label{eq:coer1}
a_1(u,u) + a_2(v,v) \geq (\beta + \alpha_{\rm spl}) |u|^2_{H^1(\Omega)} + (\beta + \alpha_1) |v|^2_{H^1(\Omega)}.
\end{equation}
We now bound the terms in the last line of~\eqref{term2}. Using the fact that $\dis \tau_{\mathbf{K}}(x)= \frac{H}{2|b(x)|}$ and $|P_{\VB}(u)|_{H^1(\Omega)} \leq |u|_{H^1(\Omega)}$, we have that
\begin{eqnarray}
& & |a_{\text{conv}}(u-P_{\VB}(u),u) + a_{\text{conv}}(v,u)|
\nonumber
\\
& \leq &
\| b \|_{L^\infty(\Omega)} \, |u-P_{\VB}(u)|_{H^1(\Omega)} \, \|u\|_{L^2(\Omega)} + \frac{H}{2} \, \| b\|_{L^\infty(\Omega)} \, |u-P_{\VB}(u)|_{H^1(\Omega)} \, |u|_{H^1(\Omega)} 
\nonumber
\\
& & 
+ \| b \|_{L^\infty(\Omega)} \, |v|_{H^1(\Omega)} \, \| u \|_{L^2(\Omega)} + \frac{H}{2} \, \| b\|_{L^\infty(\Omega)} \, |v|_{H^1(\Omega)} \, |u|_{H^1(\Omega)} 
\nonumber
\\
& \leq &
2\| b \|_{L^\infty(\Omega)} \, |u|_{H^1(\Omega)} \, \|u\|_{L^2(\Omega)} + \frac{5H}{4} \, \| b\|_{L^\infty(\Omega)} \, |u|^2_{H^1(\Omega)}
\nonumber
\\
&&
+ \| b \|_{L^\infty(\Omega)} \, |v|_{H^1(\Omega)} \, \| u \|_{L^2(\Omega)} + \frac{H}{4} \, \| b\|_{L^\infty(\Omega)} \, |v|^2_{H^1(\Omega)}
\nonumber
\\
& \leq &
\left( \frac{\alpha_{\rm spl}}{4} + \frac{5H}{4} \, \| b\|_{L^\infty(\Omega)} \right) |u|_{H^1(\Omega)}^2 + \left( \frac{4 \|b\|_{L^\infty(\Omega)}^2}{\alpha_{\rm spl}} + \frac{\|b\|^2_{L^\infty(\Omega)}}{2\left(\alpha_1- \alpha_{\rm spl}/2 \right)} \right) \|u\|_{L^2(\Omega)}^2
\nonumber
\\
&&
+ \left( \frac{1}{2} \left(\alpha_1 - \frac{\alpha_{\text{spl}}}{2} \right) + \frac{H}{4} \, \| b\|_{L^\infty(\Omega)} \right)|v|_{H^1(\Omega)}^2,
\label{eq:coer3}
\end{eqnarray}
where we have used a Young inequality in the last line. We bound the terms in the second line of~\eqref{term2} by
\begin{equation}
\label{eq:coer2}
\left| - \int_\Omega \beta \nabla v \cdot \nabla u - a_1(u,v) \right| 
\leq
\frac{\beta}{2} (|u|_{H^1(\Omega)}^2+|v|_{H^1(\Omega)}^2) + \frac{\beta+\alpha_{\rm spl}}{2} (|u|_{H^1(\Omega)}^2+|v|_{H^1(\Omega)}^2).
\end{equation} 
Collecting~\eqref{term2}, \eqref{eq:coer1}, \eqref{eq:coer3} and~\eqref{eq:coer2}, we get
\begin{eqnarray*}
\widetilde{c}_H \Big( (u,v),(u,v) \Big)
&\geq& 
\left( \frac{\alpha_{\rm spl}}{4} - \frac{5H}{4} \, \| b\|_{L^\infty(\Omega)} \right) |u|_{H^1(\Omega)}^2 
\\
&&
+ \left( \frac{1}{2} \left(\alpha_1 - \frac{\alpha_{\rm spl}}{2} \right) - \frac{H}{4} \, \| b\|_{L^\infty(\Omega)} \right) |v|_{H^1(\Omega)}^2 
\\
&&
+ \left( \lambda - \frac{4\|b\|_{L^\infty(\Omega)}^2}{\alpha_{\rm spl}} - \frac{\|b\|^2_{L^\infty(\Omega)}}{2\left(\alpha_1 - \alpha_{\rm spl}/2 \right)} \right) \|u\|^2_{L^2(\Omega)}.
\end{eqnarray*}
Under assumption~\eqref{eq:assum_sauter}, using a Poincaré inequality, we see that there exists $\eta>0$ such that 
\begin{equation}
\label{eq:coercivite}
\forall (u,v) \in H^1_0(\Omega) \times H^1_0(\Omega), \quad
\widetilde{c}_H \Big( (u,v),(u,v) \Big) \geq \eta \left( \|u\|_{H^1(\Omega)}^2 + \|v\|_{H^1(\Omega)}^2 \right).
\end{equation} 
This concludes the proof of the claim~\eqref{eq:claim_c_tilde}. 

\medskip

\noindent {\bf Step 2b: Inf-sup condition on $c_H$.} 
We want to show that there exists $H_0>0$ and $\alpha>0$ such that, for any $H\leq H_0$,
\begin{equation}
\label{eq:coer4}
\underset{U^H \in V_H \times \VB}{\inf} \quad \underset{\Phi^H \in V_H \times \VB}{\sup} \quad \frac{c_H \left(U^H,\Phi^H \right)}{\|U^H\|_{H^1(\Omega) \times H^1(\Omega)} \, \|\Phi^H\|_{H^1(\Omega) \times H^1(\Omega)}} \geq \alpha.
\end{equation}
We prove this statement by contradiction and therefore assume that~\eqref{eq:coer4} does not hold. Then, there exists a sequence $H_n$ that converges to 0 and a sequence $U^{H_n} = \left( u^{H_n}_{\rm even},u^{H_n}_{\rm odd} \right) \in V_{H_n} \times \VBn$ with $\| U^{H_n} \|_{H^1(\Omega) \times H^1(\Omega)} = 1$, such that
\begin{equation}
\label{eq:coer5}
\lim_{n \rightarrow +\infty} \ \ \underset{\Phi \in V_{H_n} \times \VBn}{\sup} \ \ \frac{c_{H_n}\left(U^{H_n},\Phi \right)}{\|\Phi\|_{H^1(\Omega) \times H^1(\Omega)}}=0.
\end{equation} 
As the sequence $U^{H_n}$ is bounded in $H^1(\Omega) \times H^1(\Omega)$, it weakly converges in $H^1_0(\Omega) \times H^1_0(\Omega)$ to some $U^\star=(u^\star_{\rm even},u^\star_{\rm odd}) \in H^1_0(\Omega) \times H^1_0(\Omega)$, up to the extraction of a subsequence that we still denote $U^{H_n}$. 

Using~\eqref{projection_VB_varf}, we also deduce from the boundedness of $u^{H_n}_{\rm even}$ that $P_{\VBn}(u^{H_n}_{\text{even}})$ is bounded in $H^1$ norm. Up to an additional extraction, we hence have that $P_{\VBn}(u^{H_n}_{\text{even}})$ weakly converges in $H^1(\Omega)$ to some $u^{\Pi}_{\text{even}}$. We claim that $u^{\Pi}_{\text{even}}=u^\star_{\text{even}}$. Let indeed $\phi \in H^1_0(\Omega)$. By density (see Appendix~\ref{sec:appB}), there exists a sequence $w_n \in \VBn$ converging strongly in $H^1_0(\Omega)$ to $\phi$. For any $n$, we have $\dis a_1(P_{\VBn}(u^{H_n}_{\text{even}}),w_n) = a_1(u^{H_n}_{\text{even}},w_n)$. Passing to the limit $n \to \infty$, we infer that $\dis a_1(u^{\Pi}_{\text{even}},\phi) = a_1(u^\star_{\text{even}},\phi)$, which holds true for any $\phi \in H^1_0(\Omega)$. This implies that $u^{\Pi}_{\text{even}}=u^\star_{\text{even}}$. Consequently, $P_{\VBn}(u^{H_n}_{\text{even}}) - u^{H_n}_{\text{even}}$ weakly converges in $H^1_0(\Omega)$ to 0.

\medskip

We first show that $U^\star=0$. We fix some $\Phi=(\phi,\psi) \in H^1_0(\Omega) \times H^1_0(\Omega)$. For any $\Phi^{H_n}=(\phi^{H_n},\psi^{H_n}) \in V_{H_n} \times \VBn$, we write
\begin{equation}
\label{eq:coer6}
c_{H_n}(U^{H_n},\Phi) = c_{H_n}(U^{H_n},\Phi^{H_n}) + c_{H_n}(U^{H_n},\Phi-\Phi^{H_n}).
\end{equation}
We have that
\begin{equation}
\label{ineq_infsup_7}
|c_{H_n}(U^{H_n},\Phi^{H_n})|
\leq
\left( \underset{\Psi \in V_{H_n} \times \VBn}{\sup} \ \frac{c_{H_n}(U^{H_n},\Psi)}{\|\Psi\|_{H^1(\Omega) \times H^1(\Omega)}}\right) \|\Phi^{H_n}\|_{H^1(\Omega) \times H^1(\Omega)}
\end{equation}
and, since $c_H$ is a continuous bilinear form,
\begin{equation}
\label{ineq_infsup_8}
|c_{H_n}(U^{H_n},\Phi-\Phi^{H_n})|
\leq
M \|U^{H_n}\|_{H^1(\Omega) \times H^1(\Omega)} \, \|\Phi -\Phi^{H_n}\|_{H^1(\Omega) \times H^1(\Omega)}.
\end{equation}
By an argument of density (see Appendix~\ref{sec:appB}), there exists a sequence $\psi^{H_n} \in \VBn$ converging strongly in $H^1_0(\Omega)$ to $\psi$. Likewise, by an argument of density and classical results on finite element methods, we know that there also exists a sequence $\phi^{H_n} \in V_{H_n}$ converging strongly in $H^1_0(\Omega)$ to $\phi$. We thus have built a sequence $\Phi^{H_n}=\left(\phi^{H_n},\psi^{H_n}\right) \in V_{H_n} \times \VBn$ such that $\dis \lim_{n \rightarrow +\infty} \| \Phi -\Phi^{H_n} \|_{H^1(\Omega) \times H^1(\Omega)} = 0$. We infer from~\eqref{eq:coer6}, \eqref{ineq_infsup_8}, \eqref{ineq_infsup_7} and~\eqref{eq:coer5} that
$$
\lim_{n \rightarrow +\infty} c_{H_n}(U^{H_n},\Phi) = 0.
$$
Making use of the explicit expression of $c_H$ and using that $U^{H_n}$ weakly converges in $H^1_0(\Omega) \times H^1_0(\Omega)$ to $U^\star=(u^\star_{\rm even},u^\star_{\rm odd})$ and that $P_{\VBn}(u^{H_n}_{\text{even}}) - u^{H_n}_{\text{even}}$ weakly converges in $H^1_0(\Omega)$ to 0, we obtain that
$$
%\begin{equation}
%\label{eq:coer7}
a_1(u^\star_{\rm even},\phi) - \int_\Omega \beta \nabla u^{\star}_{\rm odd} \cdot \nabla \phi + \int_\Omega (b \cdot \nabla u^\star_{\rm odd}) \phi + a_2(u^\star_{\rm odd},\psi) - a_1(u^\star_{\rm even},\psi) = 0.
%\end{equation}
$$
This holds for any $(\phi,\psi) \in H^1_0(\Omega) \times H^1_0(\Omega)$. Taking $\phi = \psi$, we deduce that $u^\star_{\rm odd} = 0$. This next implies that $u^\star_{\rm even} = 0$, and hence that $U^\star=0$.

\medskip

Second, we show the {\em strong} convergence in $H^1_0(\Omega) \times H^1_0(\Omega)$ of the sequence $U^{H_n}$ to $U^\star=0$. Under assumption~\eqref{eq:assum_sauter}, we have shown in Step 2a above that $\widetilde{c}_H$ is coercive. In view of~\eqref{eq:coercivite}, we thus have
\begin{eqnarray*}
\eta \| U^{H_n} \|^2_{H^1(\Omega) \times H^1(\Omega)} 
&\leq& 
\widetilde{c}_{H_n}(U^{H_n},U^{H_n})
\\
&=& 
c_{H_n}(U^{H_n},U^{H_n}) + \lambda \int_\Omega \left( u^{H_n}_{\rm even} \right)^2 
\\
&\leq& 
\left( \underset{\Phi \in V_{H_n} \times \VBn}{\sup} \ \frac{c_{H_n}(U^{H_n},\Phi)}{\| \Phi \|_{H^1(\Omega) \times H^1(\Omega)}} \right) + \lambda \, \| u^{H_n}_{\rm even} \|^2_{L^2(\Omega)}.
\end{eqnarray*}
In view of~\eqref{eq:coer5}, the first term in the above right-hand side converges to 0 when $n \to \infty$. Up to the extraction of a subsequence, $u^{H_n}_{\rm even}$ (which weakly converges to 0 in $H^1(\Omega)$) strongly converges to 0 in $L^2(\Omega)$. This implies that the second term in the above right-hand side also converges to 0 when $n \to \infty$. 

We then deduce that $\dis \lim_{n \to \infty} \| U^{H_n} \|^2_{H^1(\Omega) \times H^1(\Omega)} = 0$, which is a contradiction with the fact that, by construction, $\|U^{H_n}\|_{H^1(\Omega) \times H^1(\Omega)}=1$. This concludes the proof of~\eqref{eq:coer4}.

\medskip

\noindent {\bf Step 2c: Conclusion.} 
We are now in position to use~\cite[Lemma 2.27]{ern2004theory}, which states an upper bound on the error (see~\eqref{eq:coer10} below) under three assumptions. Assumption (i) of that lemma is that the approximation spaces are conformal. This is obviously satisfied here, as $V_H \times \VB \subset H^1_0(\Omega) \times H^1_0(\Omega)$. Assumption (ii) is that $c_H$ satisfies an inf-sup condition. It is satisfied here in view of~\eqref{eq:coer4}. Assumption (iii) is that the bilinear form $c_H$ is bounded. This is again satisfied here. The assumptions of~\cite[Lemma 2.27]{ern2004theory} being satisfied, we can write an error bound (see~\eqref{eq:coer10} below) between the solution to~\eqref{eq:coer9} and the solution to the corresponding infinite dimensional problem, that reads
\begin{equation}
\label{eq:coer8}
\begin{array}{c}
\text{Find $(u_{\text{even}},u_{\text{odd}}) \in H^1_0(\Omega) \times H^1_0(\Omega)$ such that}, 
\\
\text{for any $(v,w) \in H^1_0(\Omega) \times H^1_0(\Omega)$, \ \ $c \Big( (u_{\text{even}},u_{\text{odd}}),(v,w) \Big)=B(v,w)$},
\end{array}
\end{equation}
where 
\begin{multline*}
c \Big( (u_{\text{even}},u_{\text{odd}}),(v,w) \Big) 
=
a_1(u_{\text{even}},v) 
- \int_\Omega \beta \nabla u_{\text{odd}} \cdot \nabla v + \int_\Omega (b \cdot \nabla u_{\text{odd}}) v
\\
+ a_2(u_{\text{odd}},w) - a_1(u_{\text{even}},w)
\end{multline*}
and 
$$
B(v,w) = \int_\Omega fv.
$$
It is obvious that $(u_{\text{even}},u_{\text{odd}})$ is a solution to~\eqref{eq:coer8} if and only if $(u_{\text{even}},u_{\text{odd}})$ is a solution to the system
\begin{align}
&\left\{\begin{aligned}
&-(\beta + \alpha_{\rm spl})\Delta u_{\rm even} = f - b\cdot\nabla u_{\rm odd}-\beta \Delta u_{\rm odd}\quad\text{ in }\Omega, \\
&u_{\rm even}=0 \quad\text{ on }\partial\Omega ,
\end{aligned}\right.
\label{pb_spl_beta_1_app}
\\
&\left\{\begin{aligned}
&-\text{div }((\beta \text{Id}+A^\eps)\nabla u_{\rm odd})=-(\beta+\alpha_{\rm spl})\Delta u_{\rm even}\quad\text{ in }\Omega,\\
&u_{\rm odd}=0\quad\text{ on }\partial\Omega.
\end{aligned}\right.
\nonumber
%\label{pb_spl_beta_2_app}
\end{align}
This system is well-posed: by adding the two equations, we obtain that $u_{\text{odd}}$ is a solution to~\eqref{pb_multiscale}, and is therefore unique. This implies the uniqueness of $u_{\text{even}}$ in view of~\eqref{pb_spl_beta_1_app}. We denote by $U = (u_{\text{even}},u_{\text{odd}})$ the unique solution to~\eqref{eq:coer8}.

\medskip

Using~\cite[Lemma 2.27]{ern2004theory}, we obtain that
\begin{multline}
\label{eq:coer10}
\| U - U^H \|_{H^1(\Omega) \times H^1(\Omega)} 
\leq 
\frac{1}{\alpha} \ \sup_{\Phi \in V_H \times \VB} \frac{\left| B(\Phi) - B_H(\Phi) \right|}{\| \Phi \|_{H^1(\Omega) \times H^1(\Omega)}}
\\
+
\inf_{G \in V_H \times \VB} \left[ \left( 1 + \frac{M}{\alpha} \right) \| U - G \|_{H^1(\Omega) \times H^1(\Omega)} + \frac{1}{\alpha} \ \sup_{\Phi \in V_H \times \VB} \frac{\left| c(G,\Phi) - c_H(G,\Phi) \right|}{\| \Phi \|_{H^1(\Omega) \times H^1(\Omega)}} \right],
\end{multline}
where $M$ is the continuity constant of the bilinear form $c$. We successively study the two terms in the right-hand side of~\eqref{eq:coer10}. 

For the first term, we write, for any $\Phi=(\phi,\psi) \in V_H \times \VB$, that
$$
\left| B(\Phi) - B_H(\Phi) \right| 
\leq
\frac{H}{2} \sum_{\mathbf{K} \in \mathcal{T}_H} \| f \|_{L^2(\mathbf{K})} \| \nabla \phi \|_{L^2(\mathbf{K})}
\leq
\frac{H}{2} \| f \|_{L^2(\Omega)} \| \Phi \|_{H^1(\Omega) \times H^1(\Omega)},
$$
which implies that
\begin{equation}
\label{eq:coer11}
\lim_{H \to 0} \ \sup_{\Phi \in V_H \times \VB} \frac{\left| B(\Phi) - B_H(\Phi) \right|}{\| \Phi \|_{H^1(\Omega) \times H^1(\Omega)}} = 0.
\end{equation}
For the second term of the right-hand side of~\eqref{eq:coer10}, we write, for any $\Phi=(\phi,\psi) \in V_H \times \VB$ and any $G=(g,h) \in V_H \times \VB$, that
$$
c_H(G,\Phi) - c(G,\Phi) 
= 
a_{\text{conv}}(g-P_{\VB}(g),\phi) + \sum_{\mathbf{K}\in\mathcal{T}_H} (\tau_\mathbf{K} b\cdot \nabla h,b\cdot\nabla \phi)_{L^2(\mathbf{K})}.
$$
We therefore deduce, using an integration by parts in the first line, that
\begin{eqnarray*}
&&\left| c_H(G,\Phi) - c(G,\Phi) \right|
\\
& \leq &
\left| \int_\Omega \left[ g-P_{\VB}(g) \right] b \cdot \nabla \phi \right|
\\
&&
+ \frac{H \| b \|_{L^\infty(\Omega)}}{2} \, | g-P_{\VB}(g) |_{H^1(\Omega)} | \phi |_{H^1(\Omega)}
+ \frac{H \| b \|_{L^\infty(\Omega)}}{2} \, | h |_{H^1(\Omega)} | \phi |_{H^1(\Omega)}
\\
& \leq &
\| b \|_{L^\infty(\Omega)} \| g - P_{\VB}(g) \|_{L^2(\Omega)} \| \Phi \|_{H^1(\Omega) \times H^1(\Omega)}
\\
&&
+ H \| b \|_{L^\infty(\Omega)} \, \left( | g |_{H^1(\Omega)} + | h |_{H^1(\Omega)} \right) \| \Phi \|_{H^1(\Omega) \times H^1(\Omega)}.
\end{eqnarray*}
We hence write, for the second term of the right-hand side of~\eqref{eq:coer10}, that
\begin{multline*}
\left( 1 + \frac{M}{\alpha} \right) \| U - G \|_{H^1(\Omega) \times H^1(\Omega)} + \frac{1}{\alpha} \ \sup_{\Phi \in V_H \times \VB} \frac{\left| c(G,\Phi) - c_H(G,\Phi) \right|}{\| \Phi \|_{H^1(\Omega) \times H^1(\Omega)}} 
\\
\leq
\mathcal{C} \| U - G \|_{H^1(\Omega) \times H^1(\Omega)}
+
\mathcal{C} \| g - P_{\VB}(g) \|_{L^2(\Omega)} 
+
\mathcal{C} H \| G \|_{H^1(\Omega) \times H^1(\Omega)}
\end{multline*}
where $\mathcal{C}$ is independent of $H$. Using the density of the families $V_H$ and $\VB$ in $H^1_0(\Omega)$ (see Appendix~\ref{sec:appB} for the latter property), we build $G^H=(g^H,h^H) \in V_H \times \VB$ such that $\dis \lim_{H \to 0} \| U - G^H \|_{H^1(\Omega) \times H^1(\Omega)} = 0$. We thus have that
\begin{multline*}
\inf_{G \in V_H \times \VB} \left[ \left( 1 + \frac{M}{\alpha} \right) \| U - G \|_{H^1(\Omega) \times H^1(\Omega)} + \frac{1}{\alpha} \ \sup_{\Phi \in V_H \times \VB} \frac{\left| c(G,\Phi) - c_H(G,\Phi) \right|}{\| \Phi \|_{H^1(\Omega) \times H^1(\Omega)}} \right]
\\
\leq
\mathcal{C} \| U - G^H \|_{H^1(\Omega) \times H^1(\Omega)}
+
\mathcal{C} \| g^H - P_{\VB}(g^H) \|_{L^2(\Omega)} 
+
\mathcal{C} H \| G^H \|_{H^1(\Omega) \times H^1(\Omega)}.
\end{multline*}
The above three terms converge to 0 when $H \to 0$ (for the second term, this is a consequence of the fact that, for any bounded sequence $\tau_H \in H^1_0(\Omega)$, we have that $\tau^H - P_{\VB}(\tau^H)$ weakly converges to 0 in $H^1(\Omega)$). Collecting this result with~\eqref{eq:coer10} and~\eqref{eq:coer11}, we deduce that $\dis \lim_{H \to 0} \| U - U^H \|_{H^1(\Omega) \times H^1(\Omega)} = 0$. This concludes the proof of Lemma~\ref{lem:iter_alternative}.

%\nocite{*}
% permet de citer tout ce qu'il y a dans le .bib, que ce soit ou non appele.

\bibliographystyle{plain}
\bibliography{biblioarticle}
\end{document}